\RequirePackage{fix-cm}
\documentclass{svjour3}
\usepackage{silence}
\usepackage{bbm}
\usepackage{amsfonts}
\usepackage{amsmath}
\usepackage{amssymb}
\usepackage{mathrsfs}
\usepackage{enumitem}
\usepackage{multirow}
\usepackage{booktabs}
\usepackage{makecell}
\usepackage{graphicx}
\usepackage{subcaption}
\usepackage{tcolorbox}
\usepackage{anyfontsize}
\usepackage{algorithm}
\usepackage{algpseudocode}
\usepackage{hyperref}[6.83]

\colorlet{color1}{blue}
\colorlet{color2}{red!50!black}
\hypersetup{
colorlinks=true,
frenchlinks=false,
pdfborder={0 0 0},
naturalnames=false,
hypertexnames=false,
breaklinks,
allcolors = color1,
urlcolor = color2,
}

\newtheorem{assumption}{Assumption}

\newcommand{\upDelta}{\mathrm{\Delta}}

\newcommand{\upLambda}{\mathrm{\Lambda}}

\newcommand{\upXi}{\mathrm{\Xi}}

\newcommand{\upPi}{\mathrm{\Pi}}

\newcommand{\upSigma}{\mathrm{\Sigma}}

\newcommand{\upUpsilon}{\mathrm{\Upsilon}}

\newcommand{\upOmega}{\mathrm{\Omega}}

\title{ODE-based Learning to Optimize}

\author{Zhonglin Xie \and Wotao Yin \and Zaiwen Wen}

\institute{
Zhonglin Xie
\at
Beijing International Center for Mathematical
Research, Peking University\\
E-mail: zlxie@pku.edu.cn
\and
Wotao Yin
\at
Alibaba US, DAMO Academy\\
E-mail: wotao.yin@alibaba-inc.com
\and
Zaiwen Wen
\at
Beijing International Center for Mathematical
Research, Peking University. Research supported in part by the NSFC grant 12331010. \\
E-mail: wenzw@pku.edu.cn
}

\begin{document}

\date{Received: date / Accepted: date}
\maketitle

\begin{abstract}
    Recent years have seen a growing interest in understanding acceleration methods through the lens of ordinary differential equations (ODEs). Despite the theoretical advancements, translating the rapid convergence observed in continuous-time models to discrete-time iterative methods poses significant challenges. In this paper, we present a comprehensive framework integrating the inertial systems with Hessian-driven damping equation (ISHD) and learning-based approaches for developing optimization methods through a deep synergy of theoretical insights. We first establish the convergence condition for ensuring the convergence of the solution trajectory of ISHD. Then, we show that provided the stability condition, another relaxed requirement on the coefficients of ISHD, the sequence generated through the explicit Euler discretization of ISHD converges, which gives a large family of practical optimization methods. In order to select the best optimization method in this family for certain problems, we introduce the stopping time, the time required for an optimization method derived from ISHD to achieve a predefined level of suboptimality. Then, we formulate a novel learning to optimize (L2O) problem aimed at minimizing the stopping time subject to the convergence and stability condition. To navigate this learning problem, we present an algorithm combining stochastic optimization and the penalty method (StoPM). The convergence of StoPM using the conservative gradient is proved. Empirical validation of our framework is conducted through extensive numerical experiments across a diverse set of optimization problems. These experiments showcase the superior performance of the learned optimization methods.
\end{abstract}

\section{Introduction}
\label{sec:intro}

In recent years, there has been substantial work aimed at understanding the nature of acceleration methods using ordinary differential equations (ODEs). The seminal work of Su, Boyd, and Candes \cite{suDifferentialEquationModeling} proposes a second-order differential equation as the continuous-time counterpart of the Nesterov accelerated gradient method (NAG) \cite{nesterov1983method}. This work provides theoretical insights into the nature of NAG. By combining the continuous-time model of NAG with Hessian-driven damping, a differential equation called the inertial system with Hessian-driven damping (ISHD) has been investigated in \cite{attouchFirstorderOptimizationAlgorithms2020,attouch2016fast}:
\begin{equation}\label{eq:ISHD}
\ddot{x}(t)+\frac{\alpha}{t}\dot{x}(t)+\beta(t)\nabla^2f(x(t))\dot{x}(t)+\gamma(t)\nabla f(x(t))=0,
\end{equation}
where \( x(t) \) belongs to \( \mathbb{R}^{n} \) and \( f \) is twice differentiable and convex. Here, \( t_0, \alpha > 0 \), and \( \beta, \gamma \) are non-negative continuous functions defined on \( [t_0,+\infty) \). By choosing different values for \( \alpha, \beta \), and \( \gamma \), the equation \eqref{eq:ISHD} can coincide with many other optimization-inspired ODEs, providing a unified framework for understanding various optimization methods.

Although the ODE viewpoint has been successful in understanding optimization methods, significant practical gaps exist between iterative optimization methods and their continuous-time counterparts. The fast convergence rate of $x(t)$ in continuous-time cases may not be translated directly to the sequence $\{x_{k}\}_{k=0}^{\infty}$ obtained via discretizing the corresponding ODE. Designing iterative optimization methods using the ODE viewpoint often requires a case-by-case discussion and equation-specific discretization schemes.

For example, in \cite{wibisonoVariationalPerspectiveAccelerated2016}, it focuses on a subfamily of \eqref{eq:ISHD} with $\alpha=p+1,\beta(t)\equiv 0$, and $\gamma(t)=Cp^2t^{p-2}$, where $p,C>0$. According to \cite{wibisonoVariationalPerspectiveAccelerated2016}, the convergence rate of the trajectory $x(t)$ is given by $f(x(t)) - f_{\star} \leq \mathcal{O}(1/t^{p})$. However, the arbitrary nature of the parameter $p$ suggests that the convergence rate of the trajectory cannot be bounded by \emph{any inverse polynomial function} for any convex differentiable functions. This finding starkly contrasts with the well-known lower bound of $\mathcal{O}(1/k^2)$ for first-order methods when the function $f$ is $L$-smooth convex, as discussed in \cite{nesterovLecturesConvexOptimization2018}. The underlying issue is attributed to the notation $x(t)$, which glosses over the challenges posed by differential equations with varying curvatures. In numerical computations, achieving a solution of comparable accuracy significantly depends on the curvature of the objective functions. It is misleading to directly compare the convergence rates of continuous-time models with their discrete-time counterparts. The empirical results show that a naive explicit Euler discretization scheme diverges even for a quadratic objective function. A stable discretization that recovers the $\mathcal{O}(1/k^p)$ rate requires an implicit update using extra higher-order gradient. Hence, a core problem in this field remains:
\begin{center}
    (\textbf{P1})
    \textit{``Can we translate the fast convergence properties of the trajectory of ODEs into the discretized sequences using explicit discretization schemes?''}
\end{center}
%

We provide a positive answer to this problem by giving conditions for discretizing the ODE \eqref{eq:ISHD} stably using the explicit Euler scheme. Our emphasis on explicit schemes, such as the explicit Euler discretization, stems from the desire to develop a general framework for translating continuous-time ODE properties into discrete-time algorithms. By doing so, we can avoid the need for ad hoc, equation-specific designs and case-by-case analyses, which are often required for implicit schemes. This exploration of explicit Euler discretization serves as a foundational step towards adapting more sophisticated explicit schemes, such as Runge-Kutta methods, to preserve the properties of continuous trajectories in discrete iterative algorithms.


Given the flexibility of the coefficients \(\alpha, \beta(t)\), and \(\gamma(t)\) in the ODE \eqref{eq:ISHD}, the choice of these coefficients can significantly impact the convergence rate and stability of the ODE solution, which in turn affects the performance of the discretized iterative algorithm. While the ODE \eqref{eq:ISHD} provides a unifying framework, finding the optimal coefficients for a given problem is a non-trivial task. This leads to a crucial question:
\begin{center}
    (\textbf{P2})
\textit{``What are the optimal coefficients of the ODE \eqref{eq:ISHD} for specific problems?''}
\end{center}

This issue, referred to as the ``best tuning of the coefficients'', has also been highlighted by \cite{attouch2022fast}. We demonstrate that the optimal coefficients can be determined numerically using learning to optimize.

\subsection{Related works}
\subsubsection{ODE viewpoint of optimizetion methods}
\label{sec:related-ODE}

The ODE viewpoint of optimization methods have a long history. The work by Su, Boyd, and Candes \cite{suDifferentialEquationModeling} has revived interest in using ODEs to understand acceleration methods. The concept of gradient correction, which uncovers the mechanism behind acceleration, is found within the high-resolution differential framework \cite{shiUnderstandingAccelerationPhenomenon2021}. This framework notably explains the linear convergence of FISTA in composite optimization problems when the smooth function are strongly convex \cite{li2024linear}. These advancements pave the way for further research and development in optimization algorithms. It is crucial to recognize that, despite their similar forms, the high-resolution differential equation and ISHD have different motivations. The high-resolution differential equation is designed for NAG, whereas the second-order information in ISHD is derived from Newton's method \cite{attouch2014dynamical}, which is used to analyze the forward-backward algorithm.

Besides the high-resolution framework and the ISHD, the acceleration methods have also been explained through the lens of numerical stability \cite{luoDifferentialEquationSolvers2021,zhang2021revisiting}. Additionally, some studies consider closed-loop ODEs as analogous to adaptive optimization methods. For instance, a closed-loop dynamical system has been proposed to analyze high-order tensor algorithms for unconstrained smooth optimization from a control-theoretic perspective \cite{linControltheoreticPerspectiveOptimal2021}. Another study explores the development of fast optimization methods through inertial continuous dynamics with nonlinear damping \cite{attouch2023closed}.

There are several works focused on designing iterative optimization methods by discretizing ODEs. In \cite{zhangDirectRungeKuttaDiscretization2018}, the authors demonstrate that to ensure the $l$-th Runge-Kutta integrator applied to a subfamily of \eqref{eq:ISHD} with \(\alpha = 2p + 1\), \(\beta(t) \equiv 0\), and \(\gamma(t) = Cp^2t^{p-2}\), where \(p, C > 0\), is convergent, the objective function must be sufficiently flat and the step size must be diminishing. Other methods that develop iterative optimization techniques by discretizing corresponding ODEs typically only consider strongly convex cases \cite{cortes2019ratematching,vaquero2023resource,zhang2019discretization}, which are not general enough and may difficult to generalize to other scenarios.

Some research, such as \cite{taylor2023systematic}, considers the optimal selection of ODEs through the performance estimation problem (PEP). However, this approach does not focus on designing practical iterative optimization methods. Another straightforward approach is to tune the coefficients analytically. By exploiting the geometric properties of the function \( f \) in \eqref{eq:ISHD}, several studies focus on tuning coefficients analytically with provable fast convergence rate \cite{aujolFastConvergenceInertial2022,aujolOptimalConvergenceRates2019,aujolConvergenceRatesHeavyBall2022,aujol2023fista}. Two main properties used in these works are the growth condition and flatness condition. When these conditions are absent, we show that the coefficients can be tuned numerically using machine learning techniques.

\subsubsection{Learning to optimize}

Learning to optimize (L2O) is a paradigm that leverages machine learning to automate the design of optimization algorithms. This approach aims to uncover the underlying structure of a collection of optimization problems through machine learning \cite{chen2022learning}. L2O can be broadly categorized into two classes: model-based methods and model-free methods. Model-based methods, which are derived from iterative optimization methods, involve identifying learnable parameters. Two prominent classes of model-based methods, algorithm unrolling \cite{monga2021algorithm} and plug-and-play \cite{ahmad2020plugandplay,kamilov2023plugandplay}, have been successfully applied to various tasks in signal processing and image processing.

Seminal work of model-based methods has demonstrated that the iterative shrinkage-thresholding algorithm (ISTA) can be viewed as a recurrent neural network (RNN), and by introducing learnable parameters, significant improvements in sparse coding can be achieved \cite{lecun2010learning}. This concept has been extended by unrolling the alternating direction method of multipliers (ADMM) algorithm to construct a neural network with prior knowledge of compressive sensing problems \cite{yang2020admmcsnet}. Further, by designing an efficient algorithm to compute the proximal operator induced by sparsity regularity, ISTA has been reformulated as a deep neural network in \cite{zhang2018istanet}, which demonstrates substantial improvements in compressed sensing performance.

In contrast, model-free methods do not rely on existing models. By treating designing an optimization algorithm as a policy search within a reinforcement learning framework, the improved optimization algorithms has been discovered \cite{li2017l2o}. The generalization of the L2O paradigm within a derivative-free black-box optimization context was realized through the use of RNNs \cite{2017l2owithoutgd}. Further advancements were made by \cite{2016l2obygd}, who employed long short-term memory networks (LSTMs) to implement learned optimization methods, achieving superior performance compared to generic algorithms. In \cite{chen2023lion}, the design of optimization algorithms is approached as a program search problem. It utilizes evolutionary search to develop the LION algorithm, which exhibits performance comparable to Adam. More recently, a transformer-based neural network has been proposed to represent the update step, incorporating a preconditioning matrix to enhance efficiency, as reported in \cite{luke2023transformerl2o}.


In spite of the significant performance improvement, only a few works establish a non-trivial convergence guarantee for learned optimization methods. For instance, the learned ISTA has been thoroughly analyzed, with \cite{liu2018theoreticallinear} showing that there exist weights that achieve a theoretical linear convergence rate on the LASSO problem, surpassing ISTA theoretically. An explicit formula to calculate the optimal learnable weights is further provided \cite{liu2019alista}, effectively removing most of the computational overhead of training, and numerical experiments verify the theoretical results. The necessary condition for a learned optimization method to converge is investigated in \cite{liu2023mathstructure}. The result partially addresses the primary difficulty in designing learned optimization algorithms with convergence guarantee. Other methods exist that search for parameters while ensuring convergence. For example, a general formulation with learnable parameters for solving nonsmooth composite optimization problems is presented \cite{banertDataDrivenNonsmoothOptimization2020}. The conditions for establishing the convergence property of the learned methods are then derived, and unsupervised learning is employed to find more efficient methods under these conditions. The effectiveness of this methodology in ensuring the convergence of learned optimizers on smooth unconstrained optimization and composite optimization problems is also demonstrated \cite{banert2024siopt}.

\subsection{Our contributions}
\label{sec:contribution}
The framework of this paper is presented in Figure \ref{fig:framework-total}.
\begin{figure}[htbp]
    \centering
    \includegraphics[width=0.7\textwidth]{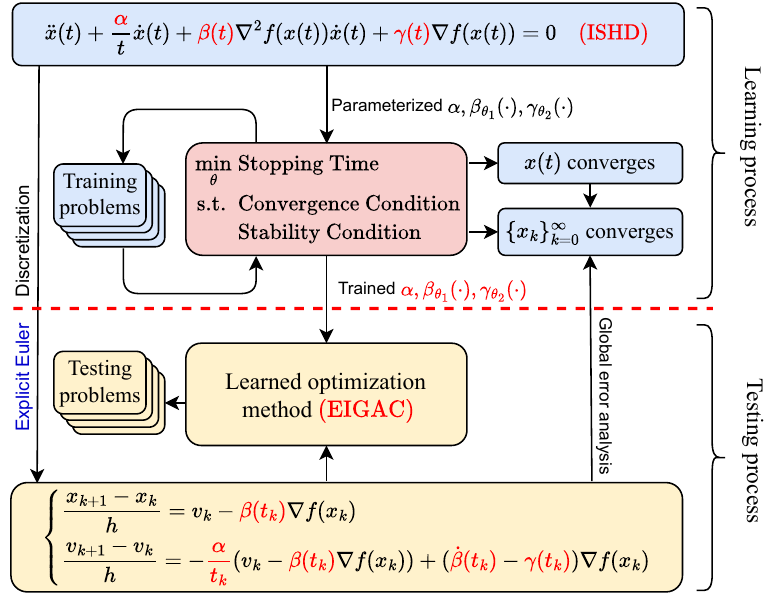}
    \caption{Our learning and testing framework.}
    \label{fig:framework-total}
\end{figure}

We address the first problem (\textbf{P1}) by providing conditions that ensure the convergence of the sequence $\{x_{k}\}_{k=0}^{\infty}$, which is generated by applying the explicit Euler discretization to equation \eqref{eq:ISHD}. These conditions comprises a convergence condition and a stability condition. The convergence condition ensures that the continuous-time trajectory $x(t)$ converges. The stability condition ensures that when applying the explicit Euler scheme to the ODE \eqref{eq:ISHD}, which satisfies the convergence condition, the discretization remains stable. The requirements of these conditions on the coefficients are relaxed, allowing us to generate numerous iterative optimization methods.

For the second problem (\textbf{P2}), we utilize a L2O framework and develop a corresponding algorithm to solve the L2O problem numerically. L2O is a paradigm that automatically designs optimization methods using machine learning techniques. The core idea is to leverage a training dataset of diverse optimization problems to tune the parameters of optimization methods, in order to improve their performance while ensuring convergence. L2O uses a metric of the efficiency of the optimization methods as the training loss. In our framework, we take the ``stopping time'' as the training loss, which measures the number of iterations that the the algorithms generated by discretizing the ODE \eqref{eq:ISHD} need to achieve a predefined suboptimality. This measure generalizes complexity from discrete-time to continuous-time cases. We then define the probability distribution of a parameterized function family by establishing equivalence with corresponding parameters. The training formulation of our L2O framework is posed as a problem that minimizes the expected stopping time under the expectation constraints that ensure the convergence and stability conditions hold for each function. Due to the existence of the expectation constraints, we combine the stochastic optimization methods with the penalty function method to solve the training formulation. Additionally, we derive conservative gradients of the stopping time and constraint functions, making our algorithm more robust and general. We also provide convergence guarantees for the training algorithm under the sufficient decrease assumption, using only the conservative gradients. Finally, we get the learned optimization methods by applying the explicit Euler scheme to the ODE \eqref{eq:ISHD} with learned coefficients.

Our contributions can be divided into two major parts. The first part involves an ODE methodology that derives the conditions necessary to ensure the convergence of the sequence obtained through discretization. The second part utilizes the L2O approach to find the optimal coefficients under these conditions. Although the convergence rate of the sequence is $\mathcal{O}(1/k)$ under these conditions, we argue that these conditions allow for more flexible coefficient selection and facilitate the design of various practical iterative optimization methods. We also provide default values that satisfy these conditions. Numerical experiments demonstrate that the sequence with these default values exhibits comparable performance to other optimization methods with a theoretical $\mathcal{O}(1/k^2)$ convergence rate. In contrast to most L2O methods, our framework has a solid theoretical foundation, which is a crucial step towards designing practical optimization methods. Compared to other methods that ensure the convergence of learned optimization methods \cite{banertDataDrivenNonsmoothOptimization2020,banert2024siopt}, our result also includes an explicit convergence rate, whereas their results do not. Furthermore, the theoretical linear rate in \cite{liu2018theoreticallinear} relies on the specific form of LASSO, whereas our result can be applied to various unconstrained convex smooth minimization problems. We summarize the key contributions of this paper as follows:
\begin{itemize}
\item We establish conditions under which the sequence $\{x_{k}\}_{k=0}^{\infty}$, obtained by applying the explicit Euler discretization to equation \eqref{eq:ISHD}, is guaranteed to converge. In the proof, we take a fresh perspective based on the global error analysis from numerical solutions of differential equations.

\item We introduce a metric named ``stopping time'', which is the continuous-time counterpart of the complexity. We show that ``stopping time'' is differentiable under mild conditions hence can be optimized numerically.

\item We propose a general L2O framework to select optimal coefficients for the ISHD, addressing the challenge of optimally tuning damping coefficients. An algorithm with convergence guarantee for this L2O framework is also derived.
\end{itemize}

We conduct extensive numerical experiments to evaluate our algorithms. The optimization problems include the logistic regression and $\ell_{p}^{p}$ minimization on real world datasets, e.g. \texttt{a5a}, \texttt{mushrooms}, \texttt{w3a}, \texttt{covtype}, and \texttt{phishing}. The training process demonstrates that our algorithm for the training formulation converges to a feasible stationary point of the training formulation, verifying the effectiveness of our theory. The numerical experiments on the testing process highlight the advantages of the learned algorithms. The code and datasets used to produce the results and analyses presented in this paper are available in the GitHub repository: \url{https://github.com/optsuite/O2O}.

\subsection{Organization}
The rest of the paper is organized as follows. We begin by analyzing the conditions for stabilizing the explicit Euler discretization of \eqref{eq:ISHD} in sec. \ref{sec:stabilize-explicit}. In sec. \ref{sec:model}, we present our methodology for finding efficient optimization methods that guarantee worst-case convergence, which leads to an stochastic optimization problem with expectation constraints. Next, sec. \ref{sec:conservative-grad} focuses on deriving expressions for the conservative gradients of the stopping time and constraint functions. We then analyze our model under modest hypotheses in sec. \ref{sec:theory}. The effectiveness of our training algorithm is demonstrated through numerical experiments in sec. \ref{sec:training-experiments}. Further numerical experiments are conducted in sec. \ref{sec:experiments} to showcase the efficiency and theoretical consistency of the learned optimization algorithms. Finally, we conclude our work and discuss future directions in sec. \ref{sec:conclusion}.

\section{Conditions that ensure a stable discretization of ISHD}
\label{sec:stabilize-explicit}

\subsection{Preliminaries}
To simplify the presentation, we bundle the coefficients \( \alpha \), \( \beta \), and \( \gamma \) into a single collection denoted as \( \upXi \). Considering the coefficients \( \upXi \), the function \( f \), and the given time \( t \), we define the trajectory associated with the system \( \eqref{eq:ISHD} \) as follows:
\begin{equation}\label{eq:def-flow}
\begin{aligned}
    X\colon&\mathbb{R}_{+}\times\mathcal{C}(\mathbb{R},\mathbb{R}_{+})\times\mathcal{C}(\mathbb{R},\mathbb{R}_{+})\times\mathbb{R}\times\mathcal{C}(\mathbb{R}^n,\mathbb{R}) \to \mathbb{R}^n,\\
    &(\upXi, t, f) \to X(\upXi, t, f).
\end{aligned}
\end{equation}
We set the initial conditions for system \( \eqref{eq:ISHD} \) as follows: the initial time \( t_0 > 0 \), initial position \( x(t_0) \), and initial velocity \( \dot{x}(t_0) \) are all fixed.

For our notation, \( \mathbb{S}^n \) and \( \mathbb{S}^n_+ \) represent the sets of all \( n \times n \) symmetric matrices and symmetric positive semidefinite matrices, respectively. We denote the convex hull of a set $S$ by $\operatorname{co}(S)$. The notation $[\cdot]_+$ represents the maximum of a value and zero, defined as $[\cdot]_+ := \max\{\cdot, 0\}$. The set \( \{0, 1, \ldots, M\} \), where \( M \) is a non-negative integer, is denoted by \( [M] \).

\subsection{A condition that ensures the trajectory of ISHD converges}
\label{sec:convergent}
Throughout this paper, we make the following basic assumptions about \eqref{eq:ISHD}.
\begin{assumption}\label{assump:differentiable}
    $\alpha> 1,t_0>0,\varepsilon>0$ are real numbers, $\beta$ and $\gamma$ are nonnegative continuously differentiable functions defined on $[t_0,+\infty)$. The function $f$ is twice differentiable convex with its domain $\operatorname{dom}f=\mathbb{R}^n$. The set of minimizers of $\min_{x} f(x)$ is not empty and the corresponding optimal value is $f_{\star}$.
\end{assumption}
Assumption \ref{assump:differentiable} also appears in the literature \cite{attouch2022fast,shiUnderstandingAccelerationPhenomenon2021} and is not restrictive. Given $\kappa\in(0,1],\lambda\in(0,\alpha-1]$, we define
\begin{equation}\label{eq:abbreviate-}
\begin{aligned}
&\delta(t)=t^2(\gamma(t)-\kappa\dot{\beta}(t)-\kappa\beta(t)/t)+(\kappa (\alpha-1-\lambda)-\lambda(1-\kappa))t\beta(t),\\
&w(t)=\gamma(t)-\dot{\beta}(t)-\beta(t)/t.
\end{aligned}
\end{equation}
Then, we provide a condition that ensures the convergence of the solution trajectory of \eqref{eq:ISHD}. 

\begin{theorem}\label{thm:continuous-convergence}
    Suppose that Assumption \ref{assump:differentiable} and the following conditions hold true:
    \begin{equation}\label{eq:converge-condition}
        \delta(t)>0,\quad \text{and}\quad \dot{\delta}(t)\leq \lambda tw(t).
    \end{equation}
    Then, the solution trajectory of \eqref{eq:ISHD}, $x(t)$, is bounded and the following inequalities can be derived:
    \begin{align}
    &f(x(t))-f_\star\leq\mathcal{O}\left(\frac{1}{\delta(t)}\right),\; \|\nabla f(x(t))\|\leq\mathcal{O}\left(\frac{1}{t\beta(t)}\right),\; \|\dot{x}(t)\|\leq\mathcal{O}\left(\frac{1}{t}\right),\label{eq:point-wise-estimate}\\
    &\int_{t_{0}}^{\infty}(\lambda tw(t)-\dot{\delta}(t))(f(x(t))-f_\star)\,\mathrm{d}t<\infty,\label{eq:value-integrable}\\
    &\int_{t_{0}}^{\infty}t(\alpha-1-\lambda)\|\dot{x}(t)\|^2\,\mathrm{d}t<\infty,\label{eq:velocity-integrable}\\
    &\int_{t_{0}}^{\infty}t^2\beta(t)w(t)\|\nabla f(x)\|^2\,\mathrm{d}t<\infty,\label{eq:grad-norm-integrable}\\
    &\int_{t_{0}}^{\infty}t^2\beta(t)\langle\nabla^2 f(x(t))\dot{x}(t),\dot{x}(t)\rangle\,\mathrm{d}t<\infty.\label{eq:hessian-velocity-integrable}
    \end{align}
\end{theorem}

The proof of Theorem \ref{thm:continuous-convergence} is presented in sec. \ref{sec:theory}. This theorem extends \cite[Theorem 1]{attouchFirstorderOptimizationAlgorithms2020} by providing estimations \eqref{eq:value-integrable}-\eqref{eq:hessian-velocity-integrable} for $f(x(t))-f_{\star}$, $\|\dot{x}(t)\|$, $\|\nabla f(x(t))\|$, and $\|\dot{x}(t)\|_{\nabla^2 f(x(t))}^2$. They guarantee that these quantities are integrable when coupling with certain coefficients. The results on \eqref{eq:point-wise-estimate} imply that the convergence rate of $f(x(t))-f_{\star}$, $\nabla f(x(t))$ can be controlled by the coefficients $\alpha,\beta(\cdot)$, and $\gamma(\cdot)$. Using the integrability with the convergence rate in \eqref{eq:point-wise-estimate}, we can analyze the long-time behavior of the solution trajectory $x(t)$. This provides the guidance for deriving the conditions that guarantee the convergence of the sequence discretized form the equation \eqref{eq:ISHD}.

\subsection{A condition that ensures the stability of the explicit Euler discretization}
\label{sec:stable}
Let $v(t_0)=x(t_0)+\beta(t_0)\nabla f(x(t_0))$ and
\begin{equation}\label{eq:psi}
    \psi_{\upXi}(x(t),v(t),t)
    =
    \left(
    \begin{gathered}
    v(t)-\beta(t)\nabla f(x(t))\\
    -\frac{\alpha}{t}(v(t)-\beta(t)\nabla f(x(t)))+(\dot{\beta}(t)-\gamma(t))\nabla f(x(t))
    \end{gathered}\right).
\end{equation}
The equation \eqref{eq:ISHD} can be reformulated as the first-order system
\begin{equation}\label{eq:first-order}
    \begin{pmatrix}
        \dot{x}(t)\\
        \dot{v}(t)
    \end{pmatrix}=\psi_{\upXi}(x(t),v(t),t).
\end{equation}
We denote the flow associated with this system as $s(t,s_{0},\upXi,f)$ with $s_{0}=(x_{0},v_{0})$. Let $h$ be the step size, $t_{k}=t_{0}+kh,k\geq 0$. The explicit Euler scheme of \eqref{eq:ISHD} writes
\begin{equation}\label{eq:first-order-discrete}
    \left\{
    \begin{aligned}
        &\frac{x_{k+1}-x_{k}}{h}=v_{k}-\beta(t_{k})\nabla f(x_{k}),\\
        &\frac{v_{k+1}-v_{k}}{h}=-\frac{\alpha}{t_{k}}\left(v_{k}-\beta(t_{k})\nabla f(x_{k})\right)+(\dot{\beta}(t_{k})-\gamma(t_{k}))\nabla f(x_{k}).
    \end{aligned}
    \right.
\end{equation}
The sequence $\{x_{k}\}_{k=0}^{\infty}$ denotes the position, while $\{v_{k}\}_{k=0}^{\infty}$ corresponds to the auxiliary variable sequence. Eliminating the auxiliary variable sequence $\{v_{k}\}_{k=0}^{\infty}$ in the equation \eqref{eq:first-order-discrete} gives
\begin{equation}
    \begin{aligned}
    x_{k+2}=&x_{k+1}-h^2\left(\gamma(t_{k})-\dot{\beta}(t_{k})+\frac{\beta(t_{k+1})-\beta(t_{k})}{h}\right)\nabla f(x_{k})\\
    &+\left(1-\frac{\alpha h}{t_{k}}\right)\underbrace{(x_{k+1}-x_{k})}_{\mathrm{inertia}}-h\beta(t_{k+1})\underbrace{\left(\nabla f(x_{k+1})-\nabla f(x_{k})\right)}_{\mathrm{gradient\ correction}}.
    \end{aligned}
\end{equation}
For clarity in our discussion, we reformulate the equation \eqref{eq:first-order-discrete} in Algorithm \ref{algo:EIGAC}. Since this algorithm is derived through the explicit Euler discretization of the ISHD, it includes an inertial term from the temporal discretization of the damping \(\dot{x}(t)\) and a gradient correction term from the temporal discretization of the Hessian-driven damping \(\nabla^2 f(x(t))\dot{x}(t)\). We name it the Explicit Inertial Gradient Algorithm with Correction (EIGAC). This algorithm is intended for the unconstrained smooth convex minimization problems:
\begin{equation}\label{eq:unconstrained-minimization}
\min_{x\in \mathbb{R}^{n}}\quad f(x),
\end{equation}
where $f$ satisfies Assumption \ref{assump:differentiable}. Importantly, the temporal discretization of the Hessian-driven damping $\nabla^2 f(x(t))\dot{x}(t)$ only contains first-order information, which makes EIGAC a first-order algorithm. Through the selection of various coefficients, Algorithm \ref{algo:EIGAC} offers a diverse range of optimization methods.

\begin{algorithm}[htbp]
    \caption{Explicit Inertial Gradient Algorithm with Correction (EIGAC)} \label{algo:EIGAC}
    \begin{algorithmic}[1]
    \State \textbf{Input:}
    the function $f$,
    the initial time $t_0$,
    the initial values $x_{0}$ and $v_{0}$,
    the step size $h$,
    the coefficients $\upXi=\{\alpha,\beta(\cdot),\gamma(\cdot)\}$.
    \For{ $k=1,2,\ldots$ }
        \State Update the time: $t_k=t_0+kh$.
        \State Update $x_{k}$ by \eqref{eq:first-order-discrete}.
        \State Update $v_{k}$ by \eqref{eq:first-order-discrete}.
    \EndFor
    \State \textbf{Output:} the solution $x_{\star}$ reaching the given accuracy.
    \end{algorithmic}
\end{algorithm}

Building upon the foundation established by Theorem \ref{thm:continuous-convergence}, we introduce a criterion termed the \emph{stability condition}. This condition ensures the convergence of the Algorithm \ref{algo:EIGAC}.

\begin{theorem}[Convergence rate]
    \label{thm:stable}
    Suppose Assumption \ref{assump:differentiable} and the convergence condition \eqref{eq:converge-condition} hold. Given an initial time $t_{0}$, an initial value $s_{0}$, and a step size $h$, the sequence $\{x_k\}_{k=0}^{\infty}$ is generated by Algorithm \ref{algo:EIGAC}. We denote the continuous time interpolation $\bar{x}(t)$ as
    \begin{equation}
        \label{eq:x-bar}
        \bar{x}(t)=x_{k}+\frac{x_{k+1}-x_{k}}{h}(t-t_{k}),\qquad t\in [t_{k},t_{k+1}).
    \end{equation}
    Assume three constants $0 \leq C_1$, $0 < C_2 \leq 1/h - 1/t_0$, and $0 < C_3$ fulfill the growth condition:
    \begin{equation}
        |\dot{\beta}(t)|\leq C_{1}\beta(t),\quad |\dot{\gamma}(t)-\ddot{\beta}(t)|\leq C_2 (\gamma(t)-\dot{\beta}(t)),\quad \beta(t)\leq C_3 w(t).\label{eq:trun-error-growth-condition-1}
    \end{equation}
    Then, it holds $f(x_k) - f_{\star} \leq \mathcal{O}(1/k)$ under the following stability condition:
    \begin{align}
        &\upLambda(x,f)\geq\|\nabla^2 f(x)\|,\quad {\alpha}\beta(t)/{t}\leq\gamma(t)-\dot{\beta}(t)\leq\beta(t)/h,\label{eq:stab-condition-1}\\
        &\sqrt{\int_{0}^{1}\upLambda((1-\tau)X(t,\upXi,f)+ \tau \bar{x}(t),f)\,\mathrm{d}\tau}\leq\frac{\sqrt{\gamma(t)-\dot{\beta}(t)}+\sqrt{\gamma(t)-\dot{\beta}(t)-\frac{\alpha}{t}\beta(t)}}{\beta(t)}.\label{eq:stab-condition-2}
    \end{align}
\end{theorem}

\begin{remark}
    Theorem \ref{thm:stable} demonstrates that for an $L$-smooth convex function $f$, selecting $\upLambda(x,f)\equiv L$, $\alpha > 3$, $\beta(t) = (4 - 2\alpha h/t)/L$, and $\gamma(t) = \beta(t)/h$ guarantees the convergence of Algorithm \ref{algo:EIGAC}. This encompasses a fairly large family of algorithms.
\end{remark}
\begin{remark}
    The condition \eqref{eq:trun-error-growth-condition-1} is not restrictive, since it only requires the growth rate of the coefficients does not excess the exponential rate. In practice, they are often dominated by some polynomials. Our stability condition shows that the Hessian-driven term $\nabla^2 f(x(t))\dot{x}(t)$ is crucial for achieving the stable discretization. When $\beta\equiv 0$, which corresponding to the vanishing Hessian-driven damping, the condition \eqref{eq:stab-condition-1} cannot be satisfied.
\end{remark}
\begin{remark}
    When $\nabla^2 f$ is $L_{H}$-Lipschitz continuous, the condition \eqref{eq:stab-condition-2} can be simplified as
    \begin{equation}
        \beta(t)\sqrt{\upLambda(X(t,\upXi,f),f)+\frac{L_{H}M_{3}}{2\sqrt{t}}}\leq {\sqrt{\gamma(t)-\dot{\beta}(t)}+\sqrt{\gamma(t)-\dot{\beta}(t)-\frac{\alpha}{t}\beta(t)}},\label{eq:stab-condition-2-simplified}
    \end{equation}
    where $M_{3}$ is a constant determined by $f,\upXi,x_{0},t_{0},\kappa$, and $\lambda$. We prove this conclusion as a corollary of Theorem \ref{thm:stable} in sec. \ref{sec:theory}. The simplified condition \eqref{eq:stab-condition-2-simplified} is important for reducing the computation overhead.
\end{remark}

The proof of Theorem \ref{thm:stable} is deferred in sec. \ref{sec:theory}. We establish this result by controlling the deviation between the sequences $\{x_{k}\}_{k=0}^{\infty}$ and $\{X(\upXi,t_{k},f)\}_{k=0}^{\infty}$. The key step involves decomposing the global discretization error and adjusting the coefficients to prevent it from exploding during the propagation process. This results in the stability condition that aligns with the conventional wisdom in numerical solutions of ODEs: the curvature of the trajectory cannot be excessively large.

The observation in sec. \ref{sec:intro} underscores that the convergence rate in Theorem \ref{thm:stable} cannot be directly analogous to its continuous-time counterpart, $\mathcal{O}(1/\delta(t))$, as presented in Theorem \ref{thm:continuous-convergence}. A stable discretization, achievable over a broad range of coefficients, comes at the cost of a reduced convergence rate. The primary factor in this reduction is the error introduced during the discretization process. To the best of our knowledge, in comparison with other conditions that ensure the stable discretization of \eqref{eq:ISHD}, our results offer distinct advantages: 1) They allow for more flexible requirements on the coefficients, whereas other approaches may mandate specific values for the coefficients; 2) They introduce tools from the numerical solution of differential equations to analyze iterative optimization algorithms, which we believe is a valuable effort in bridging these two fields; 3) They theoretically establish the stability of an explicit discretization with a constant step size applied to optimization-inspired ODEs in a general convex setting, contrasting with other convergence results that necessitate modifying the setting to strong convexity or diminishing step size \cite{cortes2019ratematching,vaquero2023resource,zhangDirectRungeKuttaDiscretization2018,zhang2019discretization}.

\section{Selecting the best coefficients of ISHD using L2O}
\label{sec:model}

\subsection{The problem formulation of L2O}

Before introducing the loss function used in the problem formulation of L2O, we briefly review \emph{oracle complexity}. It assesses the efficiency of optimization methods by the computational effort required to achieve a specified level of suboptimality. Suppose $\mathcal{F}$ is a collection that contains a family of functions with certain structure, e.g., convex functions, $L$-smooth functions. The iterative algorithm $\mathcal{M}(\cdot,\cdot,\cdot)$ maps the triplet $(f,x_{0},N)$ to the point $x_{N}=\mathcal{M}(f,x_{0},N)$. The computational cost for obtaining the output $\mathcal{M}(f,x_{0},N)$ is proportional to $N$. Let $m(f,x_n)$ be an optimality measure. For example, a popular choice in unconstrained smooth optimization is $m(f,x_n)=\|\nabla f(x_n)\|$. Then, the complexity for method $\mathcal{M}$ applied to function class $\mathcal{F}$ is defined as
\begin{equation}\label{eq:complexity-measure}
    N_{\mathcal{F}}:=\inf\{N \colon m(f,\mathcal{M}(f,x_{0},N))\leq\varepsilon,\text{ for all }f\in\mathcal{F}\}.
\end{equation}

The analysis of complexity has motivated the development of numerous efficient optimization algorithms, as comprehensively discussed in \cite{nemirovskiProblemComplexityMethod1983}. Inspired by this paradigm of developing efficient algorithms through complexity analysis, we define the continuous-time counterpart of complexity using the trajectory of \eqref{eq:ISHD} as a surrogate for optimization methods in the continuous-time case, which serves as the loss function in the problem formulation of L2O.

\begin{definition}[Stopping Time]
\label{def:stopping-time}
    Given the initial time $t_0$, the initial value $x_0$, the initial velocity $\dot{x}(t_{0})$, the trajectory $X(\upXi,t,f)$ of the ISHD \eqref{eq:ISHD}, and a tolerance $\varepsilon$, the stopping time of the criterion $\|\nabla f(x)\|\leq \varepsilon$ is
    \begin{equation}
        T(\upXi,f)=\inf\{t\mid \|\nabla f(X(\upXi,t,f))\|\leq \varepsilon,t\geq t_{0}\}.
    \end{equation}
\end{definition}

From the definition, we have $T=\infty$ if $\inf_{t\geq t_0}\|\nabla f(X(\upXi,t,f))\|\geq \varepsilon$. We mention that while $x(t)$ glosses over the challenges in numerical computation, the stopping time measures the efficiency of Algorithm \ref{algo:EIGAC} when the corresponding ODE can be stably discretized with a fixed step size.  The stopping time measures the efficiency of Algorithm \ref{algo:EIGAC}, when the sequence converges to the optimal value.

As we will see in sec. \ref{sec:conservative-grad}, the stopping time is differentiable with respect to $\upXi$. This fact gives us the hope of integrating the stopping time into L2O. In L2O, a Bayesian approach is used to describe a (parametric) function class $\mathcal{F}$. Consider a mapping from a set of parameters to a set of functions:
\begin{equation}
\mathcal{H}: \upOmega\to \mathcal{F},\quad\zeta \to f(\cdot;\zeta).
\end{equation}
To ensure measurability, we assume that $\mathcal{H}$ is a bijection, which allows us to define a probability distribution on $\mathcal{F}$ through the probability distribution on $\upOmega$.

\begin{definition}[Induced Probability Space]
\label{def:func-distribution}
    Given the probability space of the parameter $\zeta$, \((\upOmega, \mathscr{A}, \mathbb{P})\), where $\upOmega$ is the sample space, $\mathscr{A}$ is the $\sigma$-algebra, and $\mathbb{P}$ is the corresponding probability. We define a \(\sigma\)-algebra \(\mathcal{H}(\mathscr{A})\) as
    \begin{equation}
        \mathcal{H}(\mathscr{A}) = \{\mathcal{H}(A) \mid A \in \mathscr{A}\}.
    \end{equation}
    The probability for an event \(\mathcal{B} \in \mathcal{H}(\mathscr{A})\) is given by \(\mathbb{P}_{\mathcal{H}}(\mathcal{B}) = \mathbb{P}(\{\upOmega \mid \mathcal{H}(\upOmega) \in \mathcal{B}\})\). Consequently, the space \((\mathcal{F}, \mathcal{H}(\mathscr{A}), \mathbb{P}_{\mathcal{H}})\) is isomorphic to \((\upOmega, \mathscr{A}, \mathbb{P})\), and it is referred to as the induced probability space of \(\mathcal{F}\).
\end{definition}

In this paper, to support intuitive understanding, we assume that when discussing the distribution of functions, integrability over this distribution is maintained, and that the differentiation operator with respect to variables and the expectation operator with respect to the functions are interchangeable. A thorough examination of the interchangeability between the integral and differentiation operators will be left for future work.

Using this definition, we obtain an explicit expression for the expectation of $f$. This expression is exemplified by using the stopping time:
\[
\mathbb{E}_{f}[T(\upXi,f)]=\int_{f}T(\cdot,f)\,\mathrm{d}\mathbb{P}_{\mathcal{H}}(f)=\int_{\zeta}T(\cdot,\mathcal{H}^{-1}(f))\,\mathrm{d}\mathbb{P}(\zeta)=\mathbb{E}_{\zeta}[T(\cdot,f(\cdot,\zeta))].
\]
In this way, we can directly calculate the expectation with respect to a probability distribution of functions. Given that the spaces \((\mathcal{F}, \mathcal{H}(\mathscr{A}), \mathbb{P}_{\mathcal{H}})\) and \((\upOmega, \mathscr{A}, \mathbb{P})\) are isomorphic, we frequently interchange their notation without causing confusion.

Now, we present the problem formulation for selecting the best coefficients of the equation \eqref{eq:ISHD}. Given the probability space of the parameterized functions $(\mathcal{F},\mathcal{H}(\mathscr{A}),\mathbb{P}_{\mathcal{H}})$, we use the expectation of stopping time over $\mathbb{P}$ to measure the efficiency of the trajectory of the equation \eqref{eq:ISHD} with coefficients $(\alpha,\beta(\cdot),\gamma(\cdot))$. To make the stopping time also be a reasonable metric of the efficiency of Algorithm \ref{algo:EIGAC}, we add the stability condition \eqref{eq:stab-condition-1}, \eqref{eq:stab-condition-2} and convergence condition \eqref{eq:converge-condition} for each $f\in\mathcal{F}$ as constraints. These constraints ensure that the sequence $\{x_{k}\}_{k=0}^{\infty}$ generated by Algorithm \ref{algo:EIGAC} with the learned coefficients converges for each $f\in\mathcal{F}$.

To simplify the notations used in our L2O problem, with $w(t),\delta(t)$ defined in \eqref{eq:converge-condition}, and an integrable function satisfying $\upLambda(x,f)\geq \|\nabla^2 f(x)\|$, we introduce
\begin{align}
    &p(x,\bar{x},\upXi,t,f)=\Bigg[\beta(t)\sqrt{\int_{0}^{1}\upLambda((1-\tau)x+\tau \bar{x},f)\,\mathrm{d}\tau}-\sqrt{\gamma(t)-\dot{\beta}(t)}\notag\\
    &\qquad\qquad\qquad\qquad\qquad\qquad\qquad\qquad\qquad\qquad-\sqrt{\gamma(t)-\dot{\beta}(t)-\frac{\alpha}{t}\beta(t)}\Bigg]_+,\label{eq:def-p}\\
    &q(\upXi,t)=\left[\gamma(t)-\dot{\beta}(t)-\beta(t)/h\right]_++\left[\dot{\beta}(t)+\alpha\beta(t)/t-\gamma(t)\right]_+\notag\\
    &\qquad\qquad\qquad\qquad\qquad\qquad\qquad\qquad\qquad\qquad+\left[\dot{\delta}(t)-\lambda tw(t)\right]_++\left[-\delta(t)\right]_+.\label{eq:def-q}
\end{align}
Using notations \eqref{eq:def-p} and \eqref{eq:def-q}, we set
\begin{equation}
    P(\upXi,f)=\int_{t_{0}}^{T(\upXi,f)}p(X(t,\upXi,f),\bar{x}(t),\upXi,t,f)\,\mathrm{d}t,\quad Q(\upXi,f)=\int_{t_{0}}^{T(\upXi,f)}q(\upXi,t)\,\mathrm{d}t.\label{eq:def-PQ}
\end{equation}
Here $\upXi$ represents the triplet of coefficients $(\alpha,\beta(\cdot),\gamma(\cdot))$. Using the property of integration and notations in \eqref{eq:def-PQ}, simply setting $P(\upXi,f)\leq 0$ and $Q(\upXi,f)\leq 0$ ensures that the conditions \eqref{eq:converge-condition}, \eqref{eq:stab-condition-1}, and \eqref{eq:stab-condition-2} hold almost surely in $[t_0,T(\upXi,f)]$.

The L2O problem writes
\begin{subequations}
    \label{eq:learn-continuous}
    \begin{align}
    \min_{\upXi}\quad&\mathbb{E}_{f}[T(\upXi,f)],\label{eq:learn-obj-continuous}\\
    \text{s.t.}\quad
    &\mathbb{E}_{f}[P(\upXi,f)]\leq 0,\label{eq:constr-P}\\
    &\mathbb{E}_{f}[Q(\upXi,f)]\leq 0.\label{eq:constr-Q}
    \end{align}
\end{subequations}
Using the property of the expectation, constraints \eqref{eq:constr-P} and \eqref{eq:constr-Q} guarantee that $P(\upXi,f)\leq 0$ and $Q(\upXi,f)\leq 0$ hold almost surely for functions in $\mathcal{F}$. As we have argued, $P(\upXi,f)\leq 0$ and $Q(\upXi,f)\leq 0$ are sufficient to guarantee that the convergence condition \eqref{eq:converge-condition} and stability condition \eqref{eq:stab-condition-1}, \eqref{eq:stab-condition-2} hold almost surely in interval $[t_{0},T(\upXi,f)]$. Hence, given the integrability of \eqref{eq:constr-P}, \eqref{eq:constr-Q}, these constraints ensure that the learned ODE is convergent and can be stably discretized until it achieves the specific suboptimality metric $\varepsilon$, for every function that can be sampled from the distribution $\mathbb{P}$.

In a word, the L2O problem \eqref{eq:learn-continuous} selects the best ODE by minimizing the expectation of stopping time over a distribution of functions while keeping the stability of the discretization using scheme \eqref{eq:first-order-discrete} for every function that can be sampled from the distribution. We argue that the constraints of the problem are indispensable. Without the guarantee for the stability, one may achieve arbitrary fast convergence rate in continuous time. However, this rate can not be translated into discrete-time cases, and does not derive a practical optimization method.

\subsection{Solving the L2O problem using penalty method and stochastic optimization}
\label{sec:algorithm}

To solve problem \eqref{eq:learn-continuous} numerically, we parameterize $\beta(\cdot)$ as $\beta_{\theta_1}(\cdot)$ and $\gamma(\cdot)$ as $\gamma_{\theta_2}(\cdot)$, and collectively denote the parameters $(a, \theta_1, \theta_2)$ as $\theta$, the total count of all parameters as $n_{\theta}$. For instance, when $\beta$ and $\gamma$ are parameterized through neural networks, $\theta_1$ and $\theta_2$ represent the trainable weights. Alternatively, one might opt for different models, such as polynomials, for parameterization. To streamline the notation, we substitute the symbol $\upXi$ with $\theta$ to distinguish between models with and without parameterization. Consequently, the mapping $\psi_{\upXi}$, delineated in \eqref{eq:psi}, is rearticulated as $\psi_{\theta}$. The functions $P$ and $Q$, initially introduced in \eqref{eq:def-PQ}, are now expressed as $P(\theta,f)$ and $Q(\theta,f)$, respectively.

In problem \eqref{eq:learn-continuous}, both the objective function and the constraint functions are expectations of the same probability distribution of functions $f$. Based on the linearity of expectation, it is appealing to use an exact penalty method for solving problem \eqref{eq:learn-continuous}. Given the penalty parameter $\rho$, the $\ell_{1}$ exact penalty problem for framework \eqref{eq:learn-continuous} writes
\begin{equation}
\label{eq:penalty-function}
\begin{aligned}
    \min_{\theta}\;\upUpsilon(\theta)&=\mathbb{E}_{f}[T(\theta,f)]+\rho\left(\mathbb{E}_{f}[P(\theta,f)]+\mathbb{E}_{f}[Q(\theta,f)]\right)\\
    &=\mathbb{E}_{f}\left[T(\theta,f)+\rho\left(P(\theta,f)+Q(\theta,f)\right)\right].
\end{aligned}
\end{equation}
We omit the $\ell_{1}$ norm of the constraints, since they are nonnegative functions. The penalty function has an expectation form. Hence, we can solve problem \eqref{eq:penalty-function} using a suitable stochastic optimization method.

Before approaching the algorithm, we give the formulas for calculating the gradients of different components of $\upUpsilon(\theta)$ directly. We implicitly assume that the required level of smoothness is satisfied. When the parameterization is based on neural networks, this can be achieved via using differentiable activation functions like SoftPlus \cite{zheng2015improving}. The formula for  of the stopping time writes
\begin{equation}
    \label{eq:nablaT}
    \frac{\mathrm{d} T}{\mathrm{d}\theta}=\left(\nabla f(X)^\top\nabla^2 f(X)\left(v(T)-X-\beta(T)\nabla f(X)\right)\right)^{-1}\int_{t=t_0}^{T}r^\top\frac{\partial \psi_{\theta}}{\partial\theta} \,\mathrm{d}t,
\end{equation}
The notation $X$ represents the position of the trajectory at time $T$ and $v$ is calculated through the equation \eqref{eq:first-order}. The function $r$ is a solution of the following ODE
\begin{equation}\label{eq:backward}
    \left\{
    \begin{aligned}
    &\dot{r}(t)=-r(t)^\top\frac{\partial\psi_{\theta}}{\partial s}(s(t),t) \text{ from } t_1 \text{ to } t_0,\\
    &r(t_1)=(\nabla f(X)^\top\nabla^2 f(X),\mathbf{0}^\top)^\top.
    \end{aligned}
    \right.
\end{equation}
Applying the formula of the gradient of the stopping time and the chain rule to \eqref{eq:def-PQ}, we have
\begin{equation}
\label{eq:nabla_PQ}
    \begin{aligned}
\frac{\mathrm{d}P(\theta,f)}{\mathrm{d} \theta} &=\int_{t_{0}}^{T(\theta,f)}\frac{\partial \psi_{\theta}(s(t),t,f)}{\partial \theta}w(t)+ \frac{\partial p(x(t),\bar{x}(t),\theta,t,f)}{\partial \theta}\,\mathrm{d}t\\
&\qquad\qquad\qquad\qquad\qquad +p(x(T),\bar{x}(T),\theta,T,f)\frac{\mathrm{d}T(\theta,f)}{\mathrm{d} \theta},\\
\frac{\mathrm{d}Q(\theta,f)}{\mathrm{d} \theta} &=\int_{t_{0}}^{T(\theta,f)}\frac{\mathrm{d} q(\theta,t)}{\mathrm{d} \theta}\,\mathrm{d}t+q(\theta,T)\frac{\mathrm{d}T(\theta,f)}{\mathrm{d} \theta},
    \end{aligned}
\end{equation}
where $\bar{x}(\cdot)$ is the interpolation defined in \eqref{eq:x-bar}, $w(\cdot)$ is the solution of the following differential equation
\begin{equation}
    -\begin{pmatrix}
        \frac{\partial p(x(t),\bar{x}(t),\theta,t,f)}{\partial x}\\
        0_{n\times 1}
    \end{pmatrix}-\frac{\partial \psi_{\theta}(s(t),t,f)}{\partial s}w(t)=\dot{w}(t),\quad w(T(\theta,f))=0_{2n\times 1}.
\end{equation}
The gradient \(\partial p/\partial x\) are contingent on the estimator $\upLambda(x,f)$ for the local Lipschitz constant of \(\nabla f\). A detailed discussion is reserved for sec. \ref{sec:lip-estimator}.

The derivation of these formulas are postponed in sec. \ref{sec:conservative-grad}. All of them can be computed using the automatic differentiation, which largely reduces the complexity in practical implementation. We also offer a version for cases where the parameterized coefficients lack smoothness. This alternative is grounded in a recently developed concept known as the conservative gradient, a generalized derivative that accommodates the use of non-smooth activation functions in neural network parameterization. Using conservative gradients, we can also confirm that our implementation based on the automatic differentiation is reasonable and has a sound theoretical foundation.

Given the formula for the gradients with respect to the objective function and constraint functions, we use stochastic gradient descent to minimize the exact \(\ell_1\) penalty function \eqref{eq:penalty-function}. The procedure is outlined in Algorithm \ref{algo:SEPM}. In iteration \(k\), a function \(f_{k}\) is sampled from the distribution of functions, and the sampled penalty function is constructed. Then, the gradient is calculated, followed by a stochastic gradient update. This algorithm could be enhanced with mini-batch training, variance reduction, and adaptive optimization methods. For simplicity, we implement the basic form of SGD.

\begin{algorithm}[htbp]
    \caption{Stochastic Penalty Method (StoPM) for Problem \eqref{eq:learn-continuous}}
    \label{algo:SEPM}
    \begin{algorithmic}[1]
    \State \textbf{Input:}
    step size for Algorithm \ref{algo:EIGAC}: $h$,
    step size for stochastic gradient descent: $\eta$,
    initial weight: $\theta_{0}$,
    penalty coefficient: $\rho$,
    a probability space of functions: $(\upOmega,\mathscr{A},\mathbb{P})$.
    \While{ Not\texttt{(Stopping Condition)} }
        \State Sample a function in $f_{k}\in\upOmega$ according to the probability $\mathbb{P}$.
        \State Computing the (conservative) gradients $\frac{\mathrm{d}T}{\mathrm{d} \theta},\frac{\mathrm{d}P}{\mathrm{d} \theta}$ and $\frac{\mathrm{d}Q}{\mathrm{d} \theta}$ using \eqref{eq:nablaT} and \eqref{eq:nabla_PQ} (\eqref{eq:conservative-T}, \eqref{eq:conservative-P}, and \eqref{eq:conservative-Q}).
        \State Update the parameter: $\theta_{k+1}\gets\theta_{k}-\eta \left(\frac{\mathrm{d}T}{\mathrm{d} \theta}+\rho\left(\frac{\mathrm{d}P}{\mathrm{d} \theta}+\frac{\mathrm{d}Q}{\mathrm{d} \theta}\right)\right)$.
        \State Update index: $k\gets k+1$.
    \EndWhile
    \State \textbf{Output:} the trained weight $\theta_{\star}$.
    \end{algorithmic}
\end{algorithm}

We note that step 4 in Algorithm \ref{algo:SEPM} allows for the use of conservative gradients in nonsmooth cases. While the conservative gradient does not possess an explicit formula, it can be computed through automatic differentiation. This approach, albeit a bit indirect, is viable and practical.

\section{Deriving the conservative gradient of the penalty function}
\label{sec:conservative-grad}

In this section, our primary goal is to enhance the theoretical rigor of Algorithm \ref{algo:SEPM}. We utilize the conservative gradient to characterize the outputs obtained from applying automatic differentiation to $\upUpsilon(\theta)$. This characterization becomes essential when the parameterization includes nonsmooth components such as the ReLU and max operators, which are prevalent in neural network architectures. The conservative Jacobian offers a generalized approach to derivatives in scenarios where traditional differentiability does not apply. This is particularly pertinent to both forward and backward nonsmooth automatic differentiation techniques crucial in machine learning, especially within backpropagation algorithms. The parameterization discussed here aligns with that introduced in sec. \ref{sec:algorithm}. For readers more interested in the computational and implementation aspects of the algorithm, it is sufficient to directly apply the corollaries following each theorem, as they are intuitive and straightforward.

\subsection{The conservative Jacobian of the flow $X(t,\theta,f)$}

This subsection investigates the conservative Jacobian for the flow $X(t,\upXi,f)$, starting with basic concepts in set-valued maps. The double arrow notation $\rightrightarrows$ used in set-valued maps like ${D}: \mathbb{R}^{d_{1}} \rightrightarrows \mathbb{R}^{d_{2}}$ indicates that each element in the domain maps to a subset of the codomain, highlighting the multi-valued aspect of these mappings. A set-valued map is termed \emph{locally bounded} if, for any $x \in \mathbb{R}^{d_{1}}$, there exists a neighborhood $V_x$ such that the union $\bigcup_{y \in V_x} {D}(y)$ remains a bounded subset of $\mathbb{R}^{d_{2}}$. Furthermore, such a map is \emph{graph closed} if for any converging sequences \( (x_k)_{k \in \mathbb{N}} \subseteq \mathbb{R}^{d_{1}} \) and \( (v_k)_{k \in \mathbb{N}} \subseteq \mathbb{R}^{d_{2}} \), with \( v_k \in D(x_k) \), the limit $\lim_{k \rightarrow \infty} v_k$ belongs to $D(\lim_{k \rightarrow \infty} x_k)$. These properties are critical in defining and understanding the conservative Jacobian, as detailed below.

\begin{definition}[Conservative Jacobian]\label{def:conservative-jacobian}
Consider a nonempty, locally bounded, and graph-closed set-valued map $D: \mathbb{R}^{d_{1}} \rightrightarrows \mathbb{R}^{m \times d_{1}}$, and a locally Lipschitz continuous function $F: \mathbb{R}^{d_{1}} \rightarrow \mathbb{R}^m$. The map $D$ is termed the conservative Jacobian of $F$ if and only if
$$
\frac{\mathrm{d}}{\mathrm{d}t}F(r(t)) = A\dot{r}(t), \quad \text{for all } A \in D(r(t)),
$$
holds for any absolutely continuous curve $r: [0, 1] \rightarrow \mathbb{R}^{d_{1}}$ and almost every $t \in [0, 1]$.
When $m = 1$, we refer to $D$ as a conservative gradient.
\end{definition}

Based on conservative Jacobian, we give the definition of path differentiability. This concept is introduced in \cite{bolte2021conservative} and enables the application of the chain rule and various calculus operations on nonsmooth functions.

\begin{definition}[Path differentiability]\label{def:path-diff}
    A function $F: \mathbb{R}^d \rightarrow \mathbb{R}^m$ is termed path differentiable if there exists a set-valued map $D$ that serves as a conservative Jacobian for $F$.
\end{definition}

To investigate the relationship between the path differentiability of the vector field and the flow in a general ODE, we consider the ordinary differential equation for $t_{1}\geq t_{0}>0$
\begin{equation}\label{eq:general-ODE}
    \dot{Y}(t)=F(Y(t)),\quad Y(0)=y_{0},\quad\forall t\in[t_{0},t_{1}],
\end{equation}
where $F\colon\mathbb{R}^d\to\mathbb{R}^{d}$ is a path differentiable Lipschitz function and $y_{0}\in\mathbb{R}^d$. Let $\phi(y_{0},t)$ be the flow associated with \eqref{eq:general-ODE} and $D^{F}$ be a uniformly bounded convex valued conservative Jacobian of function $F$, i.e., $\sup_{y\in\mathbb{R}^{d},J\in D^{F}(y)}\|J\|_{\mathrm{op}}\leq K$ for certain $K>0$. Then, the adjoint sensitivity equation of \eqref{eq:general-ODE} based on conservative Jacobian writes
\begin{equation}\label{eq:sensitivity-equation}
    \begin{aligned}
        &\dot{A}(t)\in D^{F}(\phi(y_{0}, t))A(t), & \text{for almost all } t \in [t_{0},t_{1}], \\
        &A(t_{0})= I \in \mathbb{R}^{d \times d}.
    \end{aligned}
\end{equation}
Since $D^{F}$ is uniformly bounded, the existence of the solutions of \eqref{eq:sensitivity-equation} is established by \cite[Theorem 4, p. 101]{aubin1984differential}.

Using these notations, we now present a fundamental theorem on the path differentiability of ODE flows. For a comprehensive proof, which is both lengthy and complex, the reader is referred to Theorem 1 of \cite{marx2022path}.

\begin{theorem}[Path differentiability of ODE flows]
    \label{thm:path-diff}
    Define the set-valued map $U: \mathbb{R}^d \times [t_0, t_1] \to \mathbb{R}^{d \times d}$ as:
    \begin{equation}
        \begin{aligned}
            U\colon \mathbb{R}^d \times [t_{0},t_{1}] &\to \mathbb{R}^{d \times d},\\
            (y_{0}, t)&\rightrightarrows U(y_{0}, t):=A(t), \quad A(\cdot)\text{ is a solution to \eqref{eq:sensitivity-equation}}.
        \end{aligned}
    \end{equation}
    The map $y_0 \rightrightarrows U(y_0, t_1)$ constitutes a conservative Jacobian for the mapping $y_0 \mapsto \phi(y_0, t_1)$.
\end{theorem}

This theorem demonstrates that a path differentiable vector field ensures the path differentiability of the corresponding ODE flow. We have adapted the notation for clarity and to maintain consistency with the conventions of this paper. To apply it in \eqref{eq:first-order}, we need the following projection lemma, adopted from \cite{marx2022path}, to focus on the part of the conservative Jacobian we are interested in.

\begin{lemma}[The projection preserves the conservativity]
    \label{lem:proj-conservative}
    Let \( G(y, z) : \mathbb{R}^{d_{y}+d_{z}} \rightarrow \mathbb{R}^m \) be a path differentiable function whose conservative Jacobian is denoted by \( D^{G} : \mathbb{R}^{d_{y}+d_{z}} \rightarrow \mathbb{R}^{m \times (d_{y}+d_{z})} \). Consider
    \[ \upPi_{z}D^{G}(y, z) := \{M_2 \in \mathbb{R}^{m \times d_{z}}, \exists M_1 \in \mathbb{R}^{m \times d_{y}}, (M_1, M_2) \in D^{G}(y, z)\}. \]
    Then, for all \( y \in \mathbb{R}^{d_{y}} \), \( \upPi_{z}D^{G}(y, z) \) is a conservative Jacobian for the function \( z \mapsto G(y, z) \).
\end{lemma}

For clarity, we denote the conservative Jacobian of function \(G\) with respect to (projected onto) \(z\) as \(D^{G}_{z}(y,z) = \upPi_{z}D^{G}(y,z)\). Note that we may omit some variables when they are clear from the context. Applying Theorem \ref{thm:path-diff} and Lemma \ref{lem:proj-conservative}, we confirm the path differentiability of the flow $X(t,\theta,f)$ associated with \eqref{eq:first-order}. We consider the parameterized vector field $\psi_{\theta}$ in equation \eqref{eq:first-order} as a mapping from \(\mathbb{R}^{2n+1+n_{\theta}}\) to \(\mathbb{R}^{2n}\), defined as:
\[
\psi\colon (s,t,\theta) \mapsto \psi_{\theta}(s,t).
\]
Here, we use $\psi$ to highlight that the parameter $\theta$, previously considered fixed, is now treated as a variable alongside \(s\) and \(t\). This allows us to analyze the behavior of the vector field as both the input variables and the parameter vary.

\begin{theorem}\label{thm:path-diff-customized}
Assume that \(\psi\) is path differentiable, its corresponding conservative Jacobian \(D^{\psi} \colon \mathbb{R}^{2n+1+n_{\theta}} \rightrightarrows \mathbb{R}^{2n \times (2n+1+n_{\theta})}\) exists and $(x(t),v(t))$ is the flow of the equation \eqref{eq:first-order}. Let \(J_{s}^{\psi}(t)\) and \(J_{\theta}^{\psi}(t)\) be measurable selections such that for all \(t \in [t_{0}, t_{1}]\), \(J_{s}^{\psi}(t) \in D^{\psi}_{s}(x(t),v(t),t,\theta)\) and \(J_{\theta}^{\psi}(t) \in D^{\psi}_{\theta}(x(t),v(t),t,\theta)\) (existence ensured by \cite[Lemma 2]{marx2022path}). Consider \(A(\cdot)\) as the solution to the matrix differential inclusion
\begin{equation}\label{eq:conservative-jacobian}
    \begin{aligned}
        &\dot{A}(t)=J_{s}^{\psi}(t)A(t)+J_{\theta}^{\psi}(t),\\
        &A(t_{1})=0_{2n\times n_{\theta}}\quad\text{for all } t\in[t_{0},t_{1}].
    \end{aligned}
\end{equation}
Then, the map \(\theta \mapsto A(t_{0})\) is a conservative Jacobian of \(\theta \to s(t,s_{0},\theta,f)\), where $s$ is the flow associated with the equation \eqref{eq:first-order}.
\end{theorem}

\begin{proof}
    Concatenate the vector as \(Y(t) = (x(t), v(t), t, \theta)\), giving the autonomous form of \eqref{eq:first-order} as \(\dot{Y}(t) = F(Y(t))\) where
    \begin{equation}
    \label{eq:def-F}
    F(Y(t)) = \begin{pmatrix}
        \psi_{\theta}(x(t), v(t), t)\\
        1\\
        0_{n_{\theta} \times 1}
    \end{pmatrix}.
    \end{equation}
    With \(\psi\) being path differentiable, the function \(F\) also exhibits path differentiability. A conservative Jacobian of $F$ writes
    \[
        \begin{aligned}
            D^{F}\colon\mathbb{R}^{2n}\times\mathbb{R}\times\mathbb{R}^p&\rightrightarrows\mathbb{R}^{(2n+1+n_{\theta})\times(2n+1+n_{\theta})}\\
            (\hat{x},\hat{v},t,\theta)&\rightrightarrows\begin{pmatrix}
                J_{s} & J_{t} & J_{\theta}\\
                0_{1\times 2n} & 1 & 0_{1\times n_{\theta}}\\
                0_{n_{\theta}\times 2n} & 0_{n_{\theta}\times 1} & 0_{n_{\theta}\times n_{\theta}}
            \end{pmatrix},\quad \forall (J_{s},J_{t},J_{\theta})\in D^{\psi}.
        \end{aligned}
    \]
    Invoking Theorem \ref{thm:path-diff}, we partition the matrix $A(t)$ appearing in the sensitivity equation \eqref{eq:sensitivity-equation} as
    \[
        A(t)=\begin{pmatrix}
            A_{1}(t)&A_{2}(t)&A_{3}(t)\\
            A_{4}(t)&A_{5}(t)&A_{6}(t)\\
            A_{7}(t)&A_{8}(t)&A_{9}(t)
        \end{pmatrix},
    \]
    where $A_{1}(t)\in\mathbb{R}^{2n\times 2n},A_{5}(t)\in\mathbb{R},A_{9}(t)\in\mathbb{R}^{n_{\theta}\times n_{\theta}}$. Combining the partition above with the boundary condition $A(t_{1})=I_{(2n+1+n_{\theta})\times(2n+1+n_{\theta})}$ yields
    \[
        \begin{aligned}
            &A_{4}(t)\equiv 0_{1\times 2n},&&A_{5}(t)=\exp(t-t_{1}),&&A_{6}(t)\equiv 0_{1\times n_{\theta}},\\
            &A_{7}(t)\equiv 0_{n_{\theta}\times 2n},&&A_{8}(t)\equiv 0_{n_{\theta}\times 1},&&A_{9}\equiv I_{n_{\theta}\times n_{\theta}}.
        \end{aligned}
    \]
    The matrix differential inclusion \eqref{eq:sensitivity-equation} resolves to
    \begin{align}
        &\dot{A}_{1}(t)=J_{s}(Y(t))A_{1}(t), &&A_{1}(t_0)=I_{2n\times 2n},\label{eq:conservative-flow}\\
        &\dot{A}_{2}(t)=J_{s}(Y(t))A_{2}(t)+J_{t}(Y(t))\exp(t-t_{1}),&&A_{2}(t_0)=0_{2n\times 1},\notag\\
        &\dot{A}_{3}(t)=J_{s}(Y(t))A_{3}(t)+J_{\theta}(Y(t)),&&A_{3}(t_0)=0_{2n\times n_{\theta}},\label{eq:diff-inclu}\\
        &\text{for all }(J_{s},J_{t},J_{\theta})\in D^{\psi},t\in[t_0,t_{1}].\notag
    \end{align}
    Applying Lemma \ref{lem:proj-conservative}, we know $A_{3}$ is a conservative Jacobian of $\theta\to X(t_{1},\theta,f)$. Omitting the subscription of $A_{3}$ leads to \eqref{eq:conservative-jacobian}.\qed
\end{proof}

Theorem \ref{thm:path-diff-customized} demonstrates that the flow $s$ maintains the path differentiability of the vector field $\psi_{\theta}$ and gives the form of the conservative gradient $D^{s}_{\theta}(t,s_{0},\theta,f)$. By considering the part associated with $x$ of $D^{s}_{\theta}$, say the first $n$ rows, we get $D^{X}_{\theta}(t,\theta,f)$, the conservative gradient of $X(t,\theta,f)$, where the influence of \(x_0\) and \(v_0\) omitted as they are fixed. Consequently, we primarily need to verify the conditions ensuring the path differentiability of $\psi_{\theta}$. Notably, when parameterization employs neural networks, the vector field $\psi_{\theta}$ is almost invariably path differentiable. This follows from the observation that each element of $\psi_{\theta}$ is derived through finite operations of addition, subtraction, multiplication, and division applied to $t$, $\nabla f(x)$, $v$, $\alpha$, and the neural networks $\dot{\beta}_{\theta_{1}}$, $\beta_{\theta_{1}}$, and $\gamma_{\theta_{2}}$. Assuming the path differentiability of $\nabla f$ with respect to $x$, and based on the principle that the product and composition of path differentiable functions are also path differentiable, the remaining task is to ensure the path differentiability of the parameterized $\dot{\beta}{\theta{1}}$, $\beta_{\theta_{1}}$, and $\gamma_{\theta_{2}}$. Indeed, the path differentiability of most neural networks can be substantiated using \cite[Theorem 8]{bolte2021conservative}.

We now give a definition of the $o$-minimal structures, following \cite{coste2000introduction,van1996geometric}.
\begin{definition}[$o$-minimal structure]\label{def:o-minimal}
    An $o$-minimal structure is a collection \( \{\mathcal{O}_d\}_{d\in\mathbb{N}} \), where $\mathcal{O}_d$ is a family of subsets of \( \mathbb{R}^d \) such that for each \( d \in \mathbb{N} \):
    \begin{enumerate}
        \item $\mathcal{O}_{d}$ contains $\mathbb{R}_{d}$ and is stable under the operations of complementation, finite union, and finite intersection, meaning that any set resulting from these operations on elements of \(\mathcal{O}_{p}\) will also be a member of \(\mathcal{O}_{p}\).
        \item if \( A \) belongs to \( \mathcal{O}_d \), then \( A \times \mathbb{R} \) and \( \mathbb{R} \times A \) belong to \( \mathcal{O}_{d+1} \);
        \item if \( \upPi \colon \mathbb{R}^{d+1} \to \mathbb{R}^d \) denotes the coordinate projection on the first $d$ coordinates, then for any \( A \in \mathcal{O}_{d+1} \), we have $\upPi(A)\in\mathcal{O}_d$;
        \item \( \mathcal{O}_d \) contains all sets of the form \( \{ x \in \mathbb{R}^d \colon p(x) = 0\} \), where \( p \) is a polynomial on \( \mathbb{R}^d \);
        \item the elements of \( \mathcal{O}_1 \) are exactly the finite unions of intervals (possibly infinite) and points.
    \end{enumerate}
\end{definition}
The sets \( A \) belonging to \( \mathcal{O}_d \), for some \( d \in \mathbb{N} \), are called definable in the $o$-minimal structure. A set valued mapping (or a function) is said to be definable in $\mathcal{O}$ whenever its graph is definable in $\mathcal{O}$. In algorithmic aspect, it has been proved in \cite{bolte2021conservative} that the forward and backward automatic differentiation outputs a conservative Jacobian for the corresponding definable function \cite[Theorem 8]{bolte2021conservative}.

A typical example of the $o$-minimal structures is given by semialgebraic sets. A set \( A \subseteq \mathbb{R}^d \) is called \textit{semialgebraic} if it is a finite union of sets of the form
\begin{equation}
    \{ x \in \mathbb{R}^d : p_i(x) = 0, i = 1, \ldots, k; p_i(x) < 0, i = k+1, \ldots, m\},
\end{equation}
where $p_{f}$ are real polynomial functions and \( k \geq 1 \). We remark that the property 3 of Definition \ref{def:o-minimal} for semialgebraic sets is not trivial and can be obtained from the Tarski-Seidenberg theorem. Another profound result by Wilkie \cite{wilkie1996model} shows that there exists an $o$-minimal structure that contains both the exponential function $x\mapsto e^{x}$ and all semialgebraic functions. According to the inverse function theorem of definable functions appearing in Chapter 7.3 of \cite{dries1998tame}, the logarithm function $x\mapsto \log(x)$ is also definable to Wilkie's $o$-minimal structure. The Proposition 1.6 of \cite{coste2000introduction} ensures the composition of two definable maps is definable.

Using these results, we conclude that the following functions are definable using Wilkie's $o$-minimal structure.
\begin{align*}
    t &\mapsto t,              & t &\mapsto t^2,              & t &\mapsto \log(1 + \exp(t)), \\
    t &\mapsto \max\{0, t\},   & t &\mapsto \tanh(t),         & t &\mapsto \frac{1}{1 + \exp(-t)}.
\end{align*}
Additionally, all polynomials and functions whose domains can be segmented into finitely many intervals, coinciding with the previously mentioned functions, are definable. These results confirm that almost all activation functions utilized in deep learning are definable. Consequently, \cite[Theorem 8]{bolte2021conservative} validates the existence of conservative fields for the parameterized $\psi_{\theta}$ within neural networks.

\subsection{The conservative gradient of the stopping time $T(\theta,f)$}

We present a characterization of the conservative gradient of the stopping time, which is grounded in the integration by parts formula for absolutely continuous functions.
\begin{lemma}[Integration by parts]\label{lem:int-parts}
    Suppose $ a,b\colon [t_{0},t_{1}]\to\mathbb{R}^{d} $ are absolutely continuous functions. We have
    \[
    \int_{t_{0}}^{t_{1}} a(t)^{\top} \dot{b}(t)\,\mathrm{d}t=a(t_{1})^{\top} b(t_{1})-a(0)^{\top} b(0)-\int_{t_{0}}^{t_{1}} \dot{a}(t)^{\top} b(t)\,\mathrm{d}t.
    \]
\end{lemma}

The proof of Lemma \ref{lem:int-parts} is provided in Theorem 10 of Chapter 6 in \cite{royden1968real}. Utilizing this lemma in conjunction with Theorem \ref{thm:path-diff-customized}, we derive the conservative gradient of \(\|\nabla f(X(T, \theta, f))\|^2\), which serves as a fundamental component in the characterization of the conservative gradient of the stopping time.

\begin{proposition}\label{prop:conservative-grad-time}
    Adopting the notations from Theorem \ref{thm:path-diff-customized}, let $D^{s}_{s_{0}}$ and $D^{s}_{\theta}$ denote the conservative Jacobians of $s(t,\theta,s_{0},f)$ with respect to $s_{0}$ and $\theta$, respectively. Assume $J_{s_{0}}^{s}\colon [t_{0},T]\to\mathbb{R}^{2n\times 2n}$ and $J^{s}_{\theta}\colon [t_{0},T]\to\mathbb{R}^{2n\times n_{\theta}}$ are measurable selections satisfying $J_{s_{0}}^{s}(t) \in D^{s}_{s_{0}}(t,\theta,s_{0},f)$ and $J_{\theta}^{s}(t) \in D^{s}_{\theta}(t,\theta,s_{0},f)$ for all $t \in [t_{0},T]$. According to \cite[Chapter 0, Theorem 2]{filippov1988differential}, there exists a unique absolutely continuous solution $ \omega\colon [t_0,T]\to\mathbb{R}^{2n} $ to the differential equation
    \begin{equation}
        \label{eq:sensitivity-equation-time}
        \begin{aligned}
            &\dot{\omega}(t) = -J_{s_{0}}^{s}(t)\omega(t),\\
            &\omega(T) = (\nabla f(X)^\top\nabla^2 f(X),0_{1\times n})^\top. 
        \end{aligned}
    \end{equation}
    Then, an element of $D^{\|\nabla f(X(T,\cdot,f))\|^2}$ is given by $\int_{t=t_{0}}^{T} \omega(t)^\top J_{\theta}^{\psi}(t)\,\mathrm{d}t$.
\end{proposition}

\begin{proof}
    Consider the solution $A(t)$ of the sensitivity equation \eqref{eq:conservative-jacobian} for computing the conservative Jacobian of $X(t,\theta,f)$ in Theorem \ref{thm:path-diff-customized}. For any absolutely continuous function $\omega\colon [t_0,T]\to\mathbb{R}^{2n}$, applying Lemma \ref{lem:int-parts} and the initial condition $A(t_{0})=0_{2n\times n_{\theta}}$ gives
    \[
        \begin{aligned}
            \int_{t_{0}}^{T}\omega(t)^\top \dot{A}(t)\,\mathrm{d}t
            &=\omega(T)^\top {A}(T)-\int_{t_{0}}^{T}A(t)^\top\dot{\omega}(t)\,\mathrm{d}t\\
            &=\int_{t_{0}}^{T}\omega(t)^\top J_{s}^{\psi}(t)A(t)+\omega(t)^\top J_{\theta}^{\psi}(t)\,\mathrm{d}t.
        \end{aligned}
    \]
    Setting $\omega$ as the solution of \eqref{eq:sensitivity-equation-time} results in:
    \[
        \nabla f(X)^\top \nabla^2 f(X)A(T) = \int_{t_{0}}^{T}\omega(t)^\top J_{\theta}^{\psi}(t)\,\mathrm{d}t.
    \]
    Given that the product of conservative gradients maintains the conservative gradient property \cite[Lemma 5]{bolte2021conservative}, the expression on the left hand side of the equation qualifies as an element of the conservative gradient in question. This completes the proof.\qed
\end{proof}

The next theorem derives the conservative gradient of $T(\theta, f)$ with respect to $\theta$, applying the formal non-smooth implicit differentiation via the conservative Jacobian, as detailed in \cite[Theorem 2]{bolte2021implicit}.

\begin{proposition}\label{prop:conservative-implicit}
    Let the squared gradient norm $\|\nabla f(X(\cdot, \cdot, f))\|^2\colon\mathbb{R}\times\mathbb{R}^{n_{\theta}}\to\mathbb{R}$ be path differentiable on an open set $\mathcal{U}\times\mathcal{V}\subset\mathbb{R}\times\mathbb{R}^{n_{\theta}} $ and $T(\theta,f)\colon\mathcal{V}\to\mathcal{U} $ a locally Lipschitz function such that, for each $\theta\in\mathcal{V}$
    \begin{equation}
        \|\nabla f(X(T(\theta,f),\theta,f))\|^2=\varepsilon^2.
    \end{equation}
    Assume that for each $\theta\in\mathcal{V}$ and for each $B\in D^{\|\nabla f(X(\cdot,\cdot,f))\|^2}_{t}(T(\theta,f),\theta)$, the matrix $B$ is invertible. Then, $T:\mathcal{V}\to\mathcal{U}$ is path differentiable, with its conservative gradient given by:
    \begin{equation}\label{eq:conservative-T}
    D^{T}(\theta):=\{-B^{-1}A : \forall A\in D^{\|\nabla f(X(\cdot,\cdot,f))\|^2}_{\theta}, B\in D^{\|\nabla f(X(\cdot,\cdot,f))\|^2}_{t}\},
    \end{equation}
    where $D^{\|\nabla f(X(\cdot,\cdot,f))\|^2}_{\theta}$ is the conservative gradient of $\|\nabla f(X(T(\theta,f),\cdot,f))\|^2$, as specified by Proposition \ref{prop:conservative-grad-time}.
\end{proposition}

\begin{proof}
    Let $\theta(\iota)\colon[0,1]\to\mathcal{U}$ be an absolutely continuous loop. Since $T$ is locally Lipschitz, the composition function $T\circ \theta$ must also be absolutely continuous. Consequently, $\|f(X(\theta(\iota),T(\theta(\iota),f),f))\|^2$ is differentiable with respect to $\iota$ almost everywhere. By differentiating it and applying the chain rule, for almost every $\iota\in[0,1]$ and for any $A\in D^{\|\nabla f(X(\cdot,\cdot,f))\|^2}_{\theta}$, $B\in D^{\|\nabla f(X(\cdot,\cdot,f))\|^2}_{t}$, we have:
    \begin{equation}
    A\dot{\theta}(\iota) + B\frac{\mathrm{d}T}{\mathrm{d}\theta}\dot{\theta}(\iota) = 0.
    \end{equation}
    Thus, according to the definition of a conservative gradient \ref{def:conservative-jacobian}, we conclude that $-B^{-1}A$ is an element of the conservative gradient of $T(\theta,f)$.\qed
\end{proof}

While the conservative gradient of $T$ can indeed be computed using automatic differentiation, the process is not inherently intuitive or straightforward. To facilitate a clearer grasp of the underlying theories, we give a direct derivation of the gradient of $T$ in smooth cases, starting with a corollary of Proposition \ref{prop:conservative-grad-time}. This corollary calculates the gradient of $\|\nabla f(X(t_{1},\theta,f))\|^2$ in smooth cases.

\begin{corollary}\label{prop:adjoint}
    Given $\psi_{\theta}$ is continuously differentiable with respect to $\theta$, $r\in \mathcal{C}^{1}([t_0,t_1],\mathbb{R}^{2n})$ is the solution of \eqref{eq:backward}, we have
    \begin{equation}\label{eq:sensitivity}
        \nabla f(X)^\top\nabla^2 f(X){\frac{\partial X}{\partial \theta}}=\int_{t=t_0}^{t_1}r^\top\frac{\partial\psi_{\theta}}{\partial\theta} \,\mathrm{d}t.
    \end{equation}
\end{corollary}
\begin{proof}
    Consider the ODE system
    \[
        \frac{\mathrm{d}s}{\mathrm{d}t}(t) = \psi_{\theta}(s(t),t), \quad t\in [t_0,t_1],\quad s(t_0) = (x_0^\top,v_0^\top)^\top.
    \]
    For any continuously differentiable function $r(t)\in \mathcal{C}^{1}([t_0,t_1],\mathbb{R}^n)$, we have
    \[
    \int_{t=t_0}^{t_1}r(t)^\top \,\mathrm{d}s(t)=\int_{t=t_0}^{t_1}r(t)^\top \psi_{\theta}(s(t),t)\,\mathrm{d}t.
    \]
    Applying ${\mathrm{d}}/{\mathrm{d}\theta}$ to above equation and using $\frac{\mathrm{d}}{\mathrm{d}\theta}\frac{\mathrm{d}}{\mathrm{d}t}s(t)=\frac{\mathrm{d}}{\mathrm{d}t}\frac{\mathrm{d}}{\mathrm{d}\theta}s(t)$ yields
    \[
    \int_{t=t_0}^{t_1}r^\top \,\mathrm{d}\left(\frac{\mathrm{d}s}{\mathrm{d}\theta}\right)=\int_{t=t_0}^{t_1}r^\top \left(\frac{\partial\psi_{\theta}}{\partial s}\frac{\mathrm{d}s}{\mathrm{d}\theta}+\frac{\partial\psi_{\theta}}{\partial\theta}\right) \,\mathrm{d}t.
    \]
    Here we omit the variables for simplicity. Integrating by parts and using the fact that $\mathrm{d} s(t_0)/\mathrm{d} \theta\equiv 0$ gives
    \begin{equation}\label{eq:mediate-form}
    r(t_1)^\top \frac{\mathrm{d}s}{\mathrm{d}\theta}(t_1)=\int_{t=t_0}^{t_1}\left(\dot{r}^\top + r^\top \frac{\partial\psi_{\theta}}{\partial s}\right)\frac{\mathrm{d}s}{\mathrm{d}\theta}+r^\top\frac{\partial\psi_{\theta}}{\partial\theta} \,\mathrm{d}t.
    \end{equation}
    It should be noted that $\mathrm{d}s/\mathrm{d}\theta$ includes the term $\mathrm{d}x/\mathrm{d}\theta$. To evaluate the value of $\nabla f(X)^\top\nabla^2 f(X){\partial X}/{\partial \theta}$, we solve the backward ODE \eqref{eq:backward}. Plugging the solution into equation \eqref{eq:mediate-form} gives the formula \eqref{eq:sensitivity}.\qed
\end{proof}
\begin{remark}\label{rmk:direct-general-adjoint}
    Using the same technique in Corollary \ref{prop:adjoint}, we can give a direct derivation of Theorem \ref{thm:path-diff-customized} in smooth case through choosing $r(t_{1})=I_{:,i}$ for $i=1,2,\ldots,n$.
\end{remark}

Similar to the process for deriving conservative gradients, once the gradient of $\|\nabla f(X(t,\theta,f))\|^2$ is established, the gradient of $T$ with respect to $\theta$ can be computed using implicit function theorem. This calculation serves as a corollary to Proposition \ref{prop:conservative-implicit}.

\begin{corollary}\label{coro:nabla_T}
    Let \( T = T(\theta, f) \) and \( X = X(T(\theta, f), \theta, f) \). Assuming that \(\psi_{\theta}\) is continuously differentiable with respect to \(\theta\) and that \(\nabla f(X)^\top \nabla^2 f(X) \frac{\partial X}{\partial \theta}\) is non-zero, the gradient of \( T \) is then given by the formula \eqref{eq:nablaT}.
\end{corollary}
    \begin{proof}
        From the definition of the stopping time and the continuity of $\|\nabla f(\cdot)\|$, we have $\|\nabla f(X(T(\theta, f), \theta, f))\|^2- \varepsilon^2\equiv 0$. The implicit function theorem implies
    \begin{equation}\label{eq:IFT}
        \begin{aligned}
            \left.\frac{\partial \|\nabla f(X(t,\theta,f))\|^2}{\partial t}\right|_{t=T}\frac{\mathrm{d} T}{\mathrm{d}\theta}+\nabla f(X)^\top\nabla^2 f(X)\frac{\partial X}{\partial \theta}=0.
        \end{aligned}
    \end{equation}
    The second term can be calculated using Proposition \ref{prop:adjoint}. The chain rule gives
    \[
        \left.\frac{\partial \|\nabla f(X(t,\theta,f))\|^2}{\partial t}\right|_{t=T}=\nabla f(X)^\top\nabla^2 f(X)\left.\frac{\partial X}{\partial t}\right|_{t=T}.
    \]
    Using equation \eqref{eq:first-order}, we have
    \[
    \left.\frac{\partial X}{\partial t}\right|_{t=T}=\dot{x}(T)=v(T)-X-\beta(T)\nabla f(X),
    \]
    where $x(t)=X(t,\theta,f)$. Combining these results we get the formula \eqref{eq:nablaT}.\qed
\end{proof}

\subsection{The conservative gradients of the constraint functions $P(\theta,f)$ and $Q(\theta,f)$}
\label{sec:grad-PQ}

In this subsection, we present the conservative gradients of $P(\theta,f)$ and $Q(\theta,f)$. These functions can be expressed as definite integrals of the functions $p(\theta,t,f)$ and $q(\theta,t)$ over the interval $[t_0,T(\theta,f)]$. Assuming that both $p$ and $q$ are path differentiable, our results rely on a characterization of the conservative gradient of the integral of a path differentiable function. For the sake of completeness, we present this characterization here, while the proof can be found in \cite[Theorem 2]{marx2022path}. The notations used throughout this section are consistent with those introduced in Theorem \ref{thm:path-diff}.

\begin{theorem}
    \label{thm:forward-integral}
    Let $\delta\colon \mathbb{R}^{d}\to\mathbb{R}$ be a locally Lipschitz continuous and path differentiable function. Let $D^{\delta}\colon\mathbb{R}^{d}\to\mathbb{R}^{d}$ be a conservative gradient for $\delta$ with convex values. For $t_{1}>0$, set
    \[
        \upDelta(y_{0})=\int_{t_{0}}^{t_{1}}\delta(\phi(y_{0},t))\,\mathrm{d}t,
    \]
    where $\phi$ is the flow of the equation \eqref{eq:general-ODE}. Then, given any solution of \eqref{eq:sensitivity-equation}, denoted by $A(t)$, the following set valued field is a conservative gradient for $\upDelta$,
    \begin{equation*}
        D^{\upDelta}\colon y_{0}\rightrightarrows\left\{\int_{t_{0}}^{t_{1}}A(t)^\top w(t)\,\mathrm{d}t,w\in\mathcal{W}(y_{0})\right\}
    \end{equation*}
    where $\mathcal{W}(y_{0})$ is the set of measurable selections $w(t)\in D^{\delta}(\phi(y_{0},t))$.
\end{theorem}

Based on Theorem \ref{thm:forward-integral}, we provide the conservative gradients of $P$ and $Q$, given that the conservative gradients $D^{p}_{\theta}$ and $D^{q}_{\theta}$ can be readily evaluated.

\begin{proposition}
\label{prop:conservative-Q}
    Assume that the function $q(\theta,t)$ is path differentiable and $J^{q}_{\theta}(t)$ is a measurable selection of the conservative gradient $D^{q}_{\theta}(\theta,t)$, satisfying $J^{q}_{\theta}(t) \in D^{q}_{\theta}(\theta,t)$ for $t\in [t_{0}, T(\theta,f)]$. Then, the following holds:
    \begin{equation}
        \label{eq:conservative-Q}
        \int_{t_{0}}^{T(\theta,f)} J_{\theta}^{q}(t) \, \mathrm{d}t + q(\theta,T(\theta,f)) J^{T}_{\theta} \in D^Q,
    \end{equation}
    where $D^Q$ is the conservative gradient of $Q(\theta,f)$.
\end{proposition}
\begin{proof}
    The conservative gradient of $Q$ is determined by substituting $F$ from equation \eqref{eq:def-F} into Theorem \ref{thm:forward-integral}. Notably, since $q$ does not involve $x(t)$ or $v(t)$ as variables, this substitution leads directly to equation \eqref{eq:conservative-Q}.
    \qed
\end{proof}

Next, we explore the conservative gradient of $P$. The steps here are similar to those in Proposition \ref{prop:conservative-Q}. However, a key difference is that $p$ involves the estimator $\upLambda(x,f)$, with $X(t,\theta,f)$ as an input. This estimator is significant both theoretically and practically. We will discuss this term and its corresponding conservative gradient in the next subsection and directly apply $D^{q}_{x}$ in the following proposition.

\begin{proposition}
\label{prop:conservative-P}
    Assume that the functions $p(x,\bar{x},\theta,t,f)$ is path differentiable, $J_{x}^{p}(t)$ is a measurable selection of $D^{p}_{x}(X(t,\theta,f),\bar{x}(t),\theta,t,f)$, $J_{\theta}^{X}(t)$ is a measurable selection of $D^{X}_{\theta}(t,\theta,f)$, and $J_{\theta}^{p}(t)$ is a measurable selection of the conservative gradient $D^{p}_{\theta}(X(t,\theta,f),\bar{x}(t),\theta,t,f)$, $\omega$ is the unique solution to the following equation
    \begin{equation}\label{eq:sens-eq-constraints}
        -J^{p}_{s}(t)-J^{\psi}_{s}(t)\omega(t)=\dot{\omega}(t),\qquad \omega(T(\theta,f))=0_{2n\times 1}.
    \end{equation}
    We have $\int_{t_{0}}^{T(\theta,f)}J_{\theta}^{s}(t)^\top J_{s}^{p}(t)\,\mathrm{d}t=-\int_{t_{0}}^{T(\theta,f)}J^{\psi}_{\theta}(t)^\top w(t)\,\mathrm{d}t$ and
    \begin{equation}
        \label{eq:conservative-P}
        \int_{t_{0}}^{T}J^{\psi}_{\theta}(t)^\top w(t)+ J_{\theta}^{p}(t)\,\mathrm{d}t+p(X(T,\theta,f),\bar{x}(T),\theta,T,f)J^{T}_{\theta}\in D^P.
    \end{equation}
    For simplicity, we abbreviate \( T(\theta, f) \) as \( T \).
\end{proposition}
\begin{proof}
    The existence of the conservative gradients of $P,Q$ are ensured by Theorem \ref{thm:path-diff-customized} through augmenting the function $F$ with two additional dimensions $(p,q)$. Since the flow inherits the path differentiability of the corresponding vector field, we must have $P,Q$ are path differentiable given $p,q$ are path differentiable. Substituting the $F$ defined at the equation \eqref{eq:def-F} in Theorem \ref{thm:forward-integral}, the conservative gradient of $P$ writes
    \begin{equation}
        \label{eq:customized-integral-gradient}
        \int_{t_{0}}^{T}J_{\theta}^{X}(t)^\top J_{x}^{p}(t)+J_{\theta}^{p}(t)\,\mathrm{d}t+p(X(T,\theta,f),\bar{x}(T),\theta,T,f)J^{T}\in D^{P}(\theta).
    \end{equation}
    Proposition \ref{prop:conservative-grad-time} gives the method for obtaining an element $J^{T}$ from $D^{T}(\theta)$. However, a direct evaluation of $J_{\theta}^{X}(t)$ requires to solve a differential equation with the variable of the size $(2n+1+n_{\theta})\times (2n+1+n_{\theta})$. We derive a computational tractable way that only need solve a differential equation with size $(2n+1+n_{\theta})\times 1$ to evaluate the integral in \eqref{eq:customized-integral-gradient}. 
    Noticing that $J^{p}_{s}(t)=(J^{p}_{x}(t)^\top,0)^\top$, we know
    \[
        J_{\theta}^{X}(t)^\top J_{x}^{p}(t)=J_{\theta}^{s}(t)^\top J_{s}^{p}(t)
    \]
    holds for any $t$. We use the fact that $J^{s}_{\theta}(t)$ satisfies the equation \eqref{eq:conservative-jacobian}, say
    \[
        \frac{\mathrm{d}}{\mathrm{d}t}J^{s}_{\theta}(t)=J^{\psi}_{s}(t)^\top J^{s}_{\theta}(t)+J^{\psi}_{\theta}(t),\quad J^{s}_{\theta}(t_{1})=0_{2n\times p}.
    \]
    Setting $t_{1}=T(\theta,f)$, for any absolutely continuous function $\omega(t)\colon [t_{0},t_{1}]\to\mathbb{R}^{n}$, we have
    \begin{equation}\label{eq:u-int-by-parts}
        \begin{aligned}
        &\int_{t_{0}}^{t_{1}}J_{\theta}^{s}(t)^\top J_{s}^{p}(t)\,\mathrm{d}t\\
        =&\int_{t_{0}}^{t_{1}}J_{\theta}^{s}(t)^\top J_{s}^{p}(t)+\left(J^{\psi}_{s}(t)^\top J^{s}_{\theta}(t)+J^{\psi}_{\theta}(t) -J^{\psi}_{s}(t)^\top J^{s}_{\theta}(t) -J^{\psi}_{\theta}(t)\right)^\top w(t)\,\mathrm{d}t\\
        =&\int_{t_{0}}^{t_{1}}J_{\theta}^{s}(t)^\top J_{s}^{p}(t)+\left(J^{\psi}_{s}(t)^\top J^{s}_{\theta}(t)+J^{\psi}_{\theta}(t)\right)^\top w(t)\,\mathrm{d}t -\int_{t_{0}}^{t_{1}}\left(\frac{\mathrm{d}}{\mathrm{d}t}J^{s}_{\theta}(t)\right)^\top w(t)\,\mathrm{d}t.\\
        \end{aligned}
    \end{equation}
    Using Lemma \ref{lem:int-parts}, we get
    \begin{equation}
        \begin{aligned}
            &\int_{t_{0}}^{t_{1}}J_{\theta}^{s}(t)^\top J_{s}^{p}(t)\,\mathrm{d}t\\
            =&\int_{t_{0}}^{t_{1}}J_{\theta}^{s}(t)^\top J_{s}^{p}(t)+\left(J^{\psi}_{s}(t)^\top J^{s}_{\theta}(t)+J^{\psi}_{\theta}(t)\right)^\top w(t)\,\mathrm{d}t\\
            &\qquad\qquad\qquad - J^{s}_{\theta}(t_{1})w(t_{1})+J^{s}_{\theta}(t_{0})w(t_{0})+\int_{t_{0}}^{t_{1}}J^{s}_{\theta}(t)^\top \dot{w}(t)\,\mathrm{d}t\\
            =&\int_{t_{0}}^{t_{1}}J^{\psi}_{\theta}(t)^\top w(t)\,\mathrm{d}t+\int_{t_{0}}^{t_{1}}J_{\theta}^{s}(t)^\top \left(J_{s}^{p}(t)+J^{\psi}_{s}(t)w(t)+\dot{w}(t)\right)\,\mathrm{d}t\\
            &\qquad\qquad\qquad - J^{s}_{\theta}(t_{1})w(t_{1})+J^{s}_{\theta}(t_{0})w(t_{0}).
        \end{aligned}
    \end{equation}
    Given $J^{s}_{\theta}(t_{0})=0_{2n\times 1}$, setting $\omega$ to the solution of \eqref{eq:sens-eq-constraints} gives
    \begin{equation*}
        \int_{t_{0}}^{T(\theta,f)}J_{\theta}^{s}(t)^\top J_{s}^{p}(t)\,\mathrm{d}t=-\int_{t_{0}}^{T(\theta,f)}J^{\psi}_{\theta}(t)^\top w(t)\,\mathrm{d}t.
    \end{equation*}
    Combining this conclusion with \eqref{eq:customized-integral-gradient} completes our proof.\qed
\end{proof}

Using Remark \ref{rmk:direct-general-adjoint} and the same argument as Proposition \ref{prop:conservative-P}, we get a direct derivation of it in the smooth case. This is because the key steps \eqref{eq:customized-integral-gradient} and \eqref{eq:u-int-by-parts} can be similarly derived by replacing the conservative gradient in the argument with the corresponding partial derivatives. The only challenge lies in deriving a smooth version of Theorem \ref{thm:path-diff-customized}, which can be addressed using the approach of Proposition \ref{prop:adjoint} and the technique mentioned in Remark \ref{rmk:direct-general-adjoint}.

\begin{corollary}
    Suppose the functions $p,q$, and $T$ are differentiable with respect to $\theta$. Then, the gradients of $P$ and $Q$ are given by the equation \eqref{eq:nabla_PQ}.
\end{corollary}
\begin{proof}
    Applying the chain rule and the gradient of stopping time derived in Corollary \ref{coro:nabla_T} gives
    \begin{align*}
        \frac{\mathrm{d}P(\theta,f)}{\mathrm{d} \theta} &=\int_{t_{0}}^{T(\theta,f)}\frac{\partial p(x(t),\bar{x}(t),\theta,t,f)}{\partial x}\frac{\partial X(t,\theta,f)}{\partial \theta}+ \frac{\partial p(x(t),\bar{x}(t),\theta,t,f)}{\partial \theta}\,\mathrm{d}t\\
        &\qquad\qquad\qquad\qquad\qquad +p(x(T),\bar{x}(T),\theta,T,f)\frac{\mathrm{d}T(\theta,f)}{\mathrm{d} \theta},\\
        \frac{\mathrm{d}Q(\theta,f)}{\mathrm{d} \theta} &=\int_{t_{0}}^{T(\theta,f)}\frac{\mathrm{d} q(\theta,t)}{\mathrm{d} \theta}\,\mathrm{d}t+q(\theta,T)\frac{\mathrm{d}T(\theta,f)}{\mathrm{d} \theta}.
    \end{align*}
    Then, the equation \eqref{eq:nabla_PQ} can be proved verbosely like Lemma \ref{prop:conservative-P}. One simply needs to replace the conservative gradient $J^{\psi}_{s}$ with the partial gradient $\partial \psi/\partial s$ and so forth.
    \qed
\end{proof}

\subsection{Choices of the estimators $\upLambda(x,f)$ for the local Lipschitz constant of $\nabla f$}
\label{sec:lip-estimator}

This subsection discusses the computation of $D^{p}_{x}(X(t,\theta,f),\bar{x}(t),\theta,t,f)$. The following proposition characterizes the structure of the conservative gradient $D^{p}_{x}$.

\begin{proposition}
    \label{prop:conservative-p}
    Assume $\upLambda \colon \mathbb{R}^{n} \to \mathbb{R}$ is path differentiable. Then, an element of the conservative gradient $D^{p}_{x}(X(t,\theta,f),\bar{x}(t),\theta,t,f)$ is given by
    \begin{equation}
        \label{eq:conservative-p}
        \frac{\beta_{\theta_{1}}(t)}{2\sqrt{\int_{0}^{1}\upLambda((1-\tau)X(t,\theta,f)+\tau\bar{x}(t),f)\,\mathrm{d}\tau}} \int_{0}^{1}J^{\upLambda}_{x}(\tau)\,\mathrm{d}\tau,
    \end{equation}
    where $J^{\upLambda}_{x}(\tau)$ is a measurable selection satisfying
    \[
        J^{\upLambda}_{x}(\tau) \in D^{\upLambda}_{x}((1-\tau)X(t,\theta,f)+\tau \bar{x}(t),f),\qquad \forall \tau\in[0,1].
    \]
\end{proposition}
\begin{proof}
    The equation \eqref{eq:conservative-p} follows from the definition of conservative gradients and Theorem \ref{thm:forward-integral}.
    \qed
\end{proof}

As indicated by the proposition, the main challenge lies in computing the conservative gradient of the local Lipschitz constant estimator $\upLambda(x,f)$. We provide three choices of the estimator and derive the corresponding conservative gradients, with each choice determined by the underlying assumptions on the function \( f \). We first consider the most general case: functions with local Lipschitz continuous gradient.

\begin{proposition}
    \label{prop:conservative-largest-eigenvalue}
    Suppose $f$ is twice differentiable and $\nabla^2 f(x)$ is path differentiable. Denote $\upLambda(x,f)=\lambda_{\max}(\nabla^2 f(x))$ as the largest eigenvalue of $f$ on point $x$. Then, $\upLambda(x,f)$ is path differentiable and
    \begin{equation}
        \label{eq:direction-grad-lambda}
        \left.\frac{\mathrm{d}}{\mathrm{d}\eta_{2}}\left(\left.\frac{\mathrm{d}}{\mathrm{d}\eta_{1}}\nabla f(x+\eta_{1}z+\eta_{2}z)\right|_{\eta_{1}=0}\right)\right|_{\eta_{2}=0}\in D^{\upLambda}(x),
    \end{equation}
    where $\|z\|=1$ and $z^\top \nabla^2 f(x)z=\lambda_{\max}(\nabla^2 f(x))$.
\end{proposition}
\begin{proof}
     A well known formula for the largest eigenvalue is
    \begin{equation}
        \lambda_{\max}(A)=\sup_{\|z\|=1}z^{\top}Az,\text{ for all }A\in\mathbb{S}^{n}.
    \end{equation}
    Using the fact that the point-wise maximum of a family of convex functions is convex, we know $\lambda_{\max}$ is convex. Furthermore, the subgradient of the supremum function at any point is the convex hull of the gradients of those functions in the family that attain the supremum at that point. This property gives the following characterization of the subgradient of $\lambda_{\max}$:
    \begin{equation}
        \label{eq:conservative-eigv}
        \partial^C \lambda_{\max}(A)=D^{\lambda_{\max}}(A)=\operatorname{co}\left(\bigcup_{z^\top Az=\lambda_{\max}(A)}\{zz^\top\}\right).
    \end{equation}
    This helpful formula also appears in Corollary 2.4 of \cite{stewart1990matrix}. Here $\partial^{C}$ denotes the Clarke subdifferential as defined in Definition \ref{def:clarke-subdiff}. The first equality arises from the property of conservative gradients, as outlined in \cite{bolte2021conservative}, where the conservative gradient \( D \) coincides with \( \partial^{C} \) for convex functions. Combining the equation \eqref{eq:conservative-eigv} and the chain rule for conservative gradients \cite[Lemma 5]{bolte2021conservative}, we get the following characterization of the conservative gradient:
    \begin{equation}\label{eq:conservative-lambda}
        \begin{aligned}
            &D^{\upLambda}(x,f):=\left\{J^{\lambda_{\max}}\circ J^{\nabla^2 f}\mid J^{\lambda_{\max}}\in D^{\lambda_{\max}}(\nabla^2 f(x)),\;J^{\nabla^2 f}\in D^{\nabla^2 f}(x)\right\},\\
            &\text{where } D^{\lambda_{\max}}(\nabla^2 f(x))=\operatorname{co}\left(\bigcup_{z^\top \nabla^2 f(x)z=\lambda_{\max}(\nabla^2 f(x))}\{zz^\top\}\right).
        \end{aligned}
    \end{equation}
    The conservative gradient for $\nabla^2 f$ is a tensor of the size $\mathbb{R}^{n\times n\times n}$ and the operator $\circ$ is defined as
    \begin{equation}
        \left(J^{\lambda_{\max}}\circ J^{\nabla^2 f}\right)_{k}=\left\langle J^{\lambda_{\max}},J^{\nabla^2 f}_{k}\right\rangle,
    \end{equation}
    where the subscription $k$ denotes the projection onto the $k$-th coordinate of $x$.

    Next, we give a computationally tractable method to get an element of \eqref{eq:conservative-lambda}. Given the definition of the conservative gradients, we have
    \[
        \left.\frac{\mathrm{d}}{\mathrm{d}\eta}z^\top\nabla^2 f(x+\eta z)\right|_{\eta = 0}=zz^\top\circ J^{\nabla^2 f}, \quad \forall z\in\mathbb{R}^{n},\; J^{\nabla^2 f}\in D^{\nabla^2 f}(x).
    \]
    Setting $\|z\|=1$ and $z^\top \nabla^2 f(x)z=\lambda_{\max}(\nabla^2 f(x))$ gives $zz^\top \in D^{\lambda_{\max}}(\nabla^2 f(x))$. Hence, combining the equation above and \eqref{eq:conservative-lambda} gives the equation \eqref{eq:direction-grad-lambda}.
    \qed
\end{proof}

Note that the argument in Proposition \ref{prop:conservative-largest-eigenvalue} remains valid when the assumptions on \( f \) are more relaxed. Therefore, we do not repeat the derivation for the corollary in the smooth case. Next, we demonstrate that the computational complexity of \( zz^\top \circ J^{\nabla^2 f} \) is effectively just a constant multiple of the cost of evaluating \( f(x) \).

Suppose the function \( f(x) \) is automatically differentiable and can be evaluated with a computational complexity of \( \mathcal{O}(N) \). According to \cite{griewank2008evaluating}, the gradient \( \nabla f(x) \) can also be computed with the same complexity, \( \mathcal{O}(N) \), via backward differentiation. Furthermore, using forward differentiation for $\nabla f(x+\eta_{1}z)$, we can compute the directional derivative \( z^\top \nabla^2 f(x) \) for a vector \( z \) and a small perturbation \( \eta_1 \) with complexity \( \mathcal{O}(N) \). This method applies equally to \( z^\top \nabla^2 f(x + \eta_2 z) \), affirming that the complexity for evaluating the specified equation \( \eqref{eq:direction-grad-lambda} \) remains \( \mathcal{O}(N) \). In cases where the function \( f \) is not automatically differentiable, the gradient \( \nabla f \) must be calculated symbolically, but this does not change the overall computational complexity for these operations.

Using power iteration \cite{golub2013matrix} combined with forward automatic differentiation, the eigenvector \( z \) referenced in equation \eqref{eq:direction-grad-lambda} can be computed efficiently. The iterative process is defined as follows:
\[
    z_{k+1} = \frac{z_{k+1/2}}{\|z_{k+1/2}\|}, \quad z_{k+1/2} = \left. \frac{\mathrm{d}}{\mathrm{d}\eta_1} \nabla f(x + \eta_1 z_k) \right|_{\eta_1=0}.
\]
In our numerical experiments, iterating this process between 6 and 10 times typically yields a sufficiently accurate eigenvector \( z \). Thus, the computational complexity remains at a manageable level of \( \mathcal{O}(N) \), a constant multiple of the complexity for evaluating \( f \) and its derivatives.

To enhance the computational efficiency of the framework outlined in equation \( \eqref{eq:learn-continuous} \), we propose two alternatives for calculating the conservative gradient \( D^{\upLambda} \). In Theorem \( \ref{thm:continuous-convergence} \), we established that \( x(t) \) remains bounded under the convergence condition specified in equation \( \eqref{eq:converge-condition} \). This boundedness justifies the assumption that \( \nabla f \) has a global Lipschitz constant, particularly since \( \nabla^2 f(x) \) is continuous over the bounded set. One straightforward approach to estimating this constant is to use \( \|\nabla^2 f(x_0)\| \), under the assumption that the local Lipschitz constant decreases as the iterations progress. This estimation method does not rely on the changing values of \( x(t) \), thus simplifying the computational process. Additionally, under this assumption, the conservative gradient \( D^{p}_{x} \) consistently equals zero.

Using a global Lipschitz constant for \(\nabla f\) may limit the flexibility of our framework as represented in equation \(\eqref{eq:learn-continuous}\). As illustrated in Figure \(\ref{fig:smooth}\), the local Lipschitz constant can vary significantly during iterations. This variability is crucial for enhancing the performance of our algorithm, and our framework \(\eqref{eq:learn-continuous}\) is designed to capture such variations. The concept of \((L_{0}, L_{1})\)-smoothness, which has recently gained significant attention \cite{zhang2020l0l1smooth}, offers a useful compromise. A function is defined as \((L_{0}, L_{1})\)-smooth if it satisfies the condition:
\[
\|\nabla^2 f(x)\| \leq L_{1} \|\nabla f(x)\| + L_{0}.
\]
This definition can be extended to include:
\[
\|\nabla f(x) - \nabla f(y)\| \leq (L_{1}\|\nabla f(x)\| + L_{0}) \|x - y\|, \quad \text{for all } \|y - x\| \leq 1.
\]
This class of functions encompasses classic \(L\)-smooth functions when \(L_{1} = 0\) and \(L_{0}\) is the global Lipschitz constant of \(\nabla f\). By defining \(\upLambda(x, f) = L_{1} \|\nabla f(x)\| + L_{0}\) and assuming \(\nabla f\) is path differentiable, we obtain:
\[
L_{1} \frac{\nabla^2 f(x) J^{\nabla f}}{\|\nabla f(x)\|} \in D^{\upLambda}(x), \qquad \forall J^{\nabla f} \in D^{\nabla f}(x).
\]
This approach leverages first-order information to obtain higher-order information efficiently. It significantly reduces the computational effort by simplifying the estimation of \(\|\nabla f(x)\|\) to determining the upper bounds for constants \(L_{1}\) and \(L_{0}\).

\section{Theoretical analysis}
\label{sec:theory}

\subsection{Proof of Theorem \ref{thm:continuous-convergence}}
This subsection explores the convergence properties of the trajectory of the ISHD. We give a proof of Theorem \ref{thm:continuous-convergence} by constructing a Lyapunov function. This function includes the terms $\|x(t)-x_\star\|^2$, $\|t\beta(t)\nabla f(x(t))\|^2$, and $t\beta(t)\langle\nabla f(x(t)),x(t)-x_\star\rangle$, which are rarely seen in other Lyapunov function-based proofs. We add these terms to get a point-wise estimation of $\|\dot{x}(t)\|$ and $\|\nabla f(x(t))\|$.

\begin{proof}
    Consider the Lyapunov function
    \begin{equation}\label{eq:lyapunov}
        \begin{aligned}
            E(t)=&\delta(t)\left(f(x(t))-f_\star\right)+\frac{1}{2}\|\lambda(x(t)-x_{\star})+t(\dot{x}(t)+\kappa\beta(t)\nabla f(x(t)))\|^2\\
            &+\lambda(1-\kappa) t\beta(t)\langle\nabla f(x(t)),x(t)-x_\star\rangle+\frac{\kappa(1-\kappa)}{2}\|t\beta(t)\nabla f(x)\|^2\\
            &+\frac{\lambda(\alpha-1-\lambda)}{2}\|x(t)-x_{\star}\|^2.
        \end{aligned}
    \end{equation}
    We show that \eqref{eq:lyapunov} is a decreasing function in the remaining part of this proof. Let $u(t)=\gamma(t)-\kappa\dot{\beta}(t)-{\kappa\beta(t)}/{t}$. Notice that
    \[
        \begin{aligned}
            \frac{\mathrm{d}}{\mathrm{d}t}\Big(\lambda(x(t)-x_{\star})&+t(\dot{x}(t)+\kappa\beta(t)\nabla f(x(t)))\Big)\\
            &=(\lambda+1-\alpha)\dot{x}(t)-tu(t)\nabla f(x(t))-(1-\kappa)t\beta(t)\nabla^2 f(x(t))\dot{x}(t).
        \end{aligned}
    \]
    Plugging in the formula above and expanding the inner product give
    \[
        \begin{aligned}
            \frac{1}{2}\frac{\mathrm{d}}{\mathrm{d}t}\|&\lambda(x(t)-x_{\star})+t(\dot{x}(t)+\kappa\beta(t)\nabla f(x(t)))\|^2\\
            =&-\langle(\alpha-1-\lambda)\dot{x}(t)+(1-\kappa)t\beta(t)\nabla^2 f(x(t))\dot{x}(t)+tu(t)\nabla f(x(t)),\\
            &\lambda(x(t)-x_{\star})+t(\dot{x}(t)+\kappa\beta(t)\nabla f(x(t)))\rangle\\
            =&-\lambda(\alpha-1-\lambda)\langle\dot{x}(t),x(t)-x_\star\rangle-\lambda(1-\kappa)t\beta(t)\langle\nabla^2 f(x(t))\dot{x}(t),x(t)-x_\star\rangle\\
            &-\kappa(\alpha-1-\lambda)t\beta(t)\langle\nabla f(x(t)),\dot{x}(t)\rangle-\kappa t^2\beta(t)u(t)\|\nabla f(x(t))\|^2\\
            &-(1-\kappa)t^2\beta(t)\langle\nabla^2 f(x(t))\dot{x}(t),\dot{x}(t)\rangle-t^2u(t)\langle\nabla f(x(t)),\dot{x}(t)\rangle\\
            &-\lambda tu(t)\langle\nabla f(x(t)),x(t)-x_\star\rangle-(\alpha-1-\lambda)t\|\dot{x}(t)\|^2\\
            &-\kappa(1-\kappa)t^2\beta(t)^2\langle\nabla^2 f(x(t))\dot{x}(t),\nabla f(x(t))\rangle.
        \end{aligned}
    \]
    Direct calculation gives $\frac{\lambda(\alpha-1-\lambda)}{2}\frac{\mathrm{d}}{\mathrm{d}t}\|x(t)-x_{\star}\|^2=\lambda(\alpha-1-\lambda)\langle x(t)-x_\star,\dot{x}(t)\rangle$ and
    \[
        \begin{aligned}
            &\frac{\kappa(1-\kappa)}{2}\frac{\mathrm{d}}{\mathrm{d}t}\|t\beta(t)\nabla f(x)\|^2\\
            &=\kappa(1-\kappa)\Big(t^2\beta(t)^2\langle\nabla^2 f(x(t))\dot{x}(t),\nabla f(x(t))\rangle+t\beta(t)(\beta(t)+t\dot{\beta}(t))\|\nabla f(x(t))\|^2\Big).
        \end{aligned}
    \]
    We also have
    \[
        \begin{aligned}
            \lambda(1-\kappa)&\frac{\mathrm{d}}{\mathrm{d}t}\left(t\beta(t)\langle\nabla f(x(t)),x(t)-x_\star\rangle\right)\\
            =&\lambda(1-\kappa)t\beta(t)\langle\nabla f(x(t)),\dot{x}(t)\rangle+\lambda(1-\kappa) t\beta(t)\langle\nabla^2 f(x(t))\dot{x}(t),x(t)-x_\star\rangle\\
            &\qquad +\lambda(1-\kappa) (\beta(t)+t\dot{\beta}(t))\langle\nabla f(x(t)),x(t)-x_\star\rangle.
        \end{aligned}
    \]
    Using these results and combining similar terms yield
    \[
        \begin{aligned}
            \frac{\mathrm{d}}{\mathrm{d}t}E(t)
            =&\dot{\delta}(t)(f(x(t))-f_\star)-\lambda tw(t)\langle\nabla f(x(t)),x(t)-x_\star\rangle-(\alpha-1-\lambda)t\|\dot{x}(t)\|^2\\
            &+\Big(\delta(t)-\big(t^2u(t)+(\kappa (\alpha-1-\lambda)-\lambda(1-\kappa))t\beta(t)\big)\Big)\langle\nabla f(x(t)),\dot{x}(t)\rangle\\
            &-\kappa t^2\beta(t)w(t)\|\nabla f(x(t))\|^2-(1-\kappa)t^2\beta(t)\langle\nabla^2 f(x(t))\dot{x}(t),\dot{x}(t)\rangle.
        \end{aligned}
    \]
    According to \eqref{eq:converge-condition} and $\nabla^2 f(x(t))\succeq 0$, we have $\dot{E}(t)\leq 0$ and further
    \[
        \begin{aligned}
            \dot{E}(t)+(\lambda t &w(t)-\dot{\delta}(t))(f(x(t))-f_\star)+(\alpha-1-\lambda)t\|\dot{x}(t)\|^2\\
            &+(1-\kappa)t^2\beta(t)\langle\nabla^2 f(x(t))\dot{x}(t),\dot{x}(t)\rangle+\kappa t^2\beta(t)w(t)\|\nabla f(x(t))\|^2\leq 0.
        \end{aligned}
    \]
    Integrating the inequality above from $t_0$ to $t$ gives
    \begin{equation}\label{eq:final-argument}
        \begin{aligned}
            &E(t)+\int_{t_{0}}^{t}(\lambda s w(s)-\dot{\delta}(s))(f(x(s))-f_\star)\,\mathrm{d}s+\int_{t_{0}}^{t}\kappa s^2\beta(s)w(s)\|\nabla f(x(s))\|^2\,\mathrm{d}s\\
            &+\int_{t_{0}}^{t}(\alpha-1-\lambda)s\|\dot{x}(s)\|^2\,\mathrm{d}s+\int_{t_{0}}^{t}(1-\kappa)s^2\beta(s)\langle\nabla^2 f(x(s))\dot{x}(s),\dot{x}(s)\rangle\,\mathrm{d}s\leq E(t_0).
        \end{aligned}
    \end{equation}
    The nonnegativity of $E(t)$ and the arbitrary of $t$ imply \eqref{eq:value-integrable}, \eqref{eq:velocity-integrable}, \eqref{eq:grad-norm-integrable}, and \eqref{eq:hessian-velocity-integrable}.
    Plugging in the expression \eqref{eq:lyapunov} to \eqref{eq:final-argument} gives the first two terms of \eqref{eq:point-wise-estimate} and $\|x(t)-x_\star\|\leq\sqrt{2E(t_0)}/\sqrt{\lambda(\alpha-1-\lambda)}$. A direct calculation gives
    \[
        \|\dot{x}(t)\|\leq \frac{1}{t}\left(\sqrt{2E(t_0)}+\lambda\|x(t)-x_{0}\|\right)+\kappa \beta(t)\|\nabla f(x(t))\|\leq\mathcal{O}\left(\frac{1}{t}\right).
    \]
    Therefore, $x(t)$ is bounded and the third term in \eqref{eq:point-wise-estimate} are proved, which completes our proof.\qed
\end{proof}

Theorem \ref{thm:continuous-convergence} ensures that the stopping time $T$ is well-defined when $\beta$ is lower bounded. It also enlightens a way to analyze the sequence $\{x_{k}\}_{k=0}^{\infty}$ generated by \eqref{eq:first-order-discrete} with the help of the continuous-time trajectory, as we will show in the next subsection.

\subsection{Proof of Theorem \ref{thm:stable}: Estimating the summation of the local truncated errors}
To measure the deviation introduced at step $k$ and the cumulative deviation at step $k$ between the discretization $x_{k}$ and the continuous-time trajectory $x(t_{k})$, we introduce two quantities before proceeding to the analysis.

\begin{definition}[Local Truncated Error]
    Given the coefficients $\alpha,\beta$, and $\gamma$, the local truncated error of \eqref{eq:first-order-discrete} associated with a differentiable function $f$ using a stepsize $h$ is
    \begin{equation}\label{eq:trun-error}
        \varphi(t)=\begin{pmatrix}
            x(t+h)-x(t)\\
            v(t+h)-v(t)
        \end{pmatrix}
        -h
        \begin{pmatrix}
            v(t)-\beta(t)\nabla f(x(t))\\
            -\frac{\alpha}{t}v(t)+\left(\frac{\alpha}{t}\beta(t)+\dot{\beta}(t)-\gamma(t)\right)\nabla f(x(t))
        \end{pmatrix}.
    \end{equation}
\end{definition}

The local truncated error computes the difference between the exact solution of the differential equation at time $t+h$ and the approximate solution obtained using the explicit Euler method starting at time $t$. As pointed by its name, it only contains the information obtained at one step. To study how them accumulate in the proceeding iterations, we introduce the global truncated error, the difference between the exact solution of \eqref{eq:first-order} and the approximate solution obtained using \eqref{eq:first-order-discrete}.

\begin{definition}[Global Truncated Error]\label{def:global-error}
    Let $x(t_{0})=x_0,v(t_{0})=v_0$. Given the differentiable function $f$, the coefficients $\alpha,\beta$, $\gamma$, and the stepsize $h$, if $x(t)$ and $x_{k}$ are the solution of \eqref{eq:first-order} and \eqref{eq:first-order-discrete}, respectively, the vector
    \begin{equation}\label{eq:dis-error}
        e_{k}=
        \begin{pmatrix}
            r_{k}\\
            s_{k}
        \end{pmatrix}=
        \begin{pmatrix}
            x(t_{k})-x_{k}\\
            v(t_{k})-v_{k}
        \end{pmatrix}
    \end{equation}
    is the global truncated error at time $t_k$.
\end{definition}

Before proceeding, we present a lemma derived from the conditions \eqref{eq:stab-condition-1} and \eqref{eq:stab-condition-2}, which provides an estimate for the growth rate of $\beta(t)$.
\begin{lemma}[Growth rate of $\beta(t)$]
    \label{lem:growth-beta}
    Under the Assumption \ref{assump:differentiable} and given that the inequalities in \eqref{eq:stab-condition-1} and \eqref{eq:stab-condition-2} are satisfied, we have
    \begin{equation}
        \beta(t)\|\nabla^2 f(x(t))\|\leq 4/h.
    \end{equation}
\end{lemma}
\begin{proof}
    The non-negativity of $\alpha, \beta(\cdot)$, and $t$ ensures that $\gamma(t)-\dot{\beta}(t)-\alpha \beta(t)/t\leq \gamma(t)-\dot{\beta}(t)$. Therefore, combining \eqref{eq:stab-condition-1} and \eqref{eq:stab-condition-2} leads to
    \[
    \beta(t)\sqrt{\upLambda(f,X(t,\upXi,f))}\leq 2\sqrt{\gamma(t)-\dot{\beta}(t)}\leq 2\sqrt{\beta(t)/h}.
    \]
    Squaring both sides completes this proof.
    \qed
\end{proof}

We are now at the position to present the analysis using a two-step procedure. We begin our analysis of $\{x_{k}\}_{k=0}^{\infty}$ by investigating the local truncated error. In the first step, we provide an enhanced estimation of the sequence $\{\|\varphi(t_{k})\|\}_{k=0}^{\infty}$ using Theorem \ref{thm:continuous-convergence}. A byproduct of this estimation shows that under some mild conditions the local truncated error vanishing with $\mathcal{O}(1/t)$ rate. In the second step, we show that the global truncated error inherits the convergence properties of the local truncated error under some mild conditions.

\begin{lemma}[Estimating the Summation of the Local Truncated Error]
    \label{lem:varphi-rate}
    Assume the conditions \eqref{eq:trun-error-growth-condition-1}, \eqref{eq:stab-condition-1}, and \eqref{eq:stab-condition-2} hold. Using the notations in Theorem \ref{thm:continuous-convergence} and the energy function defined in \eqref{eq:lyapunov}, we set
    \begin{align}
        M_{1}&=\max\{1+(\alpha+1)/t_0,C_{2}/h+(1+\alpha/t_{0})(1/h+C_{1}),\alpha/t_0+1/h+1\},\label{eq:M1}\\
        M_{2}&=\max\left\{\sqrt{\frac{E(t_0)}{2(\alpha-1)(\alpha-1-\lambda)t_0}},\sqrt{\frac{C_3E(t_0)}{\kappa(2\alpha-1)}},\sqrt{\frac{4E(t_0)}{(1-\kappa)(2\alpha-1)h}}\right\}.\label{eq:M2}
    \end{align}
    Let $t_k=t_0+(k-1)h$ and $M_3=M_1M_2$. Then,
    \begin{equation}\label{eq:estimate-varphi}
        \|\varphi(t)\|\leq o(1/t)\quad\text{and}\quad
        \sum_{k=0}^{n}t_{k}^{\alpha}\|\varphi(t_{k})\|\leq M_3t_{n}^{\alpha-1/2}.
    \end{equation}
\end{lemma}
\begin{proof}
    In order to get \eqref{eq:estimate-varphi}, we first bound the local truncated error as follows:
    \begin{equation}\label{eq:bound-varphi}
        \|\varphi(t)\|\leq\int_{t}^{t+h}\int_{t}^{s}\left(\left\|\ddot{x}(\tau)\right\|+\left\|\ddot{v}(\tau)\right\|\right)\,\mathrm{d}\tau\,\mathrm{d}s\leq h\int_{t}^{t+h}\left(\left\|\ddot{x}(\tau)\right\|+\left\|\ddot{v}(\tau)\right\|\right)\,\mathrm{d}\tau.
    \end{equation}
    Using \eqref{eq:ISHD}, the second term in \eqref{eq:stab-condition-1}, and the condition \eqref{eq:trun-error-growth-condition-1}, we have the following estimation:
    \begin{equation}\label{eq:estimation-ddot-x}
        \begin{aligned}
            \int_{t}^{t+h}\left\|\ddot{x}(\tau)\right\|\,\mathrm{d}\tau
            \leq&\frac{\alpha}{t}\int_{t}^{t+h}\|\dot{x}(\tau)\|\,\mathrm{d}\tau+\int_{t}^{t+h}\beta(\tau)\|\nabla^2 f(x(\tau))\dot{x}(\tau)\|\,\mathrm{d}\tau\\
            &+\left(\frac{1}{h}+C_{1}\right)\int_{t}^{t+h}\|\beta(\tau)\nabla f(x(\tau))\|\,\mathrm{d}\tau.
        \end{aligned}
    \end{equation}
    Invoking \eqref{eq:first-order-discrete}, we obtain
    \begin{equation}\label{eq:ddot-v}
        \ddot{v}(t)=\alpha\dot{x}(t)/t^2-\alpha\ddot{x}(t)/t-(\gamma(t)-\dot{\beta}(t))\nabla^2 f(x(t))\dot{x}(t)-(\dot{\gamma}(t)-\ddot{\beta}(t))\nabla f(x(t)).
    \end{equation}
    Using \eqref{eq:estimation-ddot-x} and \eqref{eq:ddot-v} yields
    \begin{equation}\label{eq:estimation-ddot-v}
        \begin{aligned}
            \int_{t}^{t+h}\left\|\ddot{v}(\tau)\right\|\,\mathrm{d}\tau
            \leq&\int_{t}^{t+h}\Big(\frac{\alpha}{\tau}\|\ddot{x}(\tau)\|+\frac{\alpha}{\tau^2}\|\dot{x}(\tau)\|+|\dot{\gamma}(\tau)-\ddot{\beta}(\tau)|\|\nabla f(x(\tau))\|\\
            &+(\gamma(\tau)-\dot{\beta}(\tau))\|\nabla^2 f(x(\tau))\dot{x}(\tau)\|\Big)\,\mathrm{d}\tau\\
            \leq&\left(\frac{C_2}{h}+\frac{\alpha}{t}\left(\frac{1}{h}+C_{1}\right)\right)\int_{t}^{t+h}\|\beta(\tau)\nabla f(x(\tau))\|\,\mathrm{d}\tau\\
            &+\left(\frac{\alpha}{t}+\frac{1}{h}\right)\int_{t}^{t+h}\beta(\tau)\|\nabla^2 f(x(\tau))\dot{x}(\tau)\|\,\mathrm{d}\tau\\
            &+\frac{\alpha(\alpha+1)}{t^2}\int_{t}^{t+h}\|\dot{x}(\tau)\|\,\mathrm{d}\tau.
        \end{aligned}
    \end{equation}
    Combining \eqref{eq:M1}, \eqref{eq:bound-varphi}, \eqref{eq:estimation-ddot-x}, and \eqref{eq:estimation-ddot-v} gives
    \begin{equation}
        \|\varphi(t)\|\leq M_{1}\int_{t}^{t+h}\left(\frac{\alpha}{t}\|\dot{x}(\tau)\|+\|\beta(\tau)\nabla f(x(\tau))\|+\beta(\tau)\|\nabla^2 f(x(\tau))\dot{x}(\tau)\|\right)\,\mathrm{d}\tau.
    \end{equation}
    The inequalities \eqref{eq:grad-norm-integrable} and \eqref{eq:hessian-velocity-integrable} imply
    \[
        \sqrt{\beta(t)w(t)}\|\nabla f(x(t))\|=o\left(\frac{1}{t}\right),\quad \|\sqrt{\beta(t)\nabla^2 f(x(t))}\dot{x}(t)\|=o\left(\frac{1}{t}\right).
    \]
    Using Lemma \ref{lem:growth-beta}, the third inequality in \eqref{eq:point-wise-estimate}, and the third inequality in \eqref{eq:trun-error-growth-condition-1}, we get $\|\varphi(t)\|\leq o(1/t)$, which is the first inequality in \eqref{eq:estimate-varphi}.
    
    For the second inequality in \eqref{eq:estimate-varphi}, we transform the summation in \eqref{eq:estimate-varphi} to integral by
    \[
        \begin{aligned}
            &\sum_{k=0}^{n}t_{k}^\alpha\|\varphi(t_{k})\|\\
            &\leq M_{1}\sum_{k=0}^{n}t_{k}^{\alpha}\int_{t_k}^{t_{k+1}}\Big(\frac{1}{t_k}\|\dot{x}(\tau)\|+\|\beta(\tau)\nabla f(x(\tau))\|+\beta(\tau)\|\nabla^2 f(x(\tau))\dot{x}(\tau)\|\Big)\,\mathrm{d}\tau\\
            &\leq M_{1}\int_{t_0}^{t_{n+1}}\Big(\tau^{\alpha-1}\|\dot{x}(\tau)\|+\tau^{\alpha}\|\beta(\tau)\nabla f(x(\tau))\|+\tau^{\alpha}\beta(\tau)\|\nabla^2 f(x(\tau))\dot{x}(\tau)\|\Big)\,\mathrm{d}\tau.
        \end{aligned}
    \]
    Applying the Cauchy inequality with \eqref{eq:final-argument} gives
    \begin{equation}\label{eq:cauchy-velocity}
        \begin{aligned}
            \int_{t_0}^{t_{n+1}}&\tau^{\alpha-1}\|\dot{x}(\tau)\|\,\mathrm{d}\tau\\
            &\leq\sqrt{\int_{t_0}^{t_{n+1}}\tau^{2\alpha-3}\,\mathrm{d}\tau\int_{t_0}^{t_{n+1}}\tau\|\dot{x}(\tau)\|^2\,\mathrm{d}\tau}
            \leq\sqrt{\frac{E(t_0)}{2(\alpha-1)(\alpha-1-\lambda)}}t_{n+1}^{\alpha-1}.
        \end{aligned}
    \end{equation}
    Lemma \ref{lem:growth-beta} and the equation \eqref{eq:final-argument} imply
    \begin{equation}\label{eq:cauchy-grad-norm}
        \begin{aligned}
            \int_{t_0}^{t_{n+1}}&\tau^{\alpha}\|\beta(\tau)\nabla f(x(\tau))\|\,\mathrm{d}\tau\\
            &\leq\sqrt{\int_{t_0}^{t_{n+1}}\tau^{2\alpha-2}\,\mathrm{d}\tau\int_{t_0}^{t_{n+1}}\tau^2\|\beta(\tau)\nabla f(x(\tau))\|^2\,\mathrm{d}\tau}
            \leq\sqrt{\frac{C_3E(t_0)}{\kappa(2\alpha-1)}}t_{n+1}^{\alpha-1/2}.
        \end{aligned}
    \end{equation}
    Finally, using Lemma \ref{lem:growth-beta}, $\beta(t)\leq C_3w(t)$, and \eqref{eq:final-argument}, we have
    \begin{equation}\label{eq:cauchy-hessian-velocity}
        \begin{aligned}
            \int_{t_0}^{t_{k+1}}\tau^{\alpha}\beta(\tau)&\|\nabla^2 f(x(\tau))\dot{x}(\tau)\|\,\mathrm{d}\tau\\
            \leq&\sqrt{\frac{4}{h}\int_{t_0}^{t_{n+1}}\tau^{2\alpha-2}\,\mathrm{d}\tau\int_{t_0}^{t_{n+1}}\tau^2\beta(\tau)\|\dot{x}(\tau)\|^2_{\nabla^2 f(x(\tau))}\,\mathrm{d}\tau}\\
            \leq&\sqrt{\frac{4E(t_0)}{(1-\kappa)(2\alpha-1)h}}t_{n+1}^{\alpha-\frac{1}{2}}.
        \end{aligned}
    \end{equation}
    Combining \eqref{eq:M2}, \eqref{eq:cauchy-velocity}, \eqref{eq:cauchy-grad-norm}, and \eqref{eq:cauchy-hessian-velocity} yields \eqref{eq:estimate-varphi}.\qed
\end{proof}

\subsection{Proof of Theorem \ref{thm:stable}}
This subsection gives the convergence analysis of Algorithm \ref{algo:EIGAC}. As the discretization scheme \eqref{eq:first-order-discrete} is repeatedly applied, the local truncation errors $\{\|\varphi(t_{k})\|\}_{k=0}^{\infty}$ from each step accumulate and interact in a complex manner, causing the global error to grow. The contraction factor helps characterize how the local errors propagate and interact.

\begin{definition}[Contraction Factor]
    Let $G\in\mathbb{S}^{n}_{+}$ be an approximation matrix of $\nabla^2 f(x(t))$. Given the coefficients $\alpha,\beta$, and $\gamma$, the contraction factor for scheme \eqref{eq:first-order-discrete} associated with  at time $t$ is
    \begin{equation}\label{eq:define-contraction}
        \rho(t,G):=\left\|\begin{pmatrix}
            I-h\beta(t)G&hI\\
            (\alpha\beta(t)/t+\dot{\beta}(t)-\gamma(t))G&(1-\alpha h/t)I
        \end{pmatrix}\right\|.
    \end{equation}
\end{definition}

The contraction factor establishes a relationship between the amplification of global errors and the magnitude of local errors through a process of linearization. A smaller contraction factor leads to a more gradual increase in global errors, as it results in a stronger damping of local errors during their propagation. This argument is demonstrated in sec. \ref{sec:effect-rho}. Following this, we present an explicit formula for calculating the contraction factor, grounded in the following lemma.
\begin{lemma}[Root Location]\label{lem:root-location}
Given $\mu, \nu \in \mathbb{R}$ and $0<\varrho<1$, the $\varrho$-strong root condition is formulated for the roots of the equation 
\begin{equation}\label{eq:second-order-equation}
    r^2+\mu r+\nu=0,
\end{equation}
to lie within $|r|\leq \varrho$. The condition is expressed as:
\begin{equation}\label{eq:strong-root}
    \nu\leq \varrho^2,\quad \varrho\mu-\varrho^2\leq \nu,\quad -\varrho\mu-\varrho^2\leq \nu.
\end{equation}
\end{lemma}
\begin{proof}
    Based on the properties of Schur polynomials \cite[Lemma 1.5]{lambert1991numerical}, the necessary and sufficient condition for the roots (which can be complex) of \eqref{eq:second-order-equation} to lie within the unit circle is given by 
    \begin{equation}\label{eq:root-condition}
        \nu\leq 1, \quad \nu\geq\mu-1,\quad \nu\geq-\mu-1.
    \end{equation}
    Ensuring that the roots lie within the circle $|r|\leq \varrho$ is equivalent to the roots of $\Tilde{r}^2+\frac{\mu}{\varrho}\Tilde{r}+\frac{\nu}{\varrho^2}=0$ being within the unit circle, where $\Tilde{r}=\frac{r}{\varrho}$. Applying \eqref{eq:root-condition} to this condition yields
    \begin{equation}
        \nu\leq \varrho^2,\quad \varrho\mu-\varrho^2\leq \nu,\quad -\varrho\mu-\varrho^2\leq \nu,
    \end{equation}
    which is the desired result.
    \qed
\end{proof}

As an application of Lemma \ref{lem:root-location}, we give the explicit formula for the contraction factor.
\begin{proposition}
\label{prop:maximal-rho}
    Given the matrix $G\in \mathbb{S}^{n}_{+}$, coefficients $\alpha,\beta(\cdot)$, and $\gamma(\cdot)$ in \eqref{eq:define-contraction}, we denote the eigenvalues of $G$ as $\lambda_{1}\geq\ldots\geq \lambda_{n}\geq 0$ and introduce
    \begin{equation}
        A(t,L)=h^2(\gamma(t)-\dot{\beta}(t))L,\; B(t,L)=\frac{h}{2}(\beta(t)L+\alpha/t), \; C(t,L)=B(t,L)^2-A(t,L).
    \end{equation}
    Assuming the condition \eqref{eq:stab-condition-1} holds, the contraction factor is
    \begin{equation}\label{eq:contraction-factor}
    \begin{aligned}
    \rho(t,G)=
        \max\{\sqrt{[1-2B(t,\lambda_{n})+A(t,\lambda_{n})]_+},&B(t,\lambda_{1})-1+\sqrt{[C(t,\lambda_{1})]_+},\\
        &1-B(t,\lambda_{1})+\sqrt{[C(t,\lambda_{1})]_+}\}.
    \end{aligned}
    \end{equation}
\end{proposition}
\begin{proof}
    Since $G$ is symmetric semi-positive definite, it has the decomposition $G=Q\upLambda Q^\top$, where $Q$ is orthogonal and $\upLambda$ is a non-negative diagonal matrix. Using a block diagonal matrix $\operatorname{Diag}(Q,Q)$ to perform the similar transform to $G$, it is sufficient to consider $\operatorname{Diag}(G_{1},\ldots,G_{n})$ with
    \begin{equation}\label{eq:Gi}
        G_{i}=\begin{pmatrix}
            1-h\beta(t)\lambda_{i}&h\\
            (\alpha\beta(t)/t+\dot{\beta}(t)-\gamma(t))h\lambda_{i}&1-\alpha h/t
        \end{pmatrix}.
    \end{equation}
    This fact can be derived by switching the columns and rows to form a block diagonal matrix with $2\times 2$ sub-matrices.
    
    The characteristic equation of $G_{i}$ is
    \begin{equation}\label{eq:lambda-i}
        \lambda^2-\left(2-2B(t,\lambda_{i})\right)\lambda+1-2B(t,\lambda_{i})+A(t,\lambda_{i})=0.
    \end{equation}
    Denote the spectral radius of \eqref{eq:lambda-i} as $\rho_{i}$. Using Lemma \ref{lem:root-location} and elementary algebra, $\rho_{i}$ must satisfies
    \begin{align}
        &1-2B(t,\lambda_{i})+A(t,\lambda_{i})\leq \rho_{i}^2,\label{eq:strong-1}\\
        &C(t,\lambda_{i})\leq\left(1+\rho_{i}-B(t,\lambda_{i})\right)^2,\label{eq:strong-2}\\
        &C(t,\lambda_{i})\leq\left(1-\rho_{i}-B(t,\lambda_{i})\right)^2.\label{eq:strong-3}
    \end{align}
    Notice that
    \begin{equation}\label{eq:relation}
        1-2B(t,\lambda_{i})+A(t,\lambda_{i})+C(t,\lambda_{i})=(1-B(t,\lambda_{i}))^2.
    \end{equation}
    When $C(t,\lambda_{i})<0$, only \eqref{eq:strong-1} takes effect. We have $1-2B(t,\lambda_{i})+A(t,\lambda_{i})>0$ and 
    \begin{equation}\label{eq:c-leq-0}
        \sqrt{1-2B(t,\lambda_{i})+A(t,\lambda_{i})}\geq |1-B(t,\lambda_{i})|\geq \max\{1-B(t,\lambda_{i}),B(t,\lambda_{i})-1\}.
    \end{equation}
    When $C(t,\lambda_{i})\geq 0$, we have
    \begin{equation}\label{eq:c-geq-0}
        \sqrt{1-2B(t,\lambda_{i})+A(t,\lambda_{i})}\leq |1-B(t,\lambda_{i})|\leq \max\{1-B(t,\lambda_{i}),B(t,\lambda_{i})-1\}.
    \end{equation}
    This inequality excludes the possibility of $\rho_{i}\leq 1-B(t,L)-\sqrt{[C(t,L)]_+}$ and $\rho_{i}\leq B(t,L)-1-\sqrt{[C(t,L)]_+}$. Combining the equations \eqref{eq:c-leq-0} and \eqref{eq:c-geq-0}, the spectral radius of $G_{i}$ is
    \begin{equation}\label{eq:rho-i}
        \begin{aligned}
            \rho_{i}=
                \max\{\sqrt{[1-2B(t,\lambda_{i})+A(t,\lambda_{i})]_+},&B(t,\lambda_{i})-1+\sqrt{[C(t,\lambda_{i})]_+},\\
                &1-B(t,\lambda_{i})+\sqrt{[C(t,\lambda_{i})]_+}\}.
        \end{aligned}
    \end{equation}

    Since $\gamma(t)-\dot{\beta}(t)\leq\beta(t)/h$, it is easy to verify that $\sqrt{[1-2B(t,L)+A(t,L)]_+}$ is an decreasing function with respect to $L$. Meanwhile, the derivative of $B(t,L)-\sqrt{C(t,L)}$ with respect to $L$ is
    \[
        \begin{aligned}
            &\frac{h}{2}\bigg(\beta(t)-\frac{2\beta(t)(\beta(t)L+{\alpha}/{t})-4(\gamma(t)-\dot{\beta}(t))}{2\sqrt{\frac{4}{h^2}C(t,L)}}\bigg)\\
            &=\frac{h\beta(t)}{2\sqrt{\frac{4}{h^2}C(t,L)}}\bigg(2\sqrt{(\beta(t)L+\alpha/t)^2-4(\gamma(t)-\dot{\beta}(t)) L}\\
            &\qquad\qquad\qquad\qquad\qquad-\left(2(\beta(t)L+{\alpha}/{t})-4(\gamma(t)-\dot{\beta}(t))/\beta(t)\right)\bigg).
        \end{aligned}
    \]
    Elementary calculation gives
    \[
        \begin{aligned}
            4(&\beta(t)L+\alpha/t)^2-16(\gamma(t)-\dot{\beta}(t)) L-4\left((\beta(t)L+{\alpha}/{t})-2(\gamma(t)-\dot{\beta}(t))/\beta(t)\right)^2\\
            &=-16(\gamma(t)-\dot{\beta}(t)) L+16\frac{(\gamma(t)-\dot{\beta}(t))}{\beta(t)}(\beta(t)L+{\alpha}/{t})-16\frac{(\gamma(t)-\dot{\beta}(t))^2}{\beta(t)^2}\\
            &=16\frac{(\gamma(t)-\dot{\beta}(t))}{\beta(t)^2}\left(\frac{\alpha}{t}\beta(t)-(\gamma(t)-\dot{\beta}(t))\right)\leq 0.
        \end{aligned}
    \]
    We have \(\frac{\mathrm{d}}{\mathrm{d}L}\left(B(t,L) - \sqrt{C(t,L)}\right) \leq 0\). Considering the following relationships:
    \[
    \begin{aligned}
        &B(t,L) - 1 + \sqrt{[C(t,L)]_+} = 2B(t,L) - 1 - (B(t,L) - \sqrt{C(t,L)}), \\
        &1 - B(t,L) + \sqrt{[C(t,L)]_+} = 1 - (B(t,L) - \sqrt{C(t,L)}),
    \end{aligned}
    \]
    we observe that the functions \(B(t,L) - 1 + \sqrt{[C(t,L)]_+}\) and \(1 - B(t,L) + \sqrt{[C(t,L)]_+}\) are increasing. By using \(\rho(t,G) = \max_{i} \rho_{i}\) and the monotonicity of each component in \eqref{eq:rho-i}, we derive the formula \eqref{eq:contraction-factor}.
    \qed
\end{proof}

A useful estimation of the contraction factor can be obtained through a more detailed investigation of Proposition \ref{prop:maximal-rho}. We use the same notations as Proposition \ref{prop:maximal-rho} in the following corollary.

\begin{corollary}
    \label{coro:bound-rho}
    Given the condition \eqref{eq:stab-condition-1}, if the matrix \(G \in \mathbb{S}^{n}_{+}\) satisfies
    \begin{equation}
        \label{eq:bound-rho}
        \beta(t)\sqrt{\|G\|} \leq \sqrt{\gamma(t) - \dot{\beta}(t)} + \sqrt{\gamma(t) - \dot{\beta}(t) - \frac{\alpha}{t}\beta(t)},
    \end{equation}
    then \(\rho(t, G) \leq 1 - {\alpha h}/{(2t)}\).
\end{corollary}

\begin{proof}
    The roots of \(C(t, L) = 0\) are
    \[
        \begin{aligned}
            L_{\pm} &= \frac{-\left(\frac{\alpha}{t} \beta(t) - 2(\gamma(t) - \dot{\beta}(t))\right) \pm 2\sqrt{(\gamma(t) - \dot{\beta}(t))^2 - \frac{\alpha}{t} \beta(t)(\gamma(t) - \dot{\beta}(t))}}{\beta(t)^2}\\
            &= \left(\frac{\sqrt{\gamma(t) - \dot{\beta}(t)} \pm \sqrt{\gamma(t) - \dot{\beta}(t) - \frac{\alpha}{t} \beta(t)}}{\beta(t)}\right)^2.
        \end{aligned}
    \]
    When \(L_{-} \leq \lambda_{1} \leq L_{+}\), we know \(C(t, \lambda_{1}) \leq 0\). Hence, the formula \eqref{eq:c-leq-0} implies
    \[
        \begin{aligned}
            1 - B(t, \lambda_{1}) + \sqrt{[C(t, \lambda_{1})]_+} &\leq \sqrt{1 - 2B(t, \lambda_{1}) + A(t, \lambda_{1})}\\
            &\leq \sqrt{1 - 2B(t, \lambda_{n}) + A(t, \lambda_{n})}.
        \end{aligned}
    \]
    When \(L \leq L_{-}\), we have \(C(t, L) \geq 0\). The monotonicity of \(B(t, L) - \sqrt{C(t, L)}\) implies
    \[
        \begin{aligned}
            B(t, L) - 1 + \sqrt{C(t, L)} &\leq B(t, L_{-}) - 1 + \sqrt{C(t, L_{-})}\\
            &= \sqrt{1 - 2B(t, L_{-}) + A(t, L_{-})}\\
            &\leq \sqrt{1 - 2B(t, \lambda_{n}) + A(t, \lambda_{n})},
        \end{aligned}
    \]
    where the last inequality comes from \(\lambda_{n} \leq \lambda_{1} \leq L_{-}\) and the monotonicity of \(\sqrt{1 - 2B(t, L) + A(t, L)}\).

    Combining these results, we know \(\lambda_{1} \leq L_{+}\) ensures that
    \[
        \rho(t, G) \leq \sqrt{1 - 2B(t, \lambda_{n}) + A(t, \lambda_{n})} \leq 1 - \frac{\alpha h}{2t}.
    \]
    The second inequality comes from $0\leq \lambda_{n}$, \(\gamma(t) - \dot{\beta}(t) \leq {\beta(t)}/{h}\), and \(\sqrt{1 - {\alpha h}/{t}} \leq 1 - {\alpha h}/{2t}\). Simplifying \(\lambda_{1} \leq L_{+}\) gives the equation \eqref{eq:bound-rho}.
    \qed
\end{proof}

Next, we give the proof of Theorem \ref{thm:stable} by analyzing the propagation of the global truncated error.

\begin{proof}
    To investigate the global truncated error, we consider the relationship between it and the local truncated errors. Subtracting \eqref{eq:trun-error} with \eqref{eq:first-order-discrete} at time $t_{k}$ and using the definition \eqref{eq:dis-error} yield
    \begin{equation}\label{eq:error-prop}
        \begin{pmatrix}
            r_{k+1}\\
            s_{k+1}
        \end{pmatrix}
        =
        \begin{pmatrix}
            I-h\beta(t_{k})G&hI\\
            (\alpha\beta(t_{k})/t_{k}+\dot{\beta}(t_{k})-\gamma(t_{k}))G&(1-\alpha h/t_{k})I
        \end{pmatrix}
        \begin{pmatrix}
            r_{k}\\
            s_{k}
        \end{pmatrix}
        +
        h\varphi(t_{k}),
    \end{equation}
    where $G=\int_{0}^{1}\nabla^2 f(x(t_{k})+\tau r_{k})\,\mathrm{d}\tau$. Our proof is based on the error propagation equation \eqref{eq:error-prop}. Denote $\rho(t_{k})=\rho(t_{k},G)$ with $G=\int_{0}^{1}\nabla^2 f(x(t_{k})+\tau r_{k})\,\mathrm{d}\tau$. We enhance the theorem and prove the following conclusions for each $k$:
    \begin{equation}\label{eq:converge-rate}
        \rho(t_k) \leq 1 - \frac{\alpha h}{2t_k}, \quad \|e_k\| \leq \frac{M_3}{\sqrt{t_k}}, \quad \text{and} \quad f(x_k) - f_{\star} \leq \mathcal{O}\left(\frac{1}{k}\right).
    \end{equation}
    The constant $M_{3}$ come from Lemma \ref{lem:varphi-rate}. To facilitate our presentation, we designate the first inequality in \eqref{eq:converge-rate} with index $k$ as $H_0(k)$, and the second as $H_{1}(k)$. We employ a mathematical induction on $k$ to demonstrate the theorem.

    We begin by confirming the base case when $k = 0$. The restriction on $\alpha,\beta(t_0)$, and $\gamma(t_0)$ ensures $\rho(t_0)\leq 1-\alpha h/(2t_{0})$. Since, $x(t_{0})=x_{0},v(t_{0})=v_{0}$, we have $\|e_{0}\|=0$. These results guarantee that both propositions $H_{0}(0)$ and $H_{1}(0)$ are satisfied.

    Assume as our induction hypothesis that both $H_{0}(l)$ and $H_{1}(l)$ are valid for all indices less than or equal to $k$. We will now prove that they remain true for the index $k+1$. For $k\leq K$, a useful estimation of $\prod_{l=k}^{K}\rho(t_{l})$ is
    \begin{equation}
        \begin{aligned}
            \prod_{l=k}^{K}\rho(t_{l})=&\exp\left(\sum_{l=k}^{K}\ln\left(\rho(t_{l})-1+1\right)\right)
            \leq\exp\left(\sum_{l=k}^{K}(\rho(t_{l})-1)\right)\\
            \leq&\exp\left(-\sum_{l=k}^{K}\frac{\alpha h}{2t_{l}}\right)
            \leq\exp\left(-\frac{\alpha}{2}\int_{t_{k}}^{t_{K+1}}\frac{1}{t}\,\mathrm{d}t\right)=\left(\frac{t_{k}}{t_{K+1}}\right)^{\alpha/2}.
        \end{aligned}
    \end{equation}
    We first prove that $H_{1}(k+1)$ holds. Applying \eqref{eq:error-prop} successively, we have
    \[
        \begin{aligned}
            \|e_{n+1}\|\leq&\left(1-\frac{\alpha h}{2t_{k}}\right)\|e_{k}\|+\|\varphi(t_{k})\|\\
            \leq&\prod_{k=0}^{n}(1-\rho(t_{k}))\|e_0\|+\|\varphi(t_{n})\|+\sum_{k=0}^{n-1}\prod_{j=k+1}^{n}(1-\rho(t_{j}))\|\varphi(t_{k})\|\\
            \leq&\sum_{k=0}^{n}\left(\frac{t_{k+1}}{t_{n+1}}\right)^{\alpha/2} \|\varphi(t_{k})\|\leq M_{3}\frac{1}{\sqrt{t_{n+1}}}.
        \end{aligned}
    \]
    The last inequality are derived from Lemma \ref{lem:varphi-rate}.

    For $H_{0}(k+1)$, we note that the condition \eqref{eq:stab-condition-2} ensures that
    \[
        \begin{aligned}
            \beta(t)\sqrt{\left\|\int_{0}^{1}\nabla^2 f(x(t_{k})+\tau r_{k})\,\mathrm{d}\tau\right\|}
            &\leq\beta(t)\sqrt{\int_{0}^{1}\upLambda(X(t_{k},\upXi,f)+\tau \bar{x}(t_{k},\upXi,f),f)\,\mathrm{d}\tau}\\
            &\leq\sqrt{\gamma(t)-\dot{\beta}(t)}+\sqrt{\gamma(t)-\dot{\beta}(t)-\frac{\alpha}{t}\beta(t)}.
        \end{aligned}
    \]
    Using Corollary \ref{coro:bound-rho}, we get $\rho(t_{n+1})\leq 1/({2t_{n+1}})$.
    
    Combining these results, we know $H_{0}(k)$ and $H_{1}(k)$ hold for all $k$. The function value minimization rate is
    \begin{equation}
        \begin{aligned}
            f(x_{k})-f_\star\leq&|f(x_{k})-f(x(t_{k}))|+|f(x(t_{k}))-f_\star|\\
            \leq&\|\nabla f(x(t_{k}))\|\|e_{k}\|+\frac{1}{2}\left\|\int_{0}^{1}\nabla^2 f(x(t_k)+\tau e_{k})\,\mathrm{d}\tau\right\|\|e_{k}\|^2+\frac{E(t_{0})}{t_{k}^2w(t_{k})},
        \end{aligned}
    \end{equation}
    which means the convergence rate is at least $\mathcal{O}(1/k)$.
    \qed
\end{proof}
Using the same assumptions and notations as in Theorem \ref{thm:stable}, we present a simplified version of the stability condition under the Lipschitz continuity of the Hessian matrix.

\begin{corollary}
    Suppose the Hessian matrix satisfies the Lipschitz continuity condition
    \[
    \|\nabla^2 f(x) - \nabla^2 f(y)\| \leq L_H \|x - y\|.
    \]
    Using the same notations and conditions as in Theorem \ref{thm:stable}, but replacing the condition \eqref{eq:stab-condition-2} with the simplified condition \eqref{eq:stab-condition-2-simplified}, we obtain that \(f(x_k) - f_{\star} \leq \mathcal{O}(1/k)\).
\end{corollary}

\begin{proof}
    We use the same arguments as in the proof of Theorem \ref{thm:stable}. The parts besides \(H_{0}(k+1)\) can be proven verbosely as Theorem \ref{thm:stable}. For \(H_{0}(k+1)\), assuming \(H_{1}(k)\) holds, the Hessian continuity and the condition \eqref{eq:stab-condition-2-simplified} implies
    \[
        \begin{aligned}
            \beta(t) &\sqrt{\left\|\int_{0}^{1}\nabla^2 f(x(t_k) + \tau r_k)\,\mathrm{d}\tau\right\|} \\
            &\leq \beta(t) \sqrt{\nabla^2 f(x(t_k)) + \frac{L_T}{2}\|r_k\|} \leq \beta(t) \sqrt{\upLambda(X(t_k), \upXi, f) + \frac{L_T M_3}{2 \sqrt{t_k}}} \\
            &\leq \sqrt{\gamma(t) - \dot{\beta}(t)} + \sqrt{\gamma(t) - \dot{\beta}(t) - \frac{\alpha}{t}\beta(t)}.
        \end{aligned}
    \]
    The second inequality follows from \(H_{1}(k)\). Corollary \ref{coro:bound-rho} establishes \(H_{0}(k+1)\).
    \qed
\end{proof}

The convergence rate given in Theorem \ref{thm:stable} is only $\mathcal{O}(1/k)$ rather than the accelerated rate $\mathcal{O}(1/k^2)$. The reason is that  our technique in the proof makes the discretization error accumulates and eventually gives a slow convergence rate. However, the flexibility of the coefficients make it suitable for deriving a large range of optimization methods. We also show that the empirical performance of \eqref{eq:first-order-discrete} is comparable to and can even be better than the discretized scheme with $\mathcal{O}(1/k^2)$ rate in theory after some tuning of the coefficients.

\subsection{Convergence analysis of Algorithm \ref{algo:SEPM}}

Before presenting the main result, we first give several useful definitions.

\begin{definition}[Clarke directional derivative]
    \label{def:clarke-directional}
    Let $\varkappa$ be Lipschitz continuous near $\bar{\theta}$. The Clarke directional derivative of $\varkappa$ at $\bar{\theta}$ along a nonzero vector $\vartheta$ is given by
    \begin{equation}
        \varkappa^{\circ}(\bar{\theta} ; \vartheta) \triangleq \limsup _{\substack{\theta \rightarrow \bar{\theta} \\ \tau \downarrow 0}} \frac{\varkappa(\theta+\tau \vartheta)-\varkappa(\theta)}{\tau}.
    \end{equation}
\end{definition}

Based on Definition \ref{def:clarke-directional}, we define the Clarke subdifferential.
\begin{definition}[Clarke subdifferential]
\label{def:clarke-subdiff}
    For a function $\varkappa$ that is Lipschitz continuous near $\theta$, the Clarke subdifferential of $\varkappa$ at $\theta$ is given by
    \[
    \partial \varkappa(\theta) \triangleq\left\{a \in \mathbb{R}^{d_{\theta}}: \varkappa^{\circ}({\theta} ; \vartheta) \geq a^{\top} \vartheta,\quad \forall \vartheta \in \mathbb{R}^{d_{\theta}}\right\}.
    \]
\end{definition}

Another definition of the Clarke subdifferential is
\[
\partial \varkappa(\theta)=\operatorname{co}\left\{\vartheta \in \mathbb{R}^{d_{\theta}}: \exists\left\{\theta_n\right\}_{n \in \mathbb{N}}\subset \operatorname{dom}\varkappa, \text { s.t. }\left(\theta_n, \nabla \varkappa\left(\theta_n\right)\right) \rightarrow(\theta, \vartheta)\right\}.
\]
For the proof of the equivalence, one may refer to \cite[Theorem 2.5.1]{clarke1990nonsmooth}. Using the Clarke subdifferential, we define the Clarke stationarity. Besides, we also define the $D$-stationarity using conservative gradient.

\begin{definition}[Clarke stationary point]
    $\theta$ is a Clarke stationary point of $\varkappa$ if $0\in \partial\varkappa(\theta)$.
\end{definition}
\begin{definition}[$D$-stationary point]
    $\theta$ is a $D$-stationary point of $\varkappa$ if $0\in D^\varkappa(\theta)$.
\end{definition}

We apply Definition \ref{def:clarke-directional} and get the following criterion for Clarke stationary point based on Clarke directional derivative.

\begin{proposition}[Criterion of Clarke stationary point]
    \label{prop:criteria-clarke}
    $\theta$ is a Clarke stationary point of $\varkappa$ if and only if $\varkappa^{\circ}({\theta} ; \vartheta) \geq 0$ for all $\vartheta\in\mathbb{R}^{d_{\theta}}$.
\end{proposition}

After addressing these preliminaries, we begin our proof, which is divided into two parts. In the first part, we establish a condition that ensures the feasibility of each stationary point of the penalty function \(\eqref{eq:penalty-function}\). Within the context of the probability space for parameterized functions \((\mathcal{F}, \mathcal{H}(\mathscr{A}), \mathbb{P}_{\mathcal{H}})\), we define the residual function as:
\[
R(\theta) = \mathbb{E}_{f}\left[P\left(X(T(\theta,f), \theta, T(\theta,f), f)\right)\right] + \mathbb{E}_{f}\left[Q(\theta, T(\theta,f))\right].
\]
This function quantifies the violation of constraints. The feasible set is defined by the equation:
\begin{equation}
    S = \{\theta \mid R(\theta) \leq 0\}.
\end{equation}
With these notations, we proceed to introduce the sufficient decrease condition based on Clarke directional derivatives.

\begin{assumption}[Uniform sufficient decrease condition]
    \label{assu:suff-decrease}
    For each infeasible point $\theta\in\mathbb{R}^{d_{\theta}}$, i.e. $\theta\notin S$, there exists a nonzero vector $\vartheta$, such that $R^{\circ}(\theta;\vartheta)\leq -c\|\vartheta\|$. Here the constant $c$ is uniform uniform across all $\theta$.
\end{assumption}

This condition is crucial for ensuring the effectiveness of the penalty function method as noted in the literature \cite{pang2022equilibrium,pang2023modern}. It guarantees that for every infeasible point, there is a direction in which the point can move toward the feasible set. This intuitive notion ensures that the solutions generated by Algorithm \ref{algo:SEPM} are feasible. We now formally demonstrates that the sufficient decrease condition effectively precludes any infeasible stationary points.

\begin{theorem}
    \label{thm:preclude-infeasible}
    Suppose $\mathbb{E}_{f}[T(\theta,f)]$ is globally Lipschitz continuous with Lipschitz constant $L_{T}$. Let Assumption \ref{assu:suff-decrease} hold. Given the penalty parameter $\rho>L_{T}/c$, any infeasible point of the penalty function \eqref{eq:penalty-function} can not be a $D$-stationary point.
\end{theorem}
\begin{proof}
    Rewrite the penalty function as $\upUpsilon(\theta)=\mathbb{E}_{f}[T(\theta,f)]+\rho R(\theta)$. For any infeasible point $\theta$, using Assumption \ref{assu:suff-decrease}, there must exists a direction $\vartheta$, such that
    \[
        \upUpsilon^{\circ}(\theta;\vartheta)=
        \mathbb{E}_{f}[T(\cdot,f)]^{\circ}(\theta;\vartheta)+\rho R^{\circ}(\theta,\vartheta)\leq L_{T}\|\vartheta\|-c\rho\|\vartheta\|< 0.
    \]
    Invoking proposition \ref{prop:criteria-clarke}, the criteria of Clarke stationary point, we know $\theta$ cannot be a Clarke stationary point of $\upUpsilon$. Meanwhile, according to \cite[Corollary 1]{bolte2021conservative}, Clarke subdifferential is a minimal conservative gradient. For any conservative gradient $D^{\upUpsilon}$ of $\upUpsilon$, we have $\partial \upUpsilon\subset D^{\upUpsilon}$. Hence, $\theta$ can not be a $D$-stationary point of $\upUpsilon$.\qed
\end{proof}


Next, we show that Algorithm \ref{algo:SEPM} converges a $D$-stationary point of the penalty function \eqref{eq:penalty-function}. We take the viewpoint of differential inclusion and utilize some fundamental results in this field \cite{davis2020sgdtame,michel2005sto}. Consider the following differential inclusion
\begin{equation}
    \label{eq:general-diff-inclusion}
    \frac{\mathrm{d}\theta(t)}{\mathrm{d}t}=-D^{\upUpsilon}(\theta(t)),\quad\text{for a.e. } t\geq t_0.
\end{equation}
Then, the update scheme in Algorithm \ref{algo:SEPM} can be viewed as a noisy discretization of the differential inclusion \eqref{eq:general-diff-inclusion}
\begin{equation}
    \label{eq:sgd}
    \theta_{k+1}=\theta_{k}-\eta_{k}(\varrho_{k}+\xi_{k}),\quad\varrho_{k}\in D^{\upUpsilon}(\theta_{k}).
\end{equation}
The existence and uniqueness of solutions to equation \eqref{eq:general-diff-inclusion} are established in \cite{aubin1984differential}.  The following lemma, also found in \cite{davis2020sgdtame} but presented here using the conventions of conservative gradients, guarantees a descent property for $\upUpsilon(\zeta(t))$ when $\zeta$ is a trajectory of \eqref{eq:general-diff-inclusion} with $\mathcal{G}=D^{\upUpsilon}$. This property directly stems from the outer-semicontinuity of the conservative gradient.

\begin{lemma}
    \label{lem:descent}
    Set $\mathcal{G}=D^{\upUpsilon}$ in \eqref{eq:general-diff-inclusion}. Then the differential inclusion \eqref{eq:general-diff-inclusion} satisfies the descent property: For any trajectory $\zeta: \mathbb{R}_{+} \rightarrow \mathbb{R}^d$ of the differential inclusion \eqref{eq:general-diff-inclusion} where $0 \notin \mathcal{G}(\zeta(0))$, there exists a real $t_{1}>0$ such that
    \[
    \varkappa(\zeta(t_{1})) < \sup_{t \in [0, t_{1}]} \varkappa(\zeta(t)) \leq \varkappa(\zeta(0)).
    \]
\end{lemma}
\begin{proof}
    We first verify the outer-semicontinuity of $D^{\upUpsilon}$. Employing the definition from \cite[sec. 5B]{rockafellar1998variational}, the outer limit of a mapping $S$ is
    \[
    \limsup _{\theta \rightarrow \bar{\theta}} S(\theta): =\bigcup_{\theta^\nu \rightarrow \bar{\theta}} \limsup _{\nu \rightarrow \infty} S\left(\theta^\nu\right)=\left\{u \mid \exists \theta^\nu \rightarrow \bar{\theta}, \exists u^\nu \rightarrow u \text { with } u^\nu \in S\left(\theta^\nu\right)\right\}.
    \]
    A set-valued mapping $S$ is outer-semicontinuous (osc) at $\bar{\theta}$ if $\limsup _{\theta \rightarrow \bar{\theta}} S(\theta) \subset S(\bar{\theta})$. The graph-closed property ensures that the conservative gradient $D^\upUpsilon$ is outer-semicontinuous.

    Suppose $\zeta$ is a solution of the differential inclusion \eqref{eq:general-diff-inclusion} with $\mathcal{G}=D^{\upUpsilon}$ and $0\notin D^{\upUpsilon}(\zeta(0))$, the outer-semicontinuity of $D^{\upUpsilon}$ guarantees that there exist $\delta>0$ and $t_{1}>0$, such that
    \[
        \|D^{\upUpsilon}(\zeta(t))\|:=\min_{\varrho\in D^{\upUpsilon}(\zeta(t))}\|\varrho\|\geq \delta,\quad \text{for all}\quad t\in [0,t_{1}].
    \]
    Using the fact $\|\dot{\zeta}(t)\|=\|D^{\upUpsilon}(\zeta(t))\|$ \cite[Lemma 5.2]{davis2020sgdtame} and the chain rule for conservative gradient we get
    \[
        \frac{\mathrm{d}\upUpsilon(\zeta(t))}{\mathrm{d}t}=\langle\varrho,\dot{\zeta}(t)\rangle=-\|\dot{\zeta}(t)\|^2=-\|D^{\upUpsilon}(\zeta(t))\|^2,\quad\varrho\in D^{\upUpsilon}(\zeta(t))\quad \text{a.e..}
    \]
    Combining both results we get the descent property.
    \qed
\end{proof}

To derive a convergence guarantee for Algorithm \ref{algo:SEPM}, we make the following assumptions of \eqref{eq:sgd}.

\begin{assumption}[Assumptions of the SGD]
    \label{assu:standing-SGD}
    \leavevmode
    \begin{enumerate}
        \item The step sizes $\left\{\eta_k\right\}_{k\geq 1}$ satisfy
        \[
        \eta_k \geq 0, \quad \sum_{k=1}^{\infty} \eta_k=\infty, \quad \text { and } \quad \sum_{k=1}^{\infty} \eta_k^2<\infty.
        \]
        \item Almost surely, the iterates $\{\theta_{k}\}_{k\geq 1}$ are bounded, i.e., $\sup_{k\geq 1}\|\theta_{k}\|<\infty$.
        \item $\{\xi_{k}\}_{k\geq 1}$ is a uniformly bounded difference martingale sequence with respect to the increasing $\sigma$-fields
        \[
            \mathcal{F}_{k}=\sigma(\theta_{j},\varrho_{j},\xi_{j}\colon j\leq k).
        \]
        In other words, there exists a constant $M_{\xi}>0$ such that
        \[
            \mathbb{E}[\xi_{k}\mid \mathcal{F}_{k}]=0\quad \text{and}\quad \mathbb{E}[\|\xi_{k}\|^2\mid \mathcal{F}_{k}]\leq M_{\xi}\quad \text{for all}\quad k\geq 1.
        \]
        \item The complementary of $\{\upUpsilon(\theta)\mid 0\in D^{\upUpsilon}(\theta)\}$ is dense in $\mathbb{R}$.
    \end{enumerate}
\end{assumption}

Assumption \ref{assu:standing-SGD} (4) is a technical assumption termed the weak Sard property. It is necessary since there exist a famous counter example \cite{whitney1935critical}. Using \cite[Lemma 4.1]{davis2020sgdtame}, we know Assumption \ref{assu:standing-SGD} (1-3) implies \cite[Assumption A]{davis2020sgdtame}  almost surely. One difference is that they assume there exists a function $\varkappa$ that bounded on bounded sets such that
\[
    \mathbb{E}[\|\xi_{k}\|^2\mid \mathcal{F}_{k}]\leq \varkappa(\theta_{k}).
\]
Given the boundedness of $\{\theta_{k}\}$, we know that this equals to Assumption \ref{assu:standing-SGD} (3). We are now at the position to give the convergence result for Algorithm \ref{algo:SEPM}.

\begin{theorem}[Convergence guarantee for Algorithm \ref{algo:SEPM}]\label{thm:converge-learning}
    Suppose Assumption \ref{assu:suff-decrease} and \ref{assu:standing-SGD} hold, $\mathbb{E}_{f}[T(\theta,f)]$ is local Lipschitz smooth, and $\{\theta_{k}\}_{k\geq 1}$ is generated by Algorithm \ref{algo:SEPM}. Then almost surely, every limit point $\theta_{\star}$ of $\{\theta_{k}\}_{k\geq 1}$ satisfies $\theta_{\star}\in S$, $0\in D^{\upUpsilon}(\theta_{\star})$ and the sequence $\{\upUpsilon(\theta_{k})\}_{k\geq 1}$ converges.
\end{theorem}
\begin{proof}
    When Assumption \ref{assu:standing-SGD} (1-3) hold, using our comment we know that \cite[Assumption A]{davis2020sgdtame} holds almost surely. Invoking Assumption \ref{assu:standing-SGD} (4) and Lemma \ref{lem:descent}, we know Assumption \cite[Assumption B]{davis2020sgdtame} holds. Then, applying \cite[Theorem 3.2]{davis2020sgdtame}, we get $0\in D^{\upUpsilon}(\theta_{\star})$ and the convergence of $\{\upUpsilon(\theta_{k})\}_{k\geq 1}$.

    Assumption \ref{assu:standing-SGD} (2) and the local Lipschitz smoothness of \(\mathbb{E}_{f}[T(\theta,f)]\) justify treating \(\mathbb{E}_{f}[T(\theta,f)]\) as globally Lipschitz smooth. By integrating this observation with Assumption \ref{assu:suff-decrease}, we invoke Theorem \ref{thm:preclude-infeasible} to eliminate the possibility of any infeasible points being stationary. This line of reasoning ensures that all limiting points must be feasible, that is, \(\theta_{\star} \in S\), thereby completing the proof.
    \qed
\end{proof}

To conclude this section, we note that it's possible to establish a connection between the stationary points obtained using finite sample approximations and the true stationary points at the population level. Specifically, as the number of samples increases indefinitely, the limit points of the finite-sample stationary points converge to the population-level stationary points. For a detailed discussion of this result, one can refer to Proposition 17 in \cite{pang2022chance}.

\section{Numerical results on the training process}
\label{sec:training-experiments}
This section presents the results obtained from running Algorithm \ref{algo:SEPM} and verifies Theorem \ref{thm:converge-learning} through numerical experiments. We provide a detailed description of the dataset, the approach used to construct the training problems, and the implementation details of the training process. The numerical experiments were implemented using the PyTorch platform on a Lenovo workstation equipped with an Intel i9 processor, 64 GB of RAM, and an NVIDIA RTX 4090 GPU, running the Windows Subsystem for Linux.

\subsection{Methodology for constructing minimization problems}
\label{sec:methodology}
In this subsection, we describe the methods used to construct the minimization problems referenced in \eqref{eq:unconstrained-minimization}. Given a dataset $\upSigma$, in $k$-th iteration of Algorithm \ref{algo:SEPM}, we draw finite instances from $\upSigma$ and construct the set $\mathscr{D}_{k}$. Then, the function $f_{k}:=f_{\mathscr{D}_{k}}$ is constructed and can be used in Algorithm \ref{algo:SEPM}.

We test Algorithm \ref{algo:SEPM} in two type minimization problems. Consider a finite set of instances, \(\mathscr{D}\), consisting of data pairs \(\{a_{i}, b_{i}\} \in \mathbb{R}^n \times \{0,1\}, i \in [|\mathscr{D}|]\). The first type of minimization problem is a logistic regression problem defined by:
\[
    \min_{x \in \mathbb{R}^n} f_{\mathscr{D}}(x) = \frac{1}{|\mathscr{D}|} \sum_{(a_{i}, b_i) \in \mathscr{D}} \log (1 + \exp(-b_i \langle a_i, x \rangle)),
\]
Let \(\sigma(t) = \frac{1}{1 + \exp(-t)}\) belong to the interval (0,1). The Hessian matrix of \(f_{\mathscr{D}}\) is given by:
\[
    \nabla^2 f_{\mathscr{D}}(x) = \frac{1}{|\mathscr{D}|} \sum_{(a_{i}, b_i) \in \mathscr{D}} b_i^2 a_i a_i^\top \sigma(b_i \langle a_i, x \rangle)(1 - \sigma(b_i \langle a_i, x \rangle)).
\]
Let \(A = (a_1, \ldots, a_{|\mathscr{D}|})\). The Lipschitz constant of \(\nabla f_{\mathscr{D}}\) is bounded by \(L = \|AA^\top\| / |\mathscr{D}|\).

Given an even integer \(p \geq 4\), the second type of minimization problem is the \(\ell_p^p\) minimization, defined as follows:
\[
    \min_{x \in \mathbb{R}^n} f_{\mathscr{D}}(x) = \frac{1}{|\mathscr{D}|} \sum_{(a_{i}, b_i) \in \mathscr{D}} \frac{1}{p} (\langle a_i, x \rangle - b_i)^p.
\]
The Hessian matrix of \(f_{\mathscr{D}}\) is expressed as:
\[
    \nabla^2 f_{\mathscr{D}}(x) = \frac{1}{|\mathscr{D}|} \sum_{(a_{i}, b_i) \in \mathscr{D}} (p - 1) (\langle a_i, x \rangle - b_i)^{p - 2} a_i a_i^\top.
\]
As \((\langle a_i, x \rangle - b_i)^{p - 2}\) is unbounded for each \(i\), the Lipschitz constant for \(\nabla f_{\mathscr{D}}\) cannot be globally bounded.

\subsection{Datasets for constructing minimization problems}

The datasets used in our experiments are summarized in Table \ref{tab:dataset}. In this table, $n, N_{\text{train}}$, and $N_{\text{test}}$ represent the dimension of the variable, the number of instances in the training dataset, and the number of instances in the test dataset, respectively.
\begin{table}[htbp]
    \centering
    \caption{A summary of the datasets used in experiments.}
    \label{tab:dataset}
    \resizebox{.8\textwidth}{!}{%
    \begin{tabular}{cccccc}
    \toprule
    Dataset & $n$ & $N_{\text{train}}$ & $N_{\text{test}}$ & Separable & References \\ \hline
    \texttt{a5a}            & $123$       & $6,414$   & $26,147$  & No & \cite{Dua:2019} \\
    \texttt{w3a}            & $300$       & $4,912$   & $44,837$  & No & \cite{platt1998fast} \\
    \texttt{mushrooms}      & $112$       & $3,200$   & $4,924$   & Yes & \cite{Dua:2019} \\
    \texttt{covtype}        & $54$         & $102,400$   & $478,612$   & No & \cite{Dua:2019} \\
    \texttt{phishing}       & $68$        & $8,192$     & $2,863$     & No & \cite{Dua:2019} \\
    \texttt{separable}      & $101$        & $20,480$    & $20,480$    & Yes & \cite{wilson2019accelerating} \\
    \bottomrule
    \end{tabular}
    }
\end{table}

All the datasets are designed for binary classification problems, and downloaded from the \href{https://www.csie.ntu.edu.tw/~cjlin/libsvmtools/datasets/}{LIBSVM data}, except the \texttt{separable} dataset. We construct the \texttt{separable} dataset using the code snippet downloaded from \cite{wilson2019accelerating}. They are generated by sampling $10240$ instances from $\mathcal{N}(\mu,I_{d})$ with label $b_{i}=1$ and $\mathcal{N}(\mu+\nu,I_{d})$ with label $b_{i}=0$, respectively. Here, $I_{d}\in\mathbb{R}^{d\times d}$ denotes the identity matrix. Each element of the vector $\mu\in\mathbb{R}^{d}$ is sampled from $\{0,1,\ldots,19\}$ uniformly, while the elements of the margin vector $\nu$ are drawn from $\{0,0.1,\ldots,0.9\}$ uniformly.

For each dataset, the label of each sample belongs to $\{0,1\}$. The value of each attribute are normalized to $[-1,1]$ by dividing the data-matrix $(a_1,a_2,\ldots,a_N)$ with the max absolute value of each attribute. The training and testing sets are pre-specified for \texttt{a5a} and \texttt{w3a}. For datasets that do not specify the testing set and training set, we divide them manually.

\subsection{Verifying the $(L_{0},L_{1})$-smoothness assumption}

Both the objective function used in logistic regression and the \(\ell_p^p\) minimization problem exhibit \((L_{0},L_{1})\)-smoothness. This characteristic remains consistent across various datasets, with the constants \(L_{0}\) and \(L_{1}\) depending on the specific dataset \(\mathscr{D}\). To illustrate this, we applied four different algorithms to two datasets, the details of which will be provided in subsection \ref{sec:compared-methods}. Each algorithm was executed for 300 steps, and every 30 steps, we plotted the point \((\|\nabla f(x_k)\|,\upLambda(x_k,f))\) on a log-log scale scatter plot. Despite the diversity in the methods used, all points aligned along the same line, underscoring the \((L_{0},L_{1})\)-smoothness as an intrinsic property of the function \(f_{\mathscr{D}}\). Figure \ref{fig:smooth} shows a pronounced decline in \(\upLambda(x_k,f)\) corresponding to decreases in \(\|\nabla f(x_k)\|\), confirming the arguments in sec. \ref{sec:lip-estimator}.

\begin{figure}[htbp]
    \centering
    \begin{subfigure}{.49\textwidth}
    \centering
    \includegraphics[width=\linewidth]{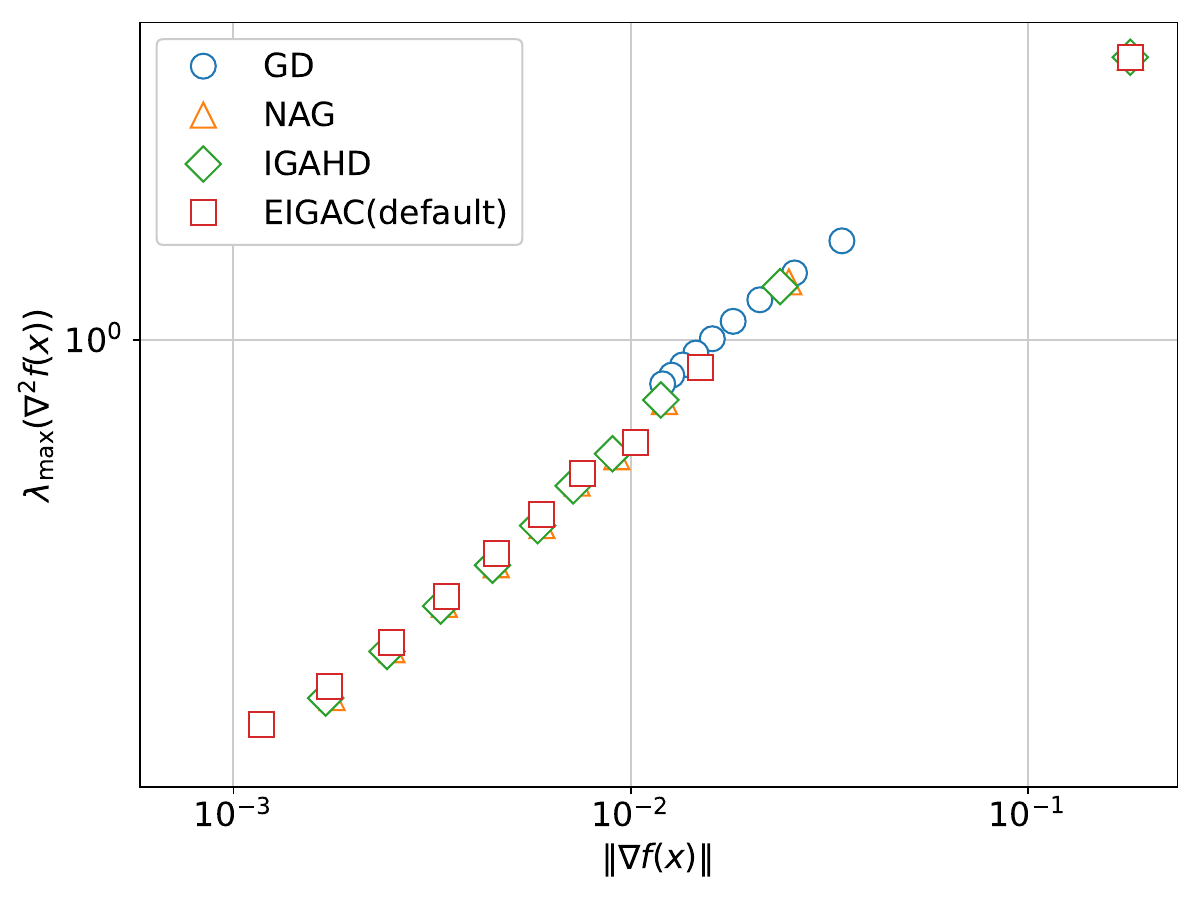}
    \caption{Logistic regression with \texttt{separable}}\label{fig:logistic-separable-smooth}
    \end{subfigure}
    \begin{subfigure}{.49\textwidth}
    \centering
    \includegraphics[width=\linewidth]{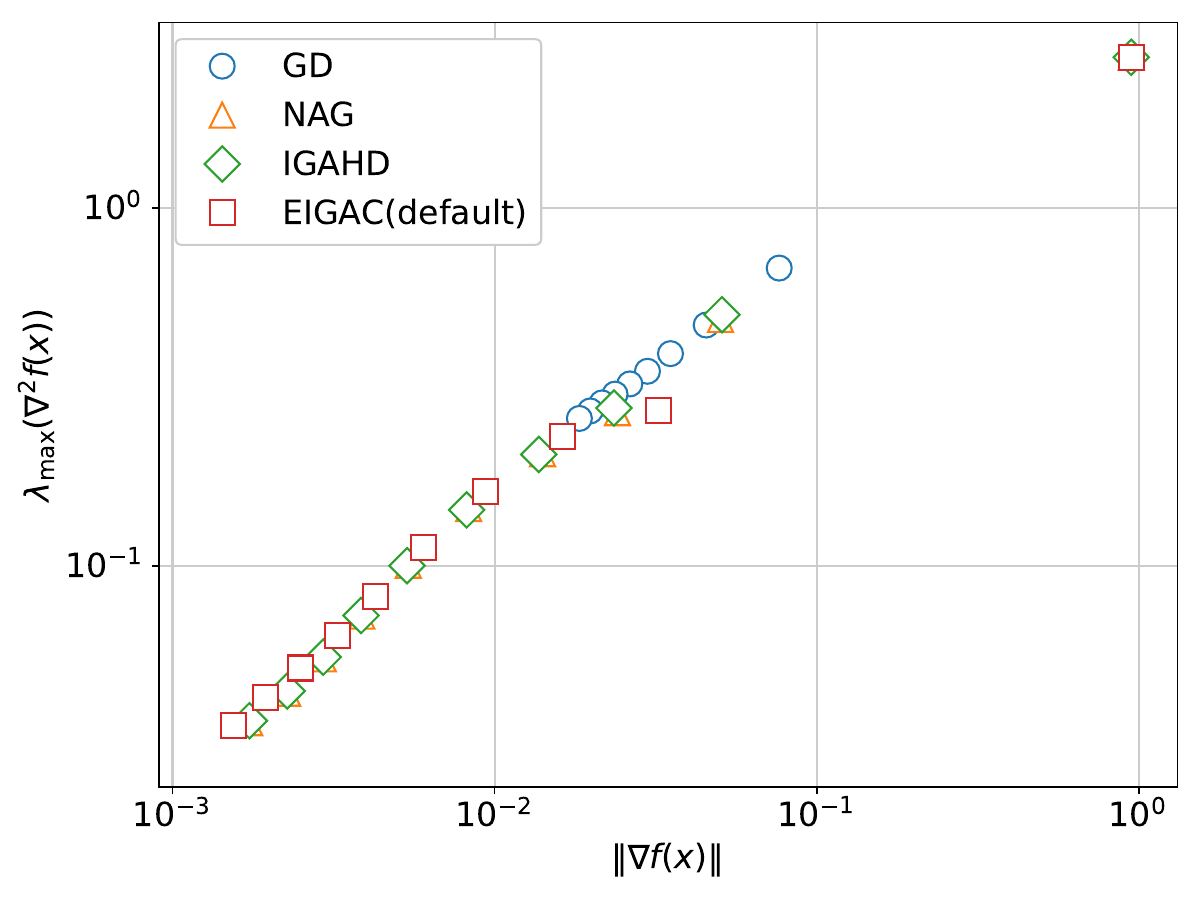}
    \caption{Logistic regression with \texttt{mushrooms}}
    \label{fig:logistic-mushrooms-smooth}
    \end{subfigure}
    \begin{subfigure}{.49\textwidth}
    \centering
    \includegraphics[width=\linewidth]{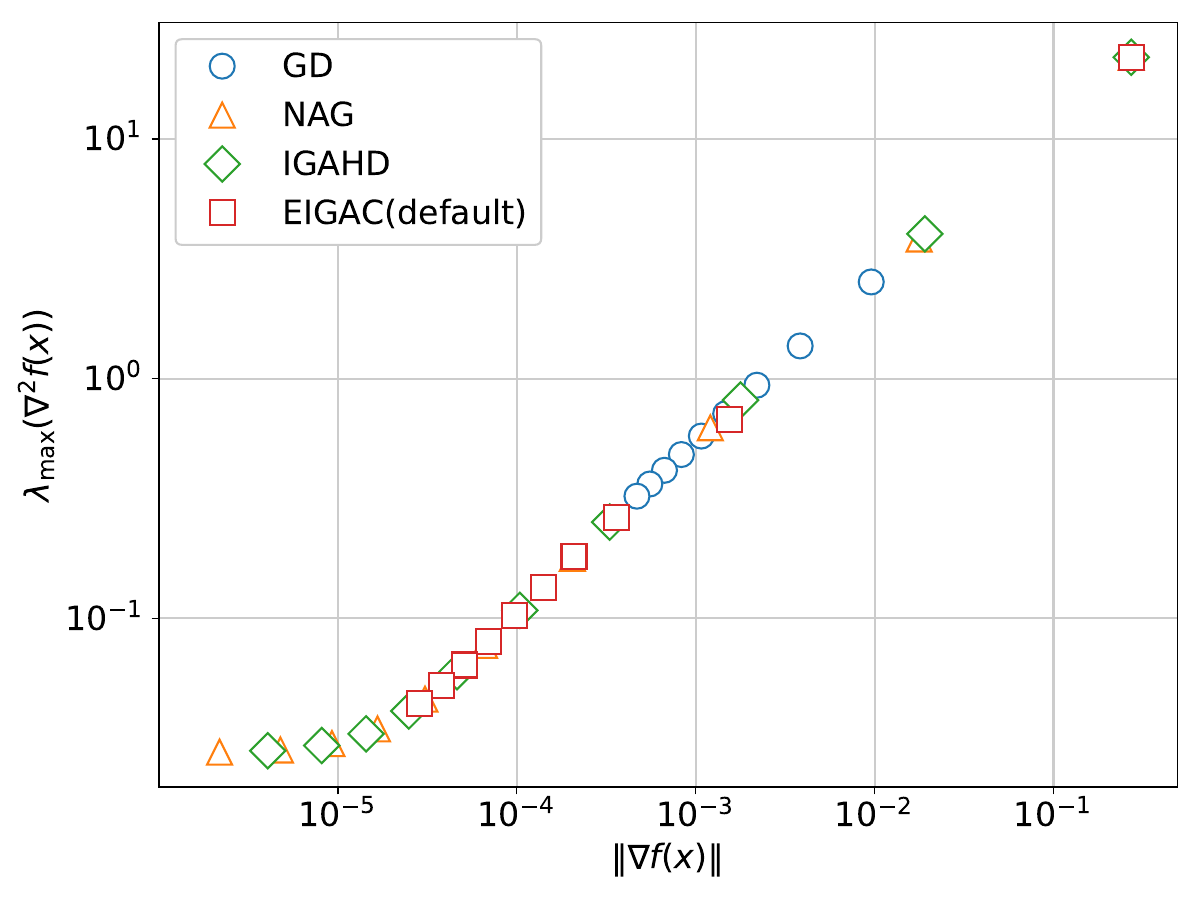}
    \caption{$\ell_p^p$ minimization with \texttt{separable}}
    \label{fig:lpp-separable-smooth}
    \end{subfigure}
    \begin{subfigure}{.49\textwidth}
    \centering
    \includegraphics[width=\linewidth]{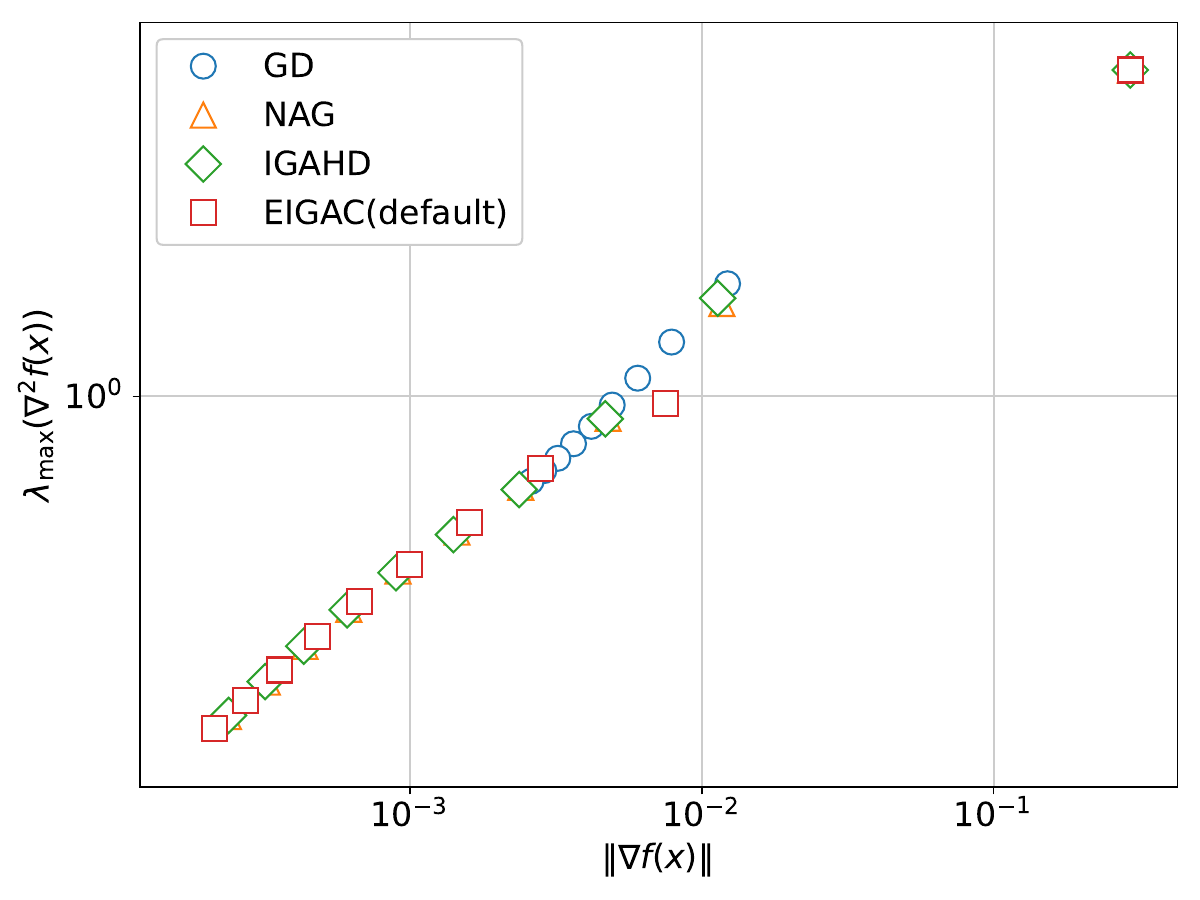}
    \caption{$\ell_p^p$ minimization with \texttt{mushrooms}}
    \label{fig:lpp-mushrooms-smooth}
    \end{subfigure}
    \caption{Numerical verification of the $(L_0,L_1)$-smoothness.\label{fig:smooth}}
\end{figure}

\subsection{Implementation details and the training process}
\label{sec:implementation-detail}
Let $\sigma(\cdot)$ represent applying the SoftPlus function element-wise, i.e., $[\sigma(x)]_{i}=\log(1+\exp(20x_{i}))/20$. Given the dimension of the hidden state $d_{\mathrm{h}}$, we parameterize $\beta$ and $\gamma$ as follows:
\[
    \beta_{\theta_{1}}(t)=W_3^\top \sigma(W_2\sigma(W_1t+b_1)+b_2)+b_3,\quad
    \gamma_{\theta_{2}}(t)=V_3^\top \sigma(V_2\sigma(V_1t+c_1)+c_2)+c_3,
\]
where $W_1, V_1, W_3, V_3,b_1,c_1,b_2,c_2\in\mathbb{R}^{d_{\mathrm{h}}}, W_2, V_2\in\mathbb{R}^{d_{\mathrm{h}}\times d_{\mathrm{h}}}$ and $b_3,c_3\in\mathbb{R}$. In this context, the SoftPlus function ensures the differentiability of $\beta_{\theta_{1}}$ and $\gamma_{\theta_{2}}$ with respect to $t$. We define \(\upLambda(x,f) := \lambda_{\max}(\nabla^2 f(x))\) to leverage the variability of the local Lipschitz constant of \(\nabla f\). To control the computational complexity for evaluating $\upLambda(x,f)$, we combine the power iteration with the forward automatic differentiation. We initiate the algorithm with a randomly generated vector $u_{0}$. In each iteration, $z_{k+1}$ is computed successively using the formula $z_{k+1}={\nabla^2 f(x)z_k}/{\|\nabla^2 f(x)z_k\|}$. The algorithm terminates when either the Euclidean norm of the difference between two successive iterations is less than or equal to $10^{-6}$, i.e., $\|z_{k+1}-z_{k}\|\leq 10^{-6}$, or the number of iterations reaches or exceeds 10, i.e., $k\geq 10$. The output of the algorithm is denoted as $u_{\star}$. We restore it for the backpropagation and use the approximation $\upLambda(x,f)\approx u_{\star}^\top\nabla^2 f(x)u_{\star}$. The penalty terms $P$ and $Q$ can be calculated by augmenting \eqref{eq:first-order} as
\[
(
    \dot{s}(t),\dot{P},\dot{Q}
)
=
\left(
    \psi_{\theta}(s(t),t), p(t,f,\theta), q(t,f,\theta)
\right).
\]
Integrating this system from $t_0$ to $T(f,\theta)$ gives the value of $P$ and $Q$. The package \texttt{torchdiffeq} implemented using \texttt{PyTroch} provides the implementation of the adjoint sensitivity method. We integrate it with our implementation of the backpropagation of $T(\theta,f)$ and $\upLambda(x,f)$. This combination enables the automatic differentiation through $\upUpsilon(\theta)$.

The initial point is chosen as $x_0=x_1=\mathbf{1}/n-\nabla f_{\mathscr{D}}(\mathbf{1}/n)/L$, and $v_{0}=x_{0}+\beta(t_{0})\nabla f(x_{0})$. All elements of the $n$-dimensional vector $\mathbf{1}$ are equal to $1$. We set $h=0.04$, $t_{0}=1$, and $L=\min\{\|A^\top A\|/N,4\upLambda(\mathbf{1}/n,f_{\mathscr{D}})\}$. Here $L$ is a tight estimate of the Lipschitz constant of $\nabla f_{\mathscr{D}}$. For each dataset, the stopping criterion is 
$\|\nabla f(x)\|=\varepsilon$ with $\varepsilon=3\times 10^{-4}$, and the penalty coefficient $\rho$ of \eqref{eq:penalty-function} is $0.1$. We adopt the \texttt{SGD} optimizer in \texttt{PyTorch} with the learning rate $0.001$ and the momentum $0.9$. The number of epochs is $n_{\mathrm{epoch}}=60$. We initialize the coefficient functions as $\alpha = 6$, $\beta(t) = (4 - 2\alpha h/t)/L$, and $\gamma(t) = \beta(t)/h$. In $k$-th step of Algorithm \ref{algo:SEPM}, we construct a problem $f_{k}:=f_{\mathscr{D}_{k}}$ by sampling $n_{\mathrm{sp}}=10240$ instances from the training dataset for \texttt{mushrooms} dataset and $n_{\mathrm{sp}}=1024$ instances for other datasets. Then, we perform one step of the Algorithm \ref{algo:SEPM} to update $\theta$.

\begin{figure}[htbp]
    \begin{subfigure}{.45\textwidth}
    \centering
    \includegraphics[width=\linewidth]{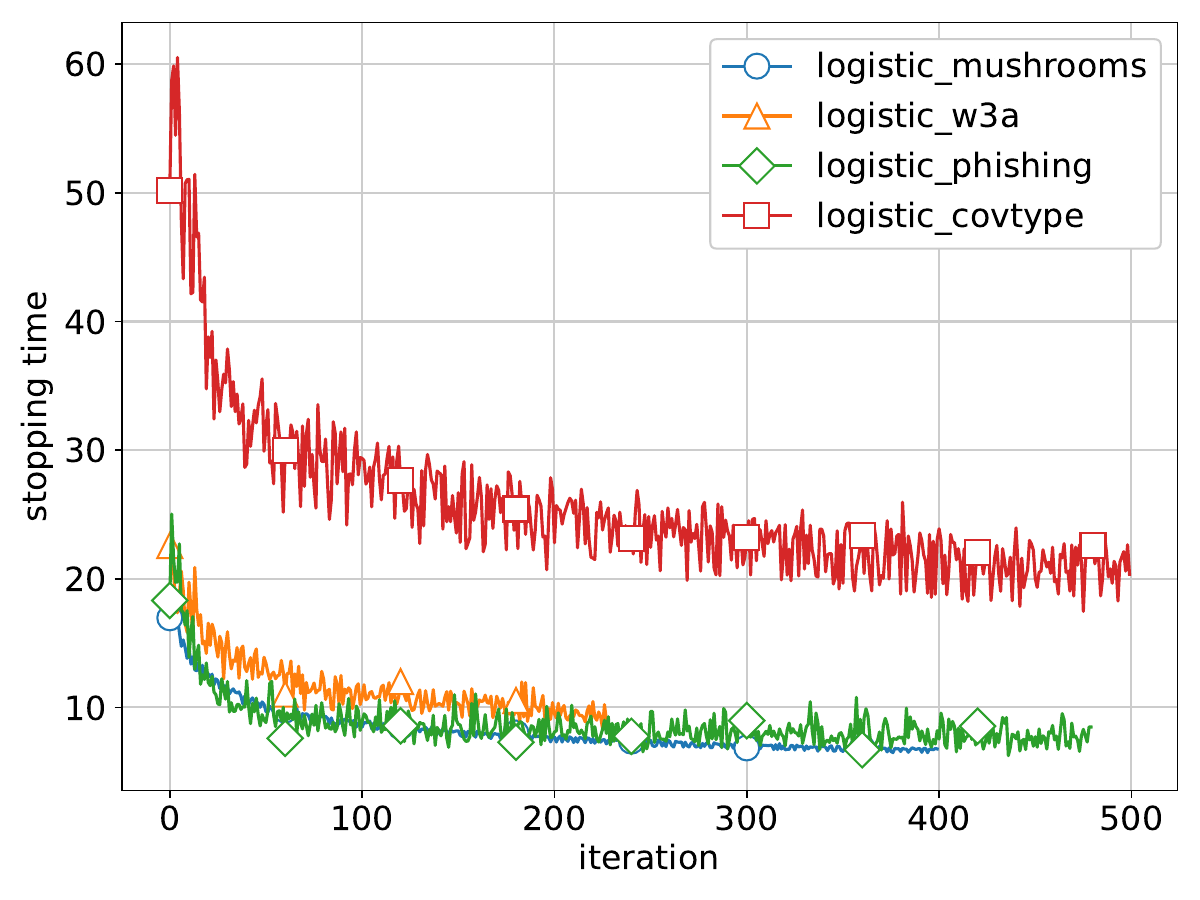}
    \caption{Stopping time on logistic regression}
    \label{fig:time-logi}
    \end{subfigure}
    \hfill
    \begin{subfigure}{.45\textwidth}
    \centering
    \includegraphics[width=\linewidth]{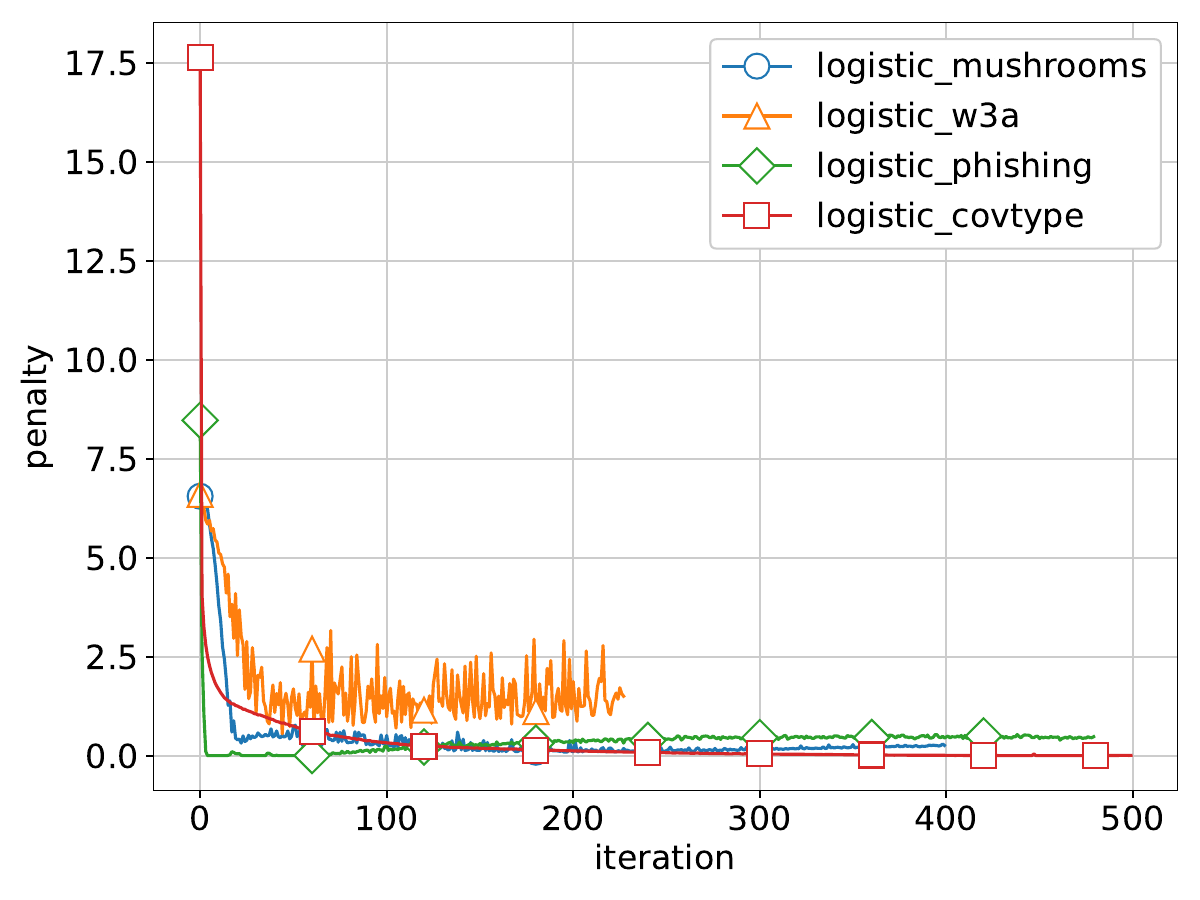}
    \caption{Penalty on logistic regression}
    \label{fig:penalty-logi}
    \end{subfigure}
    \begin{subfigure}{.45\textwidth}
    \centering
    \includegraphics[width=\linewidth]{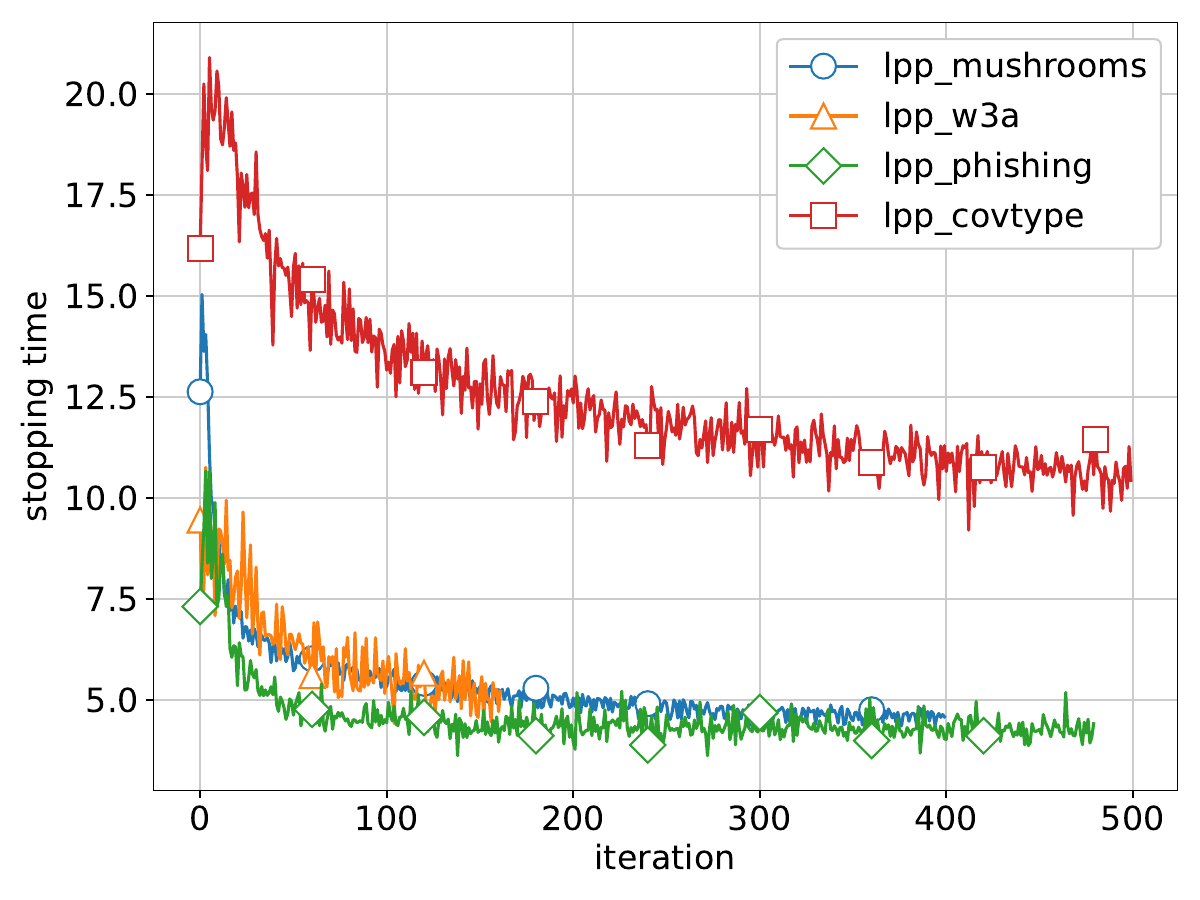}
    \caption{Stopping time on $\ell_{p}^{p}$ minimization}
    \label{fig:time-lpp}
    \end{subfigure}
    \hfill
    \begin{subfigure}{.45\textwidth}
    \centering
    \includegraphics[width=\linewidth]{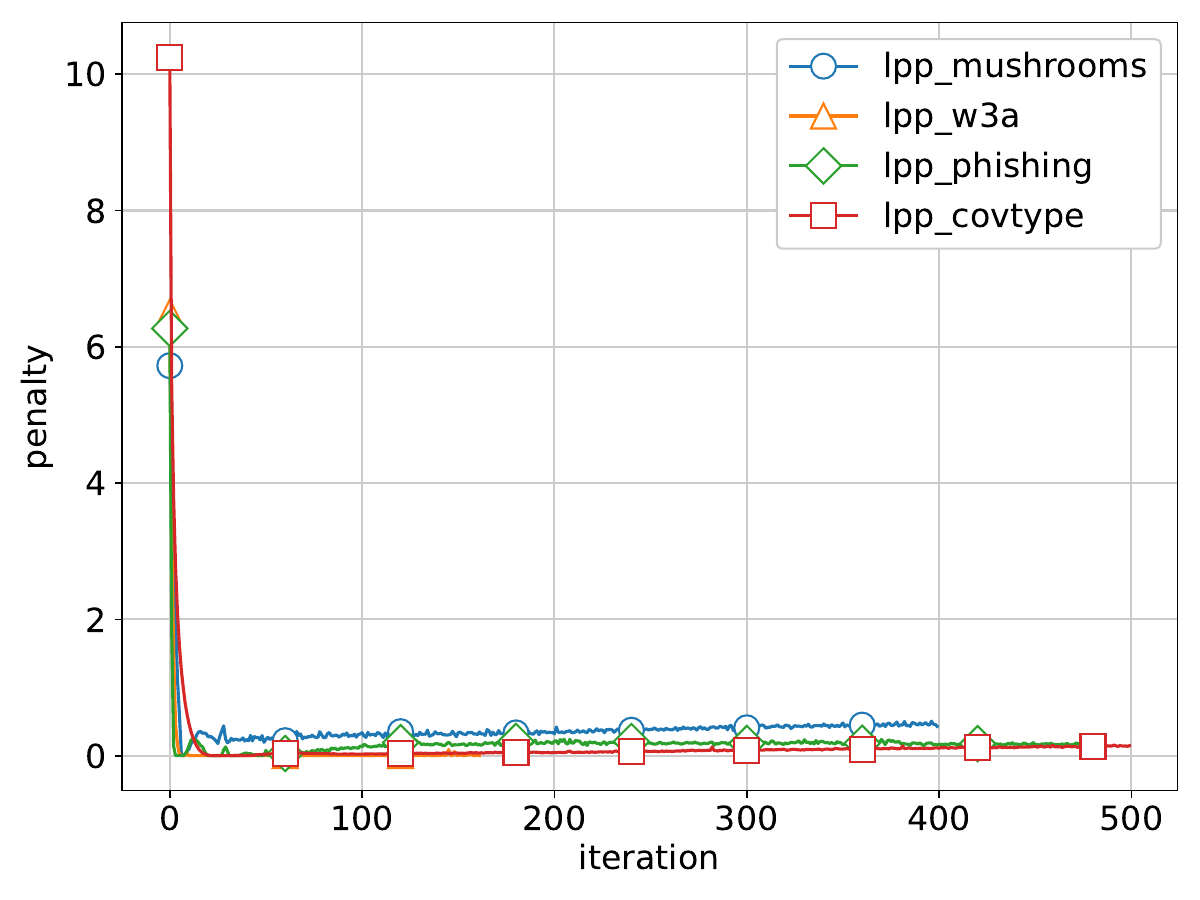}
    \caption{Penalty on $\ell_{p}^{p}$ minimization}
    \label{fig:penalty-lpp}
    \end{subfigure}
    \caption{The training process in different tasks.}
    \label{fig:training-loss}
\end{figure}

Figure \ref{fig:training-loss} illustrates the training process, where \texttt{logistic\_mushrooms} signifies logistic regression on the \texttt{mushrooms} dataset, \texttt{lpp\_mushrooms} refers to $\ell_{p}^{p}$ minimization on the same dataset, and so forth. The training loss is decomposed into the stopping time and penalty term. The subfigures within Figure \ref{fig:training-loss} clearly demonstrate a decrease in both constraint violation and stopping time across all experiments. This aligns with Theorem \ref{thm:converge-learning}, confirming that Algorithm \ref{algo:SEPM} converges to a feasible $D$-stationary point of problem \eqref{eq:learn-continuous}.

\section{Numerical results on the testing process}
\label{sec:experiments}
In this section, we present a series of numerical experiments to illustrate the advantages of our L2O framework \eqref{eq:learn-continuous}. We will demonstrate that by discretizing the learned ODE, our approach achieves acceleration compared to baseline methods.

\subsection{Effects of the contraction factor and the relay EIGAC}
\label{sec:effect-rho}
In this subsection, we investigate the effects of the contraction factor using the \texttt{lpp\_a5a} problem. A relay strategy for EIGAC is proposed to generalize the learned EIGAC to unseen scenarios while maintaining convergence. The parameterized ODE is trained for 80 epochs using Algorithm \ref{algo:SEPM}. As shown in Figure \ref{fig:epoch}, the performance of the learning-based method is compared with different training epochs on the same \texttt{lpp\_a5a} testing problem. The corresponding contraction factor at each step is plotted with the following legend descriptions:
\begin{itemize}
\item \texttt{default} represents the EIGAC method with initial coefficients $\alpha = 6$, $\beta(t) = \left({4}/{h} - {2\alpha}/{t}\right)/L$, and $\beta(t) = h\gamma(t)$, where $L = \min\left\{{\|A^\top A\|}/{N}, 4\upLambda\left({\mathbf{1}}/{n}, f\right)\right\}$;
\item \texttt{epoch\ 10} represents the EIGAC method with coefficients trained for 10 epochs, and similarly for \texttt{epoch\ 80};
\item \texttt{relay} combines the coefficients trained for 80 epochs with a safeguard strategy to extrapolate the method to untrained scenarios. When $\|\nabla f(x_{k})\| \geq 5\|\nabla f(x_{0})\|$ and $\|\nabla f(x_{k})\| < 3 \times 10^{-4}$, the coefficients $\gamma(t)$ are replaced with $\left({4}/{h} - {2\alpha}/{t}\right)/\upLambda(x_k, f)$ and $\beta(t) = h\gamma(t)$.
\end{itemize}
Here the estimator is chosen as $\upLambda(x,f)=\lambda_{\max}(\nabla^2 f(x))$. Besides the gradient norm, we also provide the figures of the contraction factor $\rho_{k}=\rho(t_{k},G_{k})$, the determination $C_{k}=C(t_{k},\|G_{k}\|)$, and $B_{k}=B(t_{k},\|G_{k}\|)$ with $G_{k}=\int_{0}^{1}\nabla^2 f((1-\tau) x(t_{k})+\tau x_{k})\,\mathrm{d}\tau$. The functions $B,C$ are defined in Proposition \ref{prop:maximal-rho}.

\begin{figure}[htbp]
    \begin{subfigure}{.49\textwidth}
    \centering
    \includegraphics[width=\linewidth]{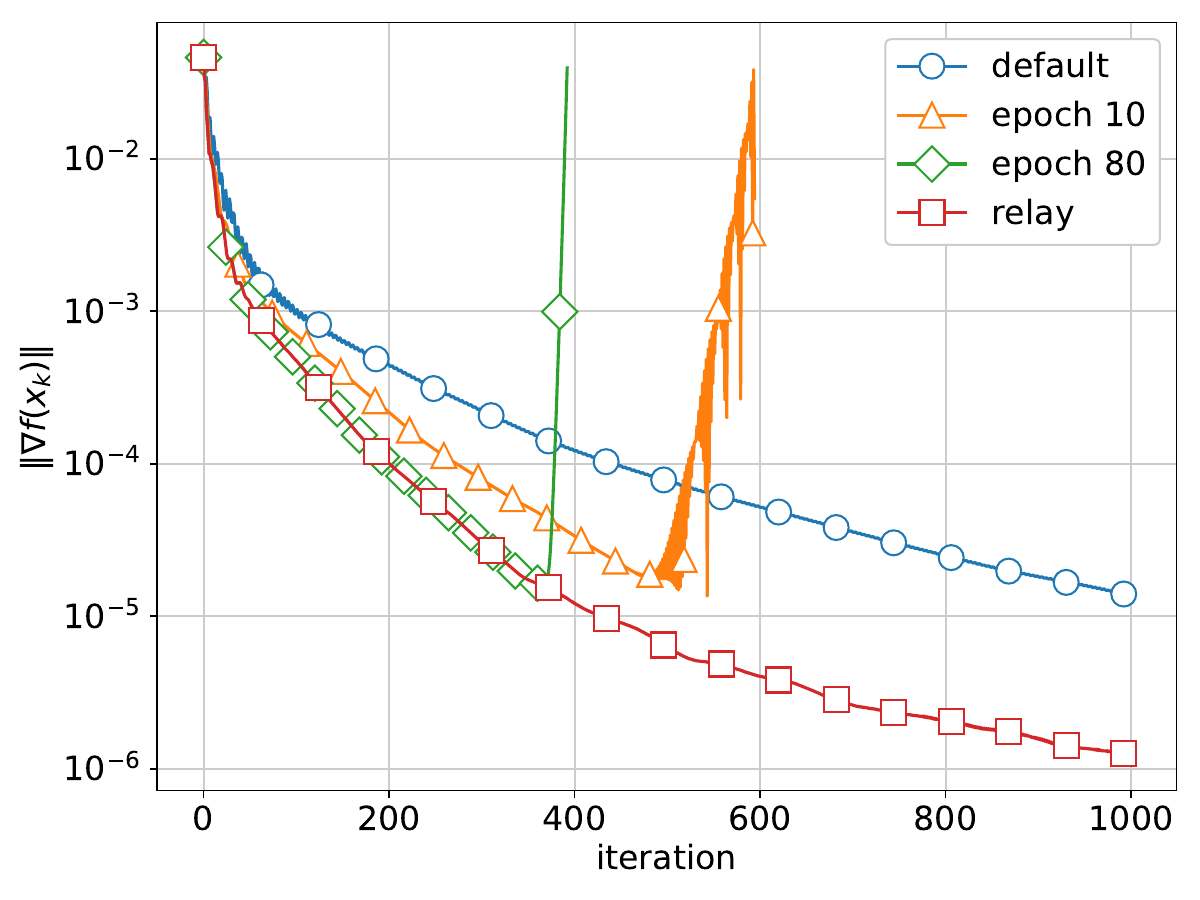}
    \caption{$\ell_{p}^{p}$ minimization on \texttt{a5a}}
    \label{fig:epoch-grad}
    \end{subfigure}
    \hfill
    \begin{subfigure}{.49\textwidth}
    \centering
    \includegraphics[width=\linewidth]{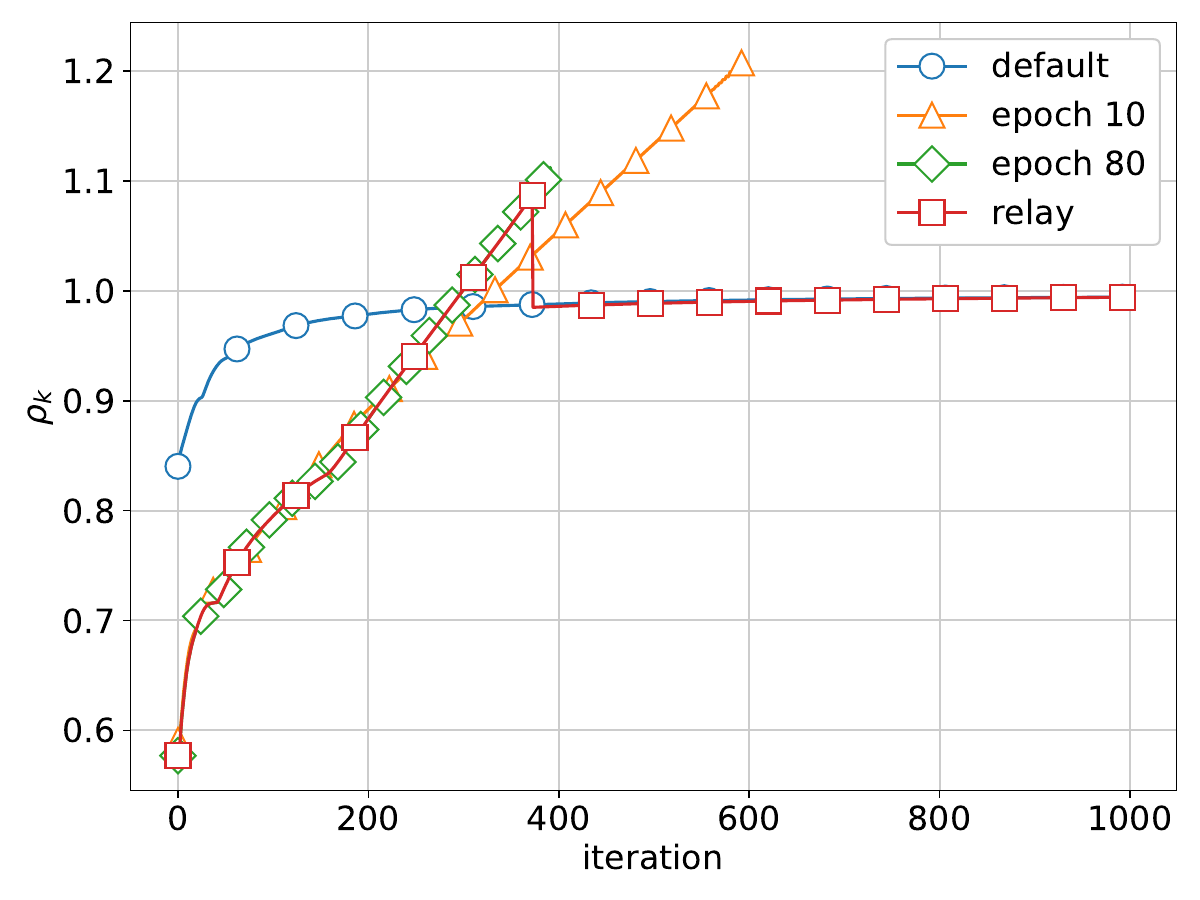}
    \caption{The contraction factor $\rho$ on \texttt{a5a}}
    \label{fig:epoch-rho}
    \end{subfigure}
    \begin{subfigure}{.49\textwidth}
    \centering
    \includegraphics[width=\linewidth]{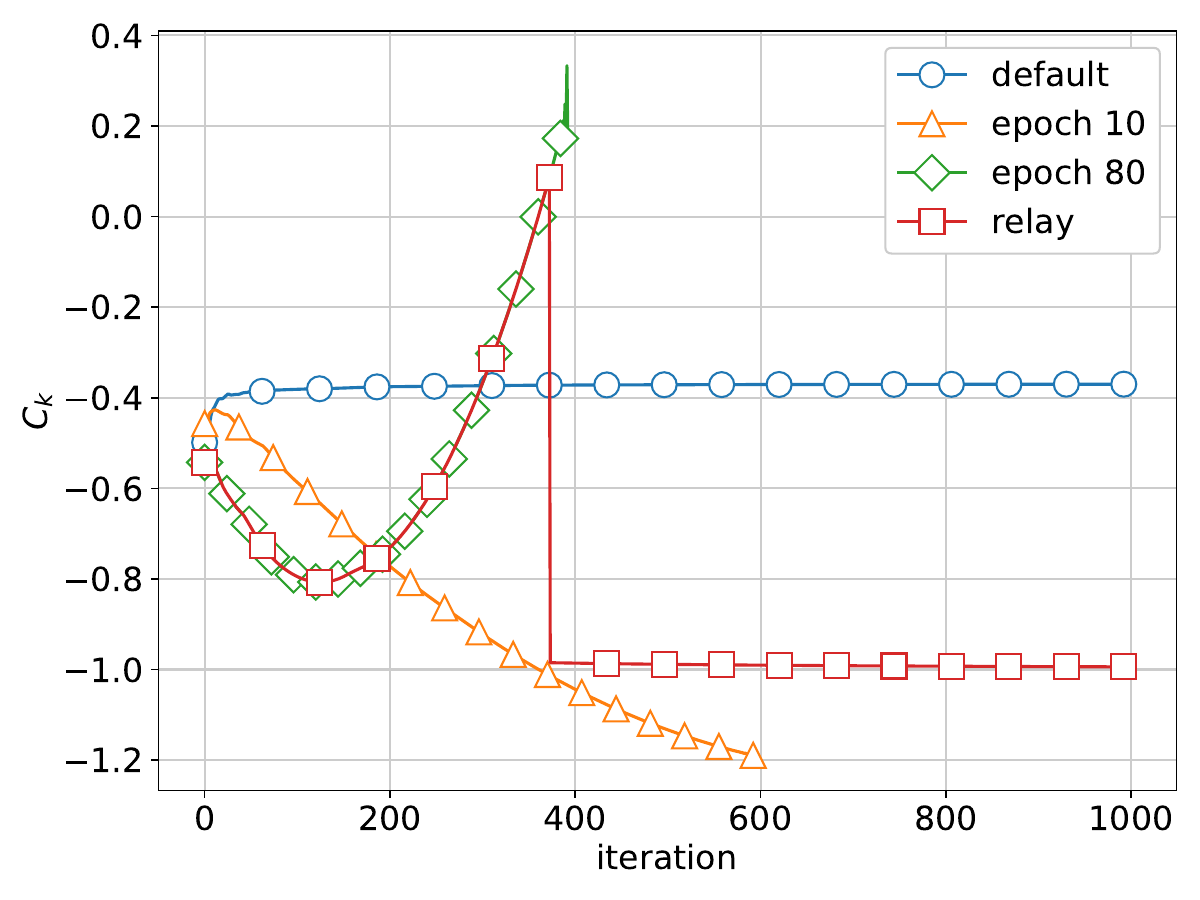}
    \caption{Determination versus iterations}
    \label{fig:epoch-delta}
    \end{subfigure}
    \hfill
    \begin{subfigure}{.49\textwidth}
    \centering
    \includegraphics[width=\linewidth]{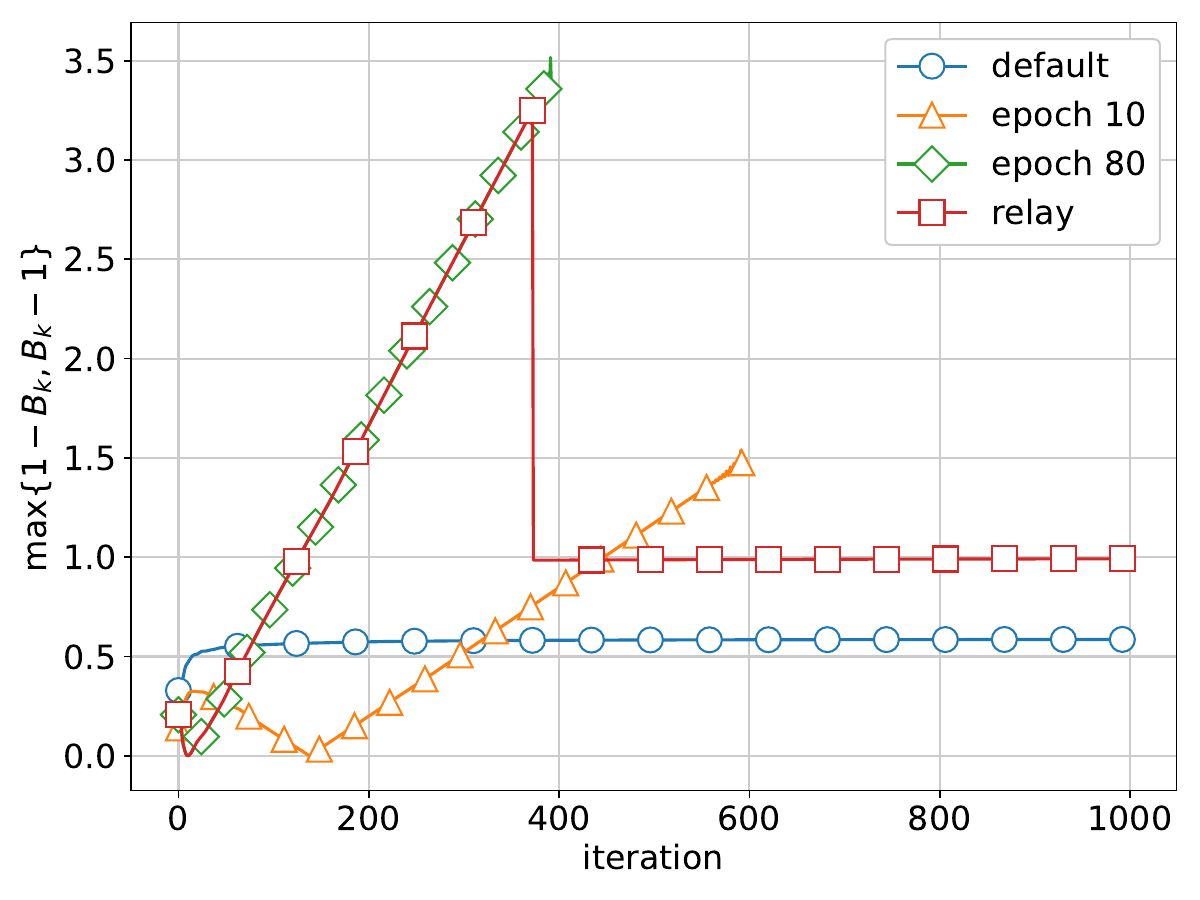}
    \caption{$\max\{B_{k}-1,1-B_{k}\}$ versus iterations}
    \label{fig:epoch-min-b}
    \end{subfigure}
    \caption{Different indicators of $\ell_{p}^{p}$ minimization problem on \texttt{a5a} dataset.\label{fig:epoch}}
\end{figure}

The results demonstrate that the contraction factor is a reliable predictor for the behavior of the sequence \(\{x_{k}\}_{k=0}^{\infty}\) generated by Algorithm \ref{algo:EIGAC}, supporting the intuition behind the proof of Theorem \ref{thm:stable}. Despite initial oscillations, the \texttt{default} method does not diverge, as the contraction factor remains strictly less than 1.

The \texttt{epoch 10} method, trained for 10 epochs, converges faster than the \texttt{default} method but diverges quickly after reaching a threshold $\|\nabla f(x)\|=3\times 10^{-5}$. This divergence occurs because its contraction factor is greater than 1 after 330 iterations, leading to an exponential accumulation of the global truncated error. The \texttt{epoch 80} method, which converges even faster, also faces divergence. However, the divergence patterns differ: the contraction factor for the \texttt{epoch 10} method increases gradually above 1, causing a gradual divergence with oscillation. In contrast, the \texttt{epoch 80} method diverges rapidly due to \(C_{k}\) becoming greater than 0, forcing \(\max\{B_{k}-1, 1 - B_{k}\}\) to quickly increase the contraction factor far above 1, leading to immediate divergence without gradual error accumulation.

The \texttt{relay} method remains stable after the \texttt{epoch 80} method diverges. This is due to the relay strategy, which reduces the contraction factor to below 1, maintaining stability. Upon detecting divergence, relay is applied to the coefficients, resulting in a quick decrease in \(C_{k}\) to ensure it stays below 0. This relay strategy method will be applied for comparison with other methods.

\subsection{Compared methods}
\label{sec:compared-methods}
In testing process, we adopt the same implementation details and convention used in the training process \ref{sec:implementation-detail}. Suppose the minimization problem used for testing is constructed from the set $\mathscr{D}$. we set $L=\min\{\|A^\top A\|/N,4\upLambda(\mathbf{1}/n,f_{\mathscr{D}})\}$, $x_0=x_1=\mathbf{1}/n-\nabla f_{\mathscr{D}}(\mathbf{1}/n)/L$, and $v_{0}=x_{0}+\beta(t_{0})\nabla f_{\mathscr{D}}(x_{0})$. The compared methods in our experiments are listed below, where we abbreviate the $f_{\mathscr{D}}$ as $f$ for simplicity.
\begin{itemize}
    \item \textbf{GD}. The vanilla gradient descent GD is the standard method in optimization. We set the stepsize as $s=1/L$ and perform $x_{k+1}=x_{k}-s\nabla f(x_{k})$ in each iteration.
    \item \textbf{NAG}. Nesterov's accelerated gradient descent method NAG is a milestone of the acceleration methods. We employ the version for convex functions
    \[
        y_{k+1}=x_{k}-s\nabla f(x_{k}),\quad x_{k+1}=y_{k+1}+\frac{k-1}{k+2}(y_{k+1}-y_{k}),
    \]
    where the stepsize is chosen as $s=1/L$.
    \item \textbf{IGAHD}. Inertial gradient algorithm with Hessian-driven damping. This method is obtained by applying a NAG inspired time discretization of
    \begin{equation}\label{eq:IGAHD}
    \ddot{x}(t)+\frac{\alpha}{t}\dot{x}(t)+\beta\nabla^2 f(x(t))\dot{x}(t)+\left(1+\frac{\beta}{t}\right)\nabla f(x(t))=0.
    \end{equation}
    Let $s=1/L$. In each iteration, setting $\alpha_{k}=1-\alpha/k$, the method performs
    \begin{equation}
        \left\{
        \begin{aligned}
        & y_k=x_k+\alpha_k\left(x_k-x_{k-1}\right)-\beta \sqrt{s}\left(\nabla f\left(x_k\right)-\nabla f\left(x_{k-1}\right)\right)-\frac{\beta \sqrt{s}}{k} \nabla f\left(x_{k-1}\right), \\
        & x_{k+1}=y_k-s \nabla f\left(y_k\right).
        \end{aligned}
        \right.
    \end{equation}
    In \cite{attouchFirstorderOptimizationAlgorithms2020}, it has been show that IGAHD owns $\mathcal{O}(1/k^2)$ convergence rate when $0\leq \beta< 2/\sqrt{s}$ and $s\leq 1/L$. Its performance may not coincide with NAG due to the existence of the gradient correction term. In our experiments, IGAHD serves as a baseline of the optimization methods derived from the ODE viewpoint without learning.
    \item \textbf{EIGAC}. Explicit inertial gradient algorithm with correction, i.e. Algorithm \ref{algo:EIGAC}. We provide two versions of EIGAC with default coefficients as described in sec. \ref{sec:effect-rho} and the coefficients learned by Algorithm \ref{algo:SEPM}. The numerical experiments effectively show that the EIGAC with default coefficients are sufficient to converge and the performance is comparable with NAG, while EIGAC with learned coefficients is superior over other methods.
\end{itemize}

\subsection{Exemplary cases}
To get a detailed observation of the behavior of each method, we show the performance of them in exemplary cases. For each problem, we generate one testing minimization problem using the method described in sec. \ref{sec:methodology}.

In Figures \ref{fig:logistic} and \ref{fig:lpp}, the gradient norm versus the number of iterations for logistic regression and $\ell_{p}^{p}$ minimization are plotted, respectively. To emphasize the improvement obtained from the training process, these methods are terminated when the gradient norm is smaller than $3\times10^{-4}$ or the number of iterations is larger than $300$. The threshold $3\times10^{-4}$ is plotted using a red dash line. In the hard cases \texttt{lpp\_phishing} and \texttt{logistic\_phishing}, EIGAC also saves at least half of the iterations. These exemplary cases suggest that EIGAC is general enough and has strong potential to be used in practice.

\begin{figure}[htbp]
    \begin{subfigure}{.47\textwidth}
    \centering
    \includegraphics[width=\linewidth]{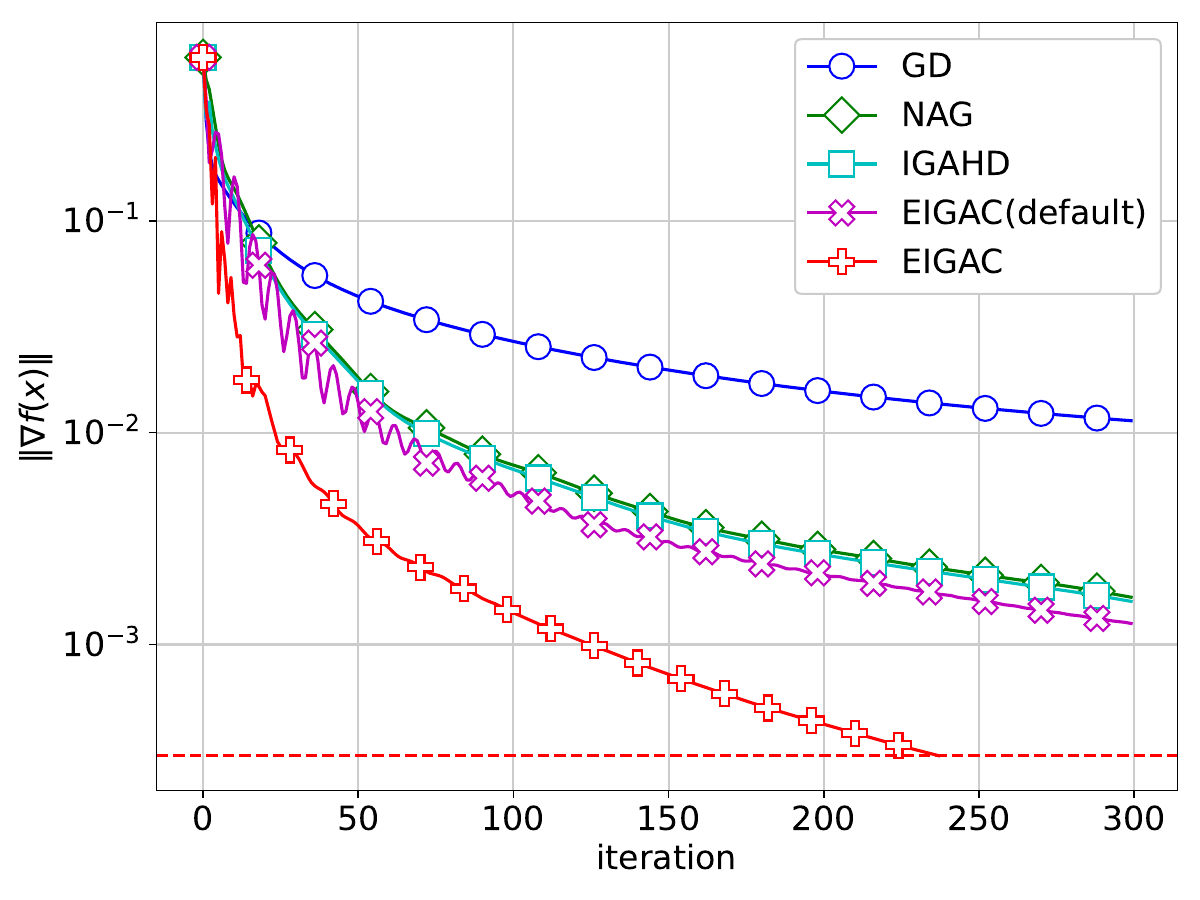}
    \caption{\texttt{a5a}}
    \label{fig:logistic-a5a}
    \end{subfigure}
    \hfill
    \begin{subfigure}{.47\textwidth}
    \centering
    \includegraphics[width=\linewidth]{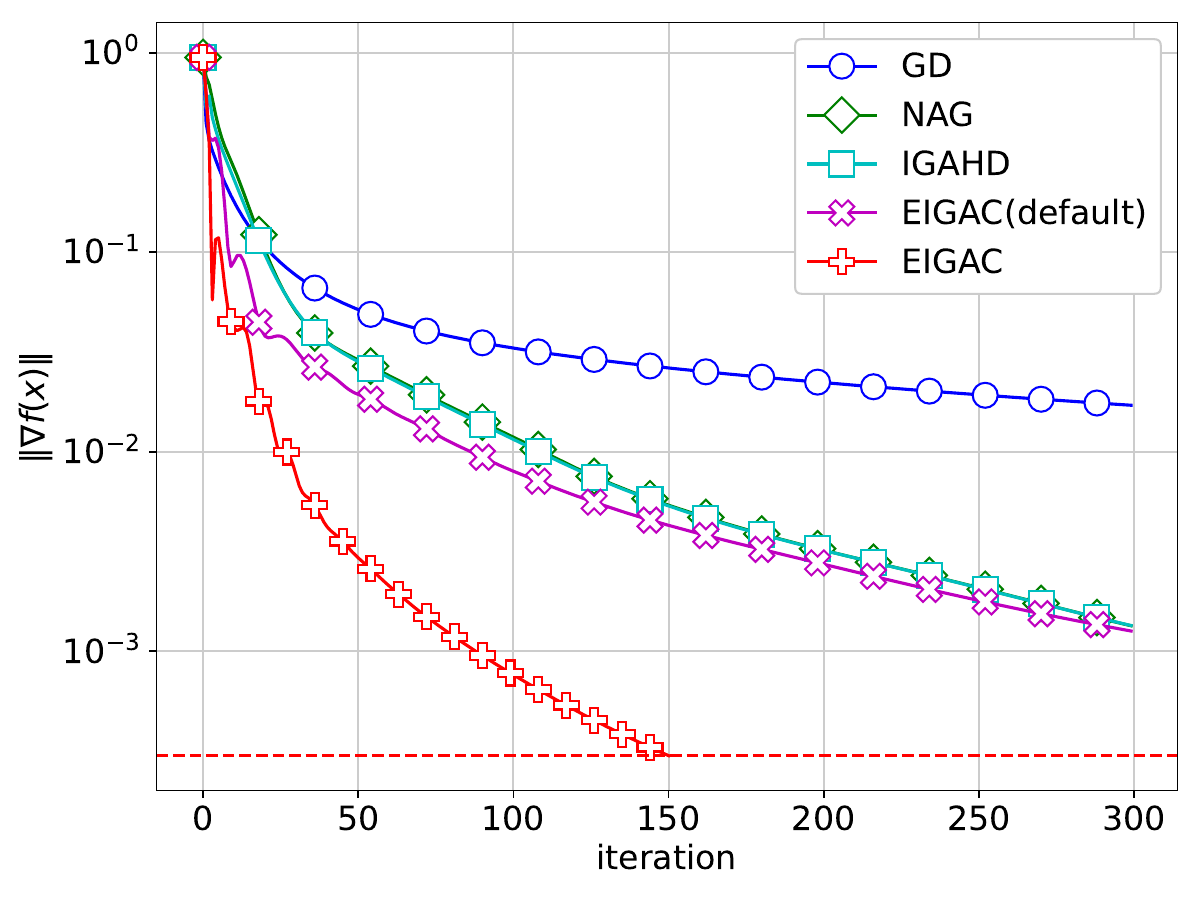}
    \caption{\texttt{mushrooms}}
    \label{fig:logistic-mushrooms}
    \end{subfigure}
    \begin{subfigure}{.47\textwidth}
    \centering
    \includegraphics[width=\linewidth]{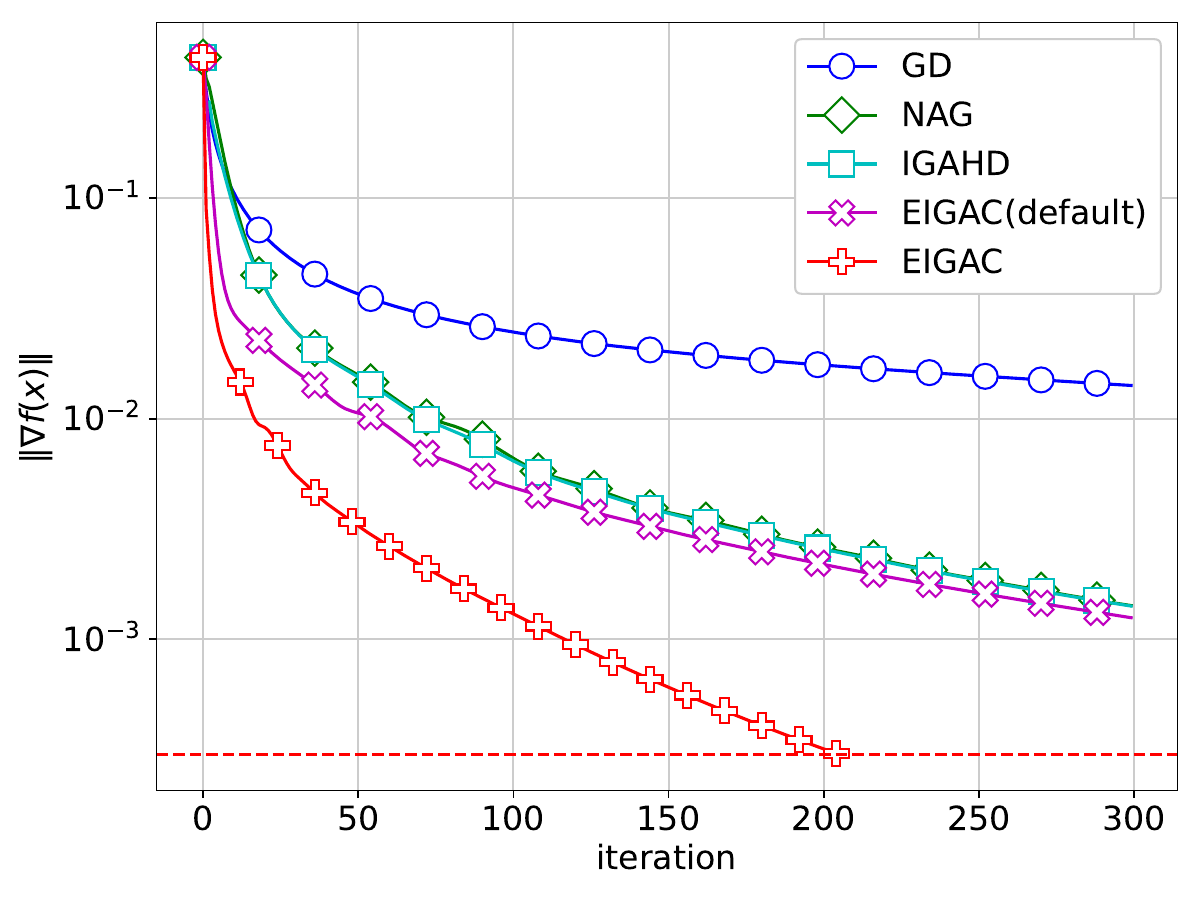}
    \caption{\texttt{w3a}}
    \label{fig:logistic-w3a}
    \end{subfigure}
    \hfill
    \begin{subfigure}{.47\textwidth}
    \centering
    \includegraphics[width=\linewidth]{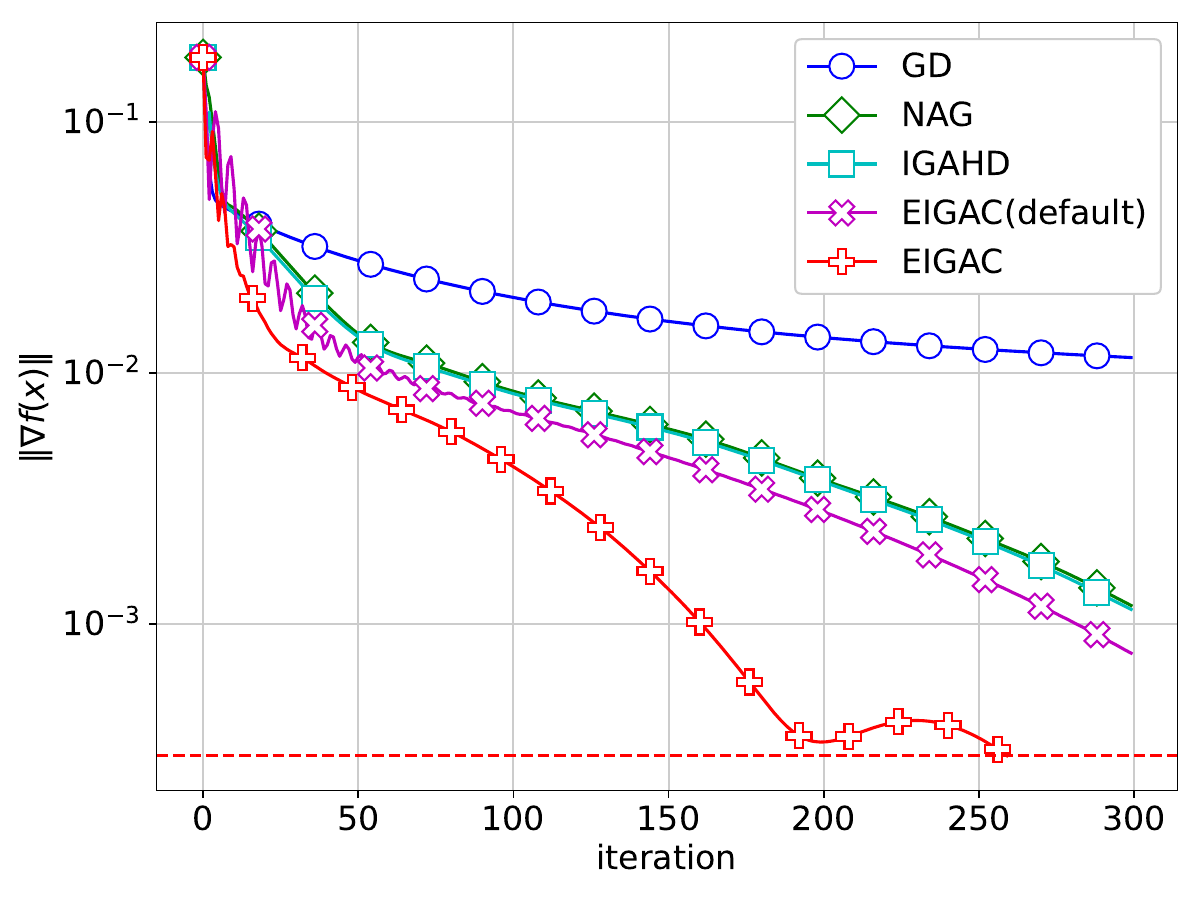}
    \caption{\texttt{covtype}}
    \label{fig:logistic-covtype}
    \end{subfigure}
    \begin{subfigure}{.47\textwidth}
    \centering
    \includegraphics[width=\linewidth]{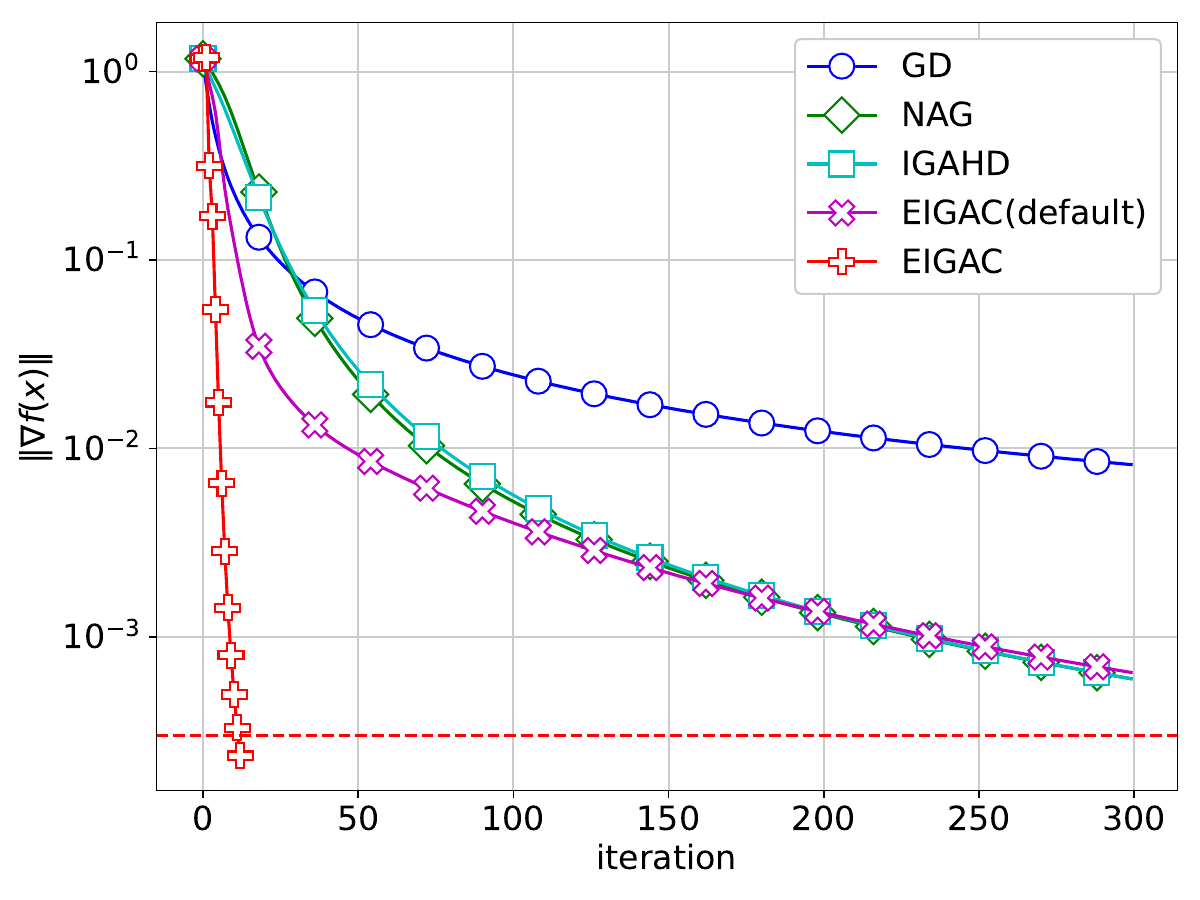}
    \caption{\texttt{separable}}
    \label{fig:logistic-separable}
    \end{subfigure}
    \hfill
    \begin{subfigure}{.47\textwidth}
    \centering
    \includegraphics[width=\linewidth]{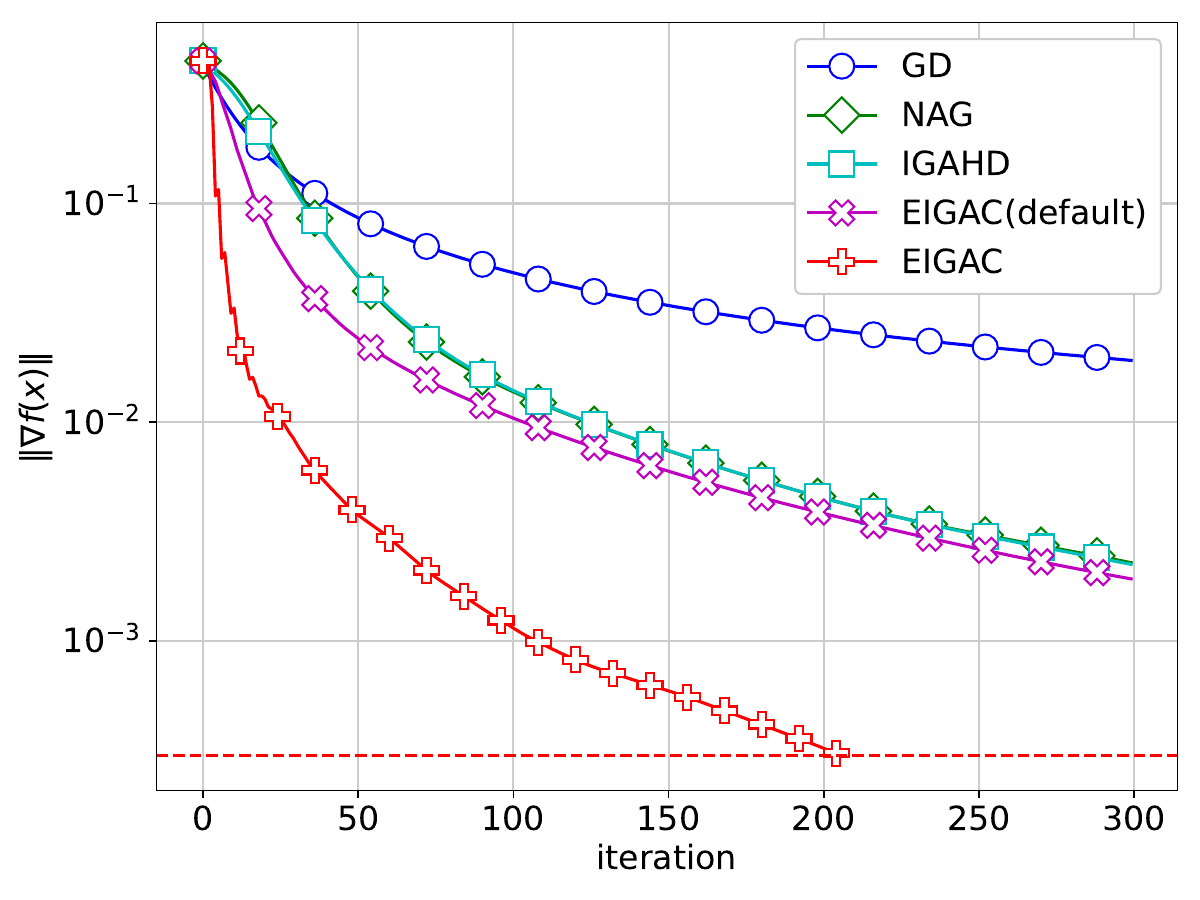}
    \caption{\texttt{phishing}}
    \label{fig:logistic-phishing}
    \end{subfigure}
    \caption{Comparison on logistic regression.\label{fig:logistic}}
\end{figure}
\begin{figure}[htbp]
    \begin{subfigure}{.47\textwidth}
    \centering
    \includegraphics[width=\linewidth]{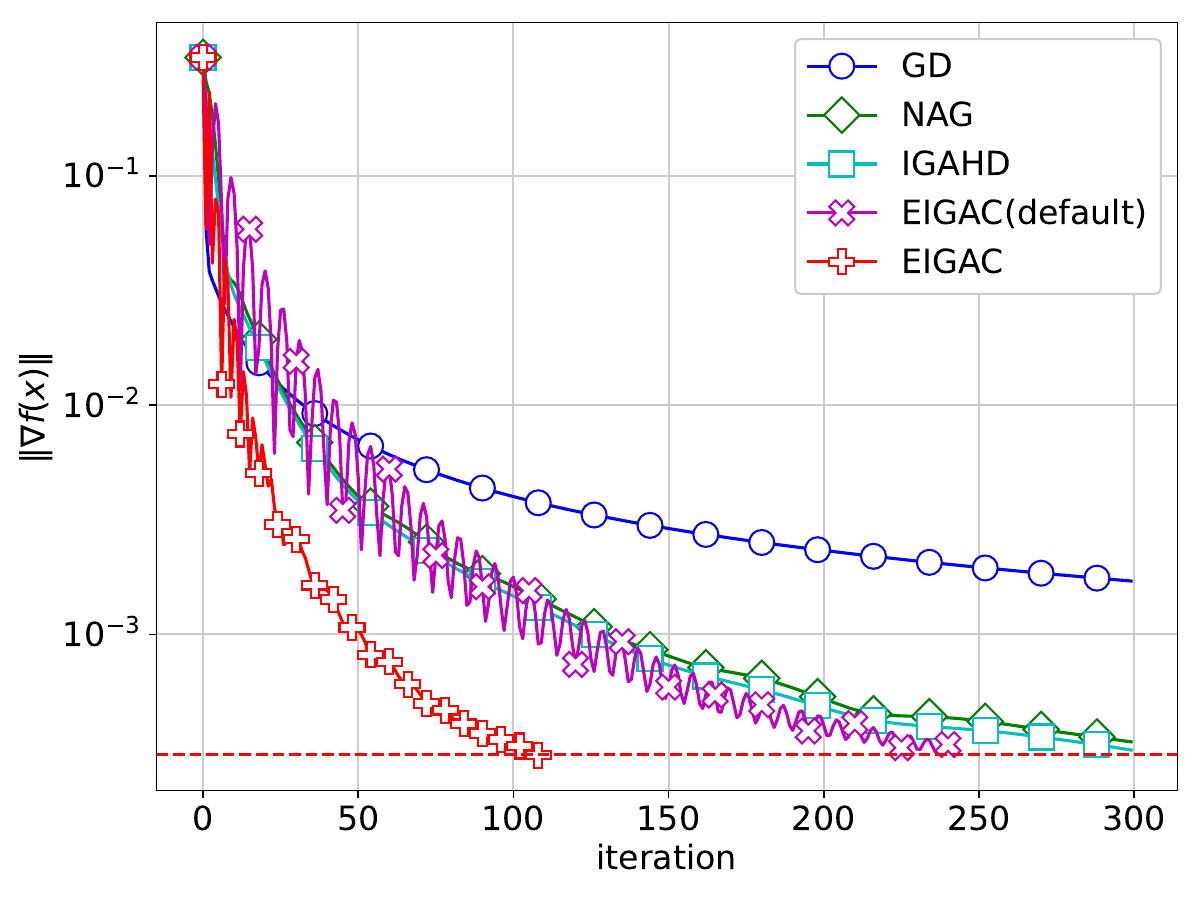}
    \caption{\texttt{a5a}}
    \label{fig:lpp-a5a}
    \end{subfigure}
    \hfill
    \begin{subfigure}{.47\textwidth}
    \centering
    \includegraphics[width=\linewidth]{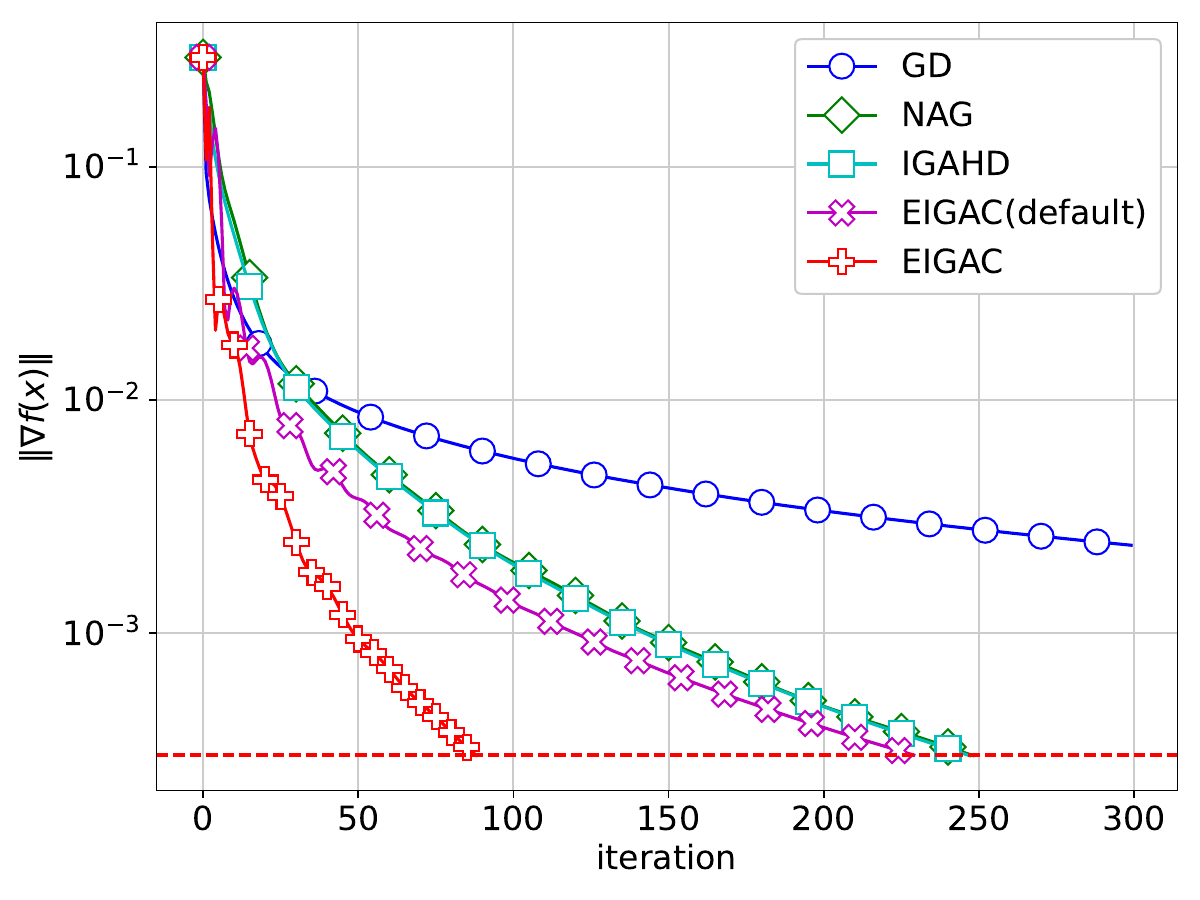}
    \caption{\texttt{mushrooms}}
    \label{fig:lpp-mushrooms}
    \end{subfigure}
    \begin{subfigure}{.47\textwidth}
    \centering
    \includegraphics[width=\linewidth]{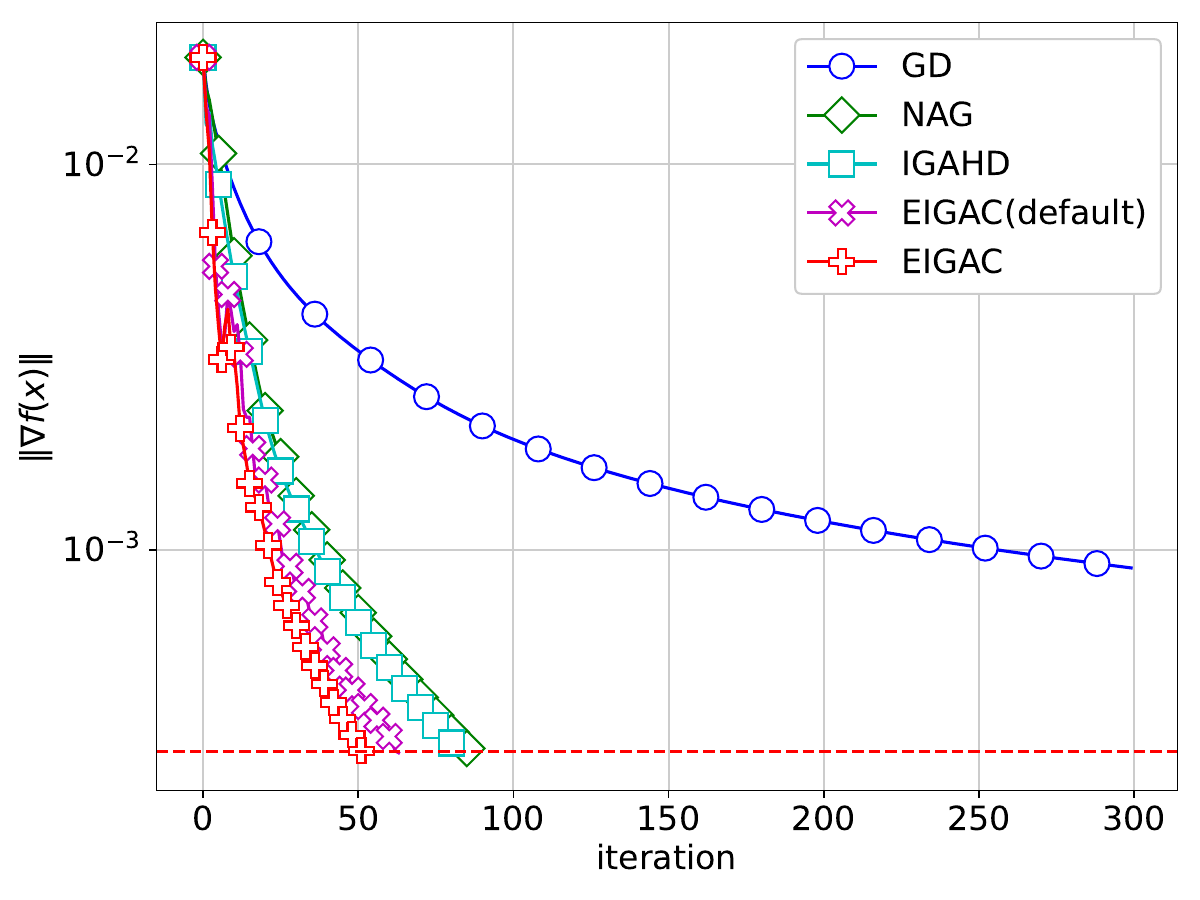}
    \caption{\texttt{w3a}}
    \label{fig:lpp-w3a}
    \end{subfigure}
    \hfill
    \begin{subfigure}{.47\textwidth}
    \centering
    \includegraphics[width=\linewidth]{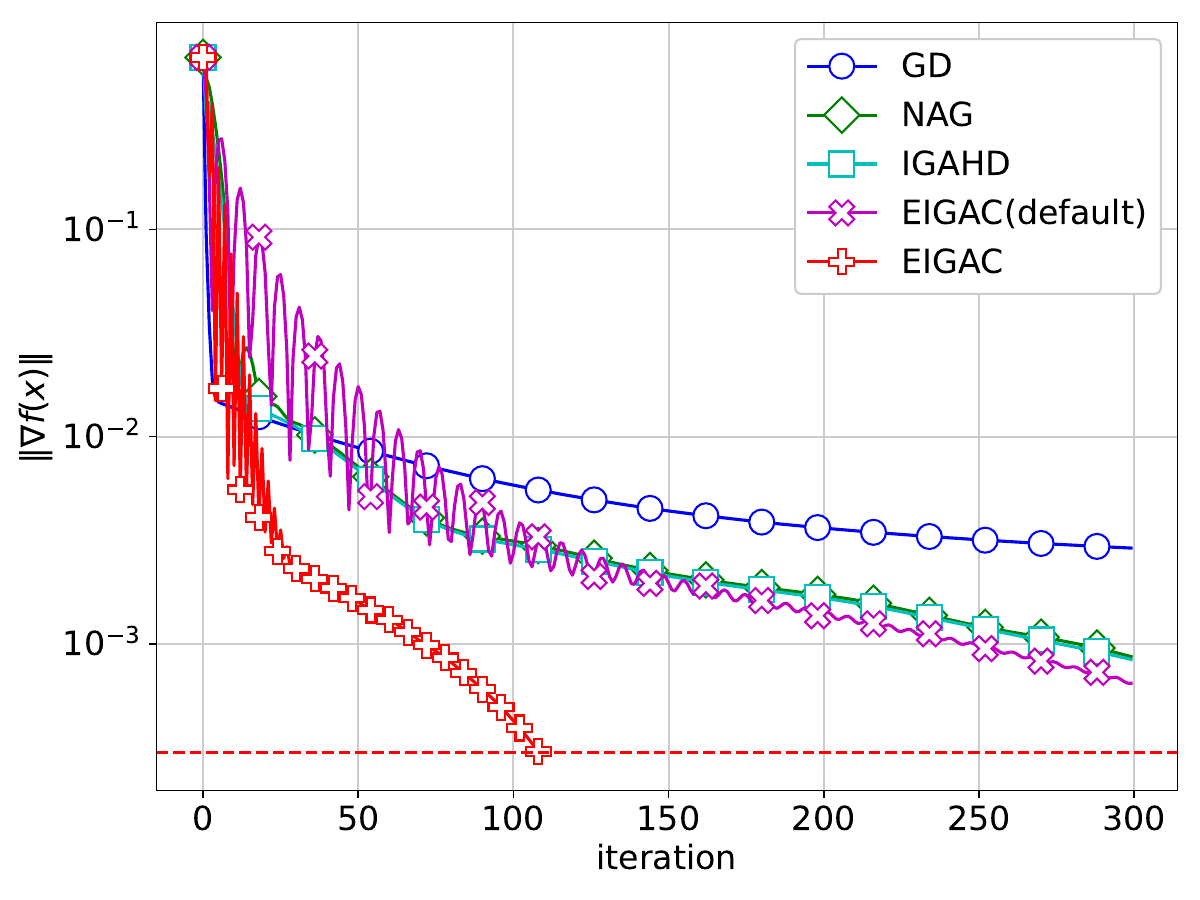}
    \caption{\texttt{covtype}}
    \label{fig:lpp-covtype}
    \end{subfigure}
    \begin{subfigure}{.47\textwidth}
    \centering
    \includegraphics[width=\linewidth]{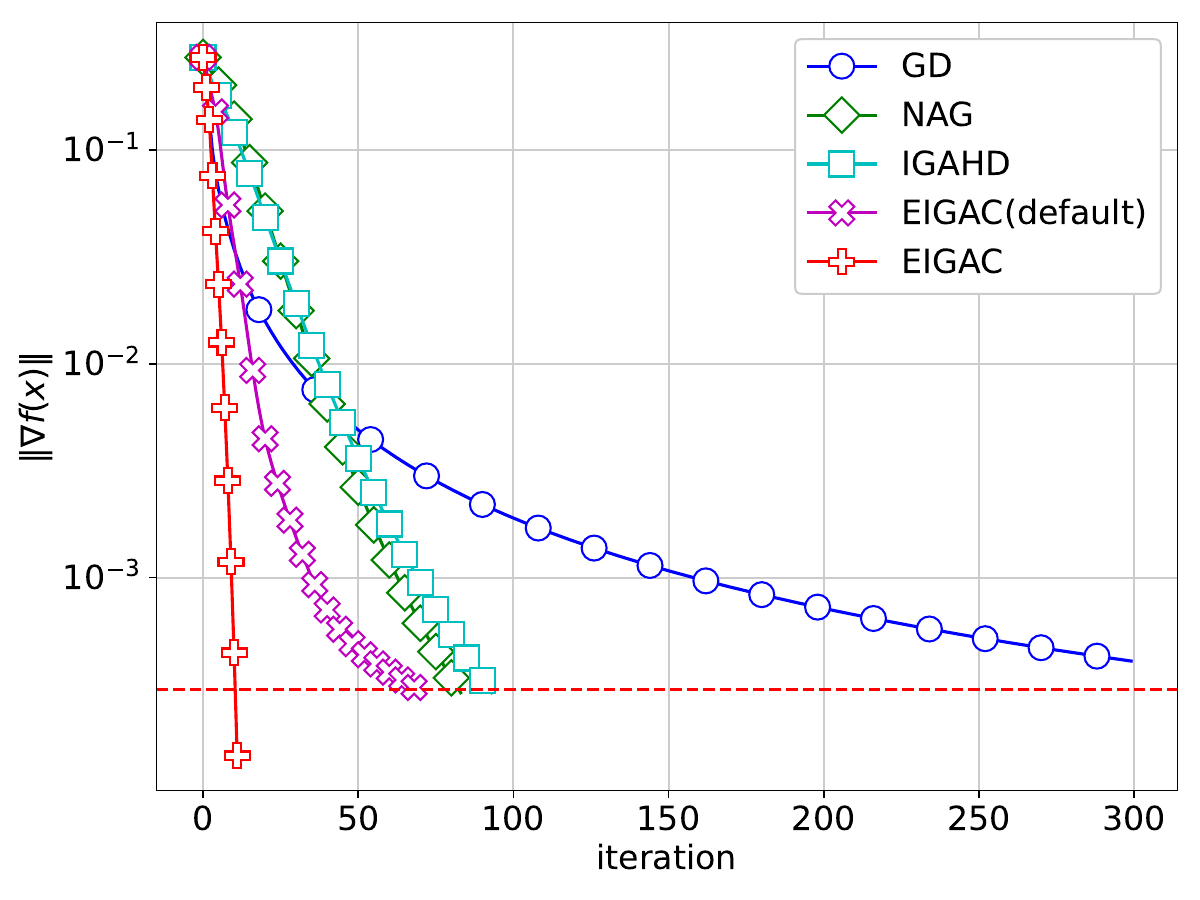}
    \caption{\texttt{separable}}
    \label{fig:lpp-separable}
    \end{subfigure}
    \hfill
    \begin{subfigure}{.47\textwidth}
    \centering
    \includegraphics[width=\linewidth]{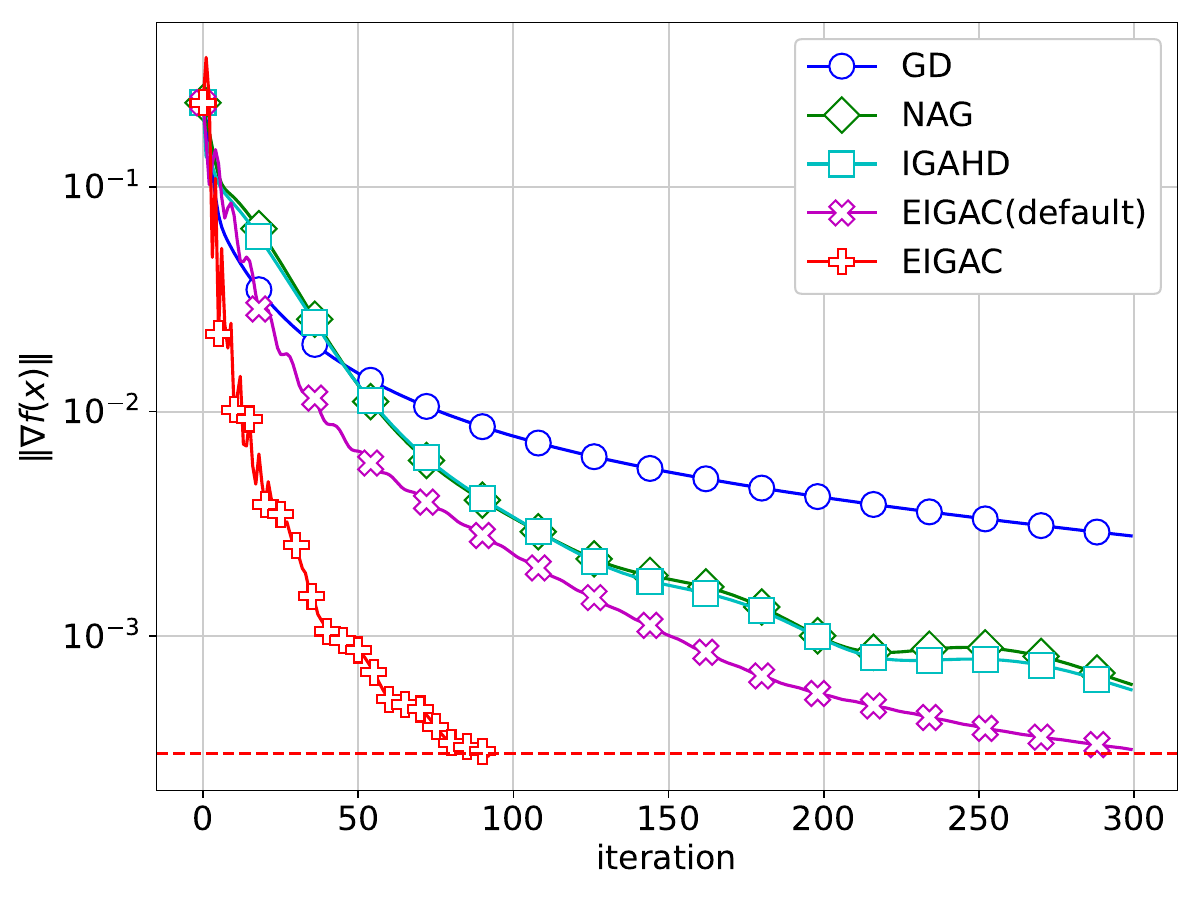}
    \caption{\texttt{phishing}}
    \label{fig:lpp-phishing}
    \end{subfigure}
    \caption{Comparison on $\ell_p^p$ Minimization.\label{fig:lpp}}
\end{figure}

Figure \ref{fig:func-compare} presents the function value versus the number of iterations of four different minimization problems. In these problems, we do not terminate the methods and show how our relay strategy helps stabilize the behavior of relay EIGAC. The results demonstrate that EIGAC surpasses other methods, even in the stage that not specifically trained for it. In the easiest case \texttt{logistic\_phishing}, EIGAC achieves the optimal value much faster than others. In Figure \ref{fig:func-lpp-phishing}, after achieving the optimal value, EIGAC does not diverge and oscillate. This shows the effectiveness of our relay strategy.

\begin{figure}[htbp]
    \begin{subfigure}{.47\textwidth}
    \centering
    \includegraphics[width=\linewidth]{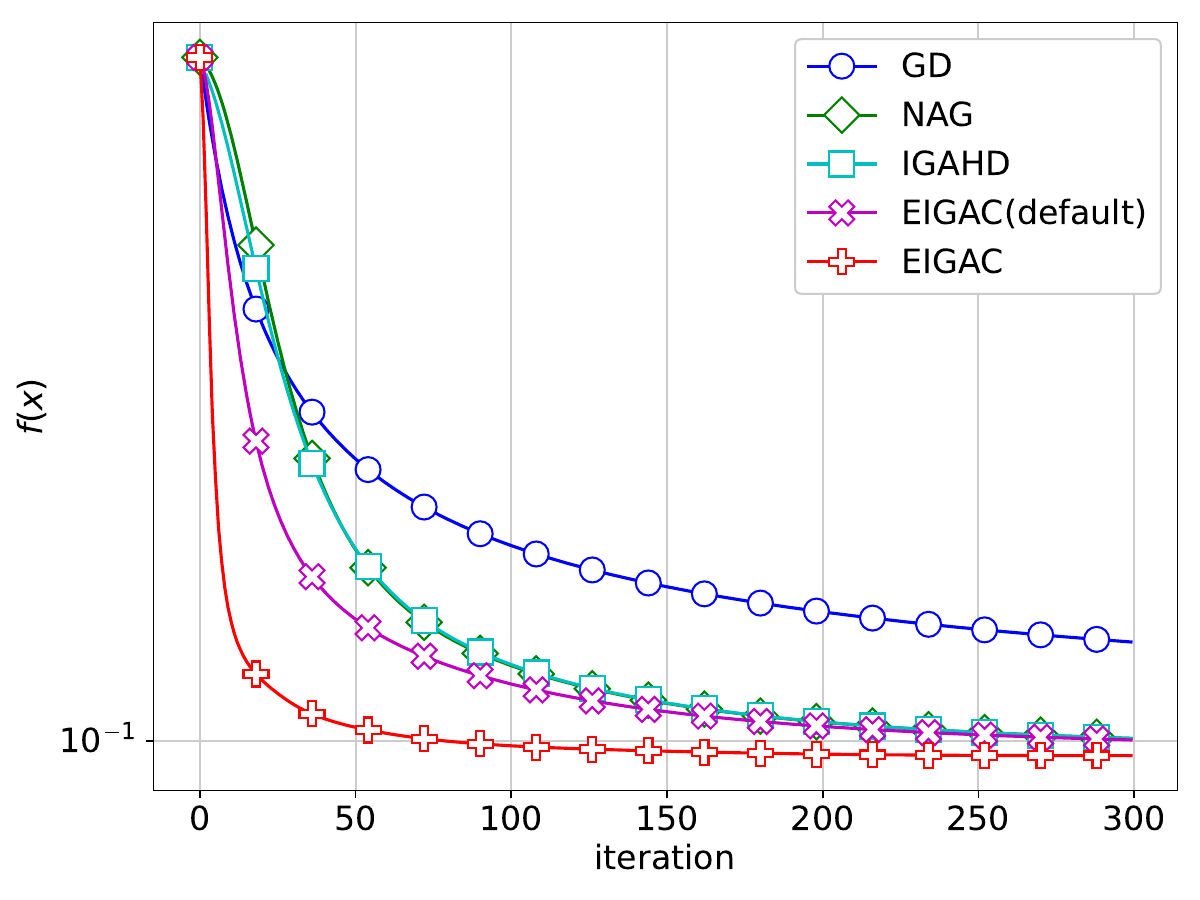}
    \caption{Logistic regression on \texttt{phishing}}
    \label{fig:func-logistic-phishing}
    \end{subfigure}
    \hfill
    \begin{subfigure}{.47\textwidth}
    \centering
    \includegraphics[width=\linewidth]{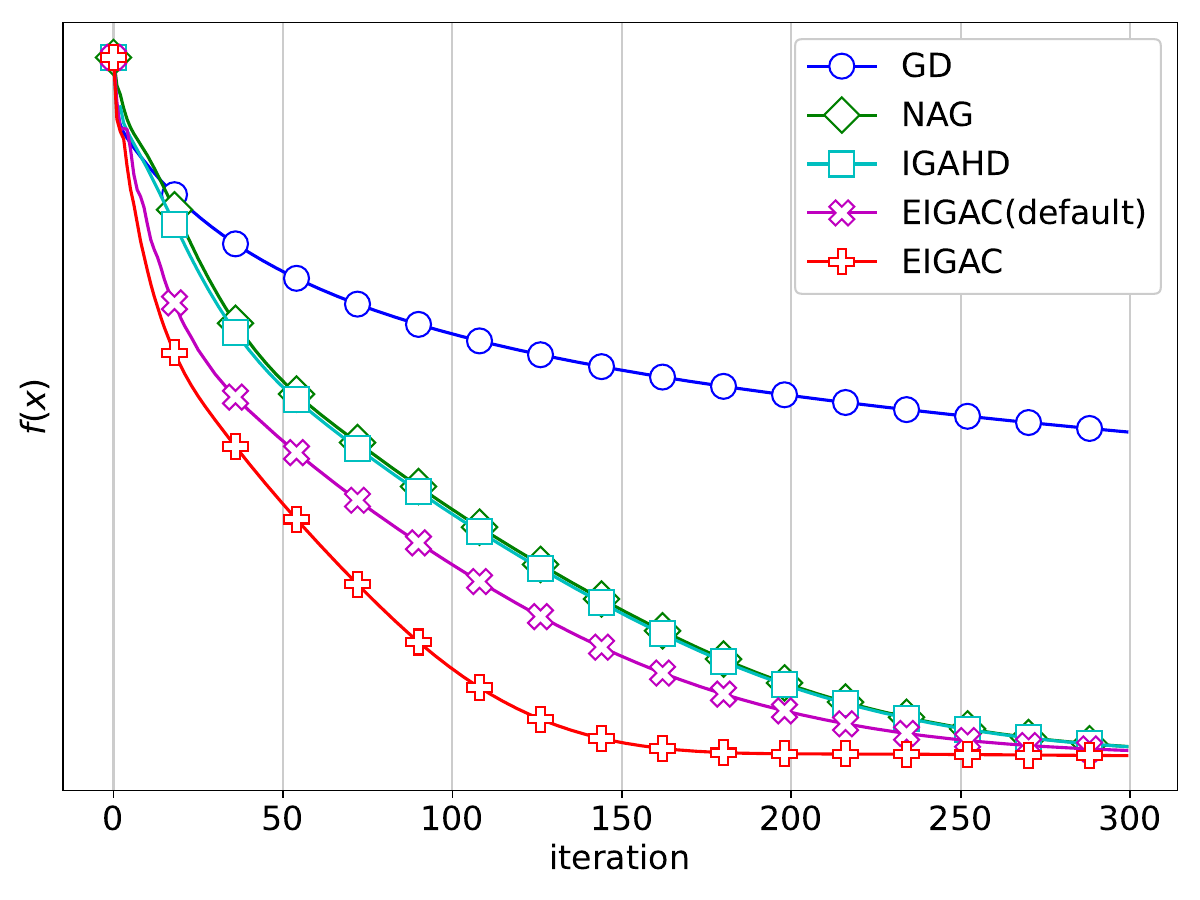}
    \caption{Logistic regression on \texttt{covtype}}
    \label{fig:func-logistic-covtype}
    \end{subfigure}
    \begin{subfigure}{.47\textwidth}
    \centering
    \includegraphics[width=\linewidth]{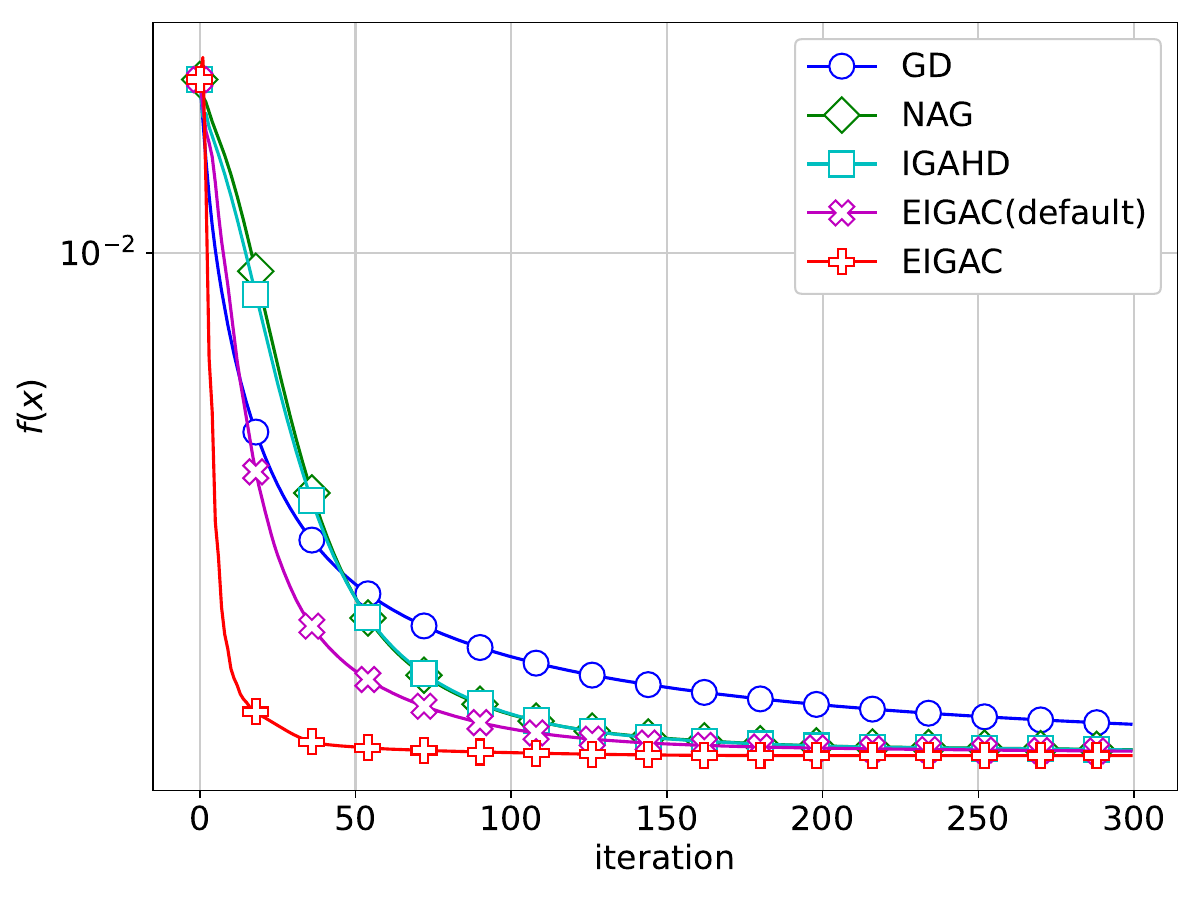}
    \caption{$\ell_{p}^{p}$ minimization on \texttt{phishing}}
    \label{fig:func-lpp-phishing}
    \end{subfigure}
    \hfill
    \begin{subfigure}{.47\textwidth}
    \centering
    \includegraphics[width=\linewidth]{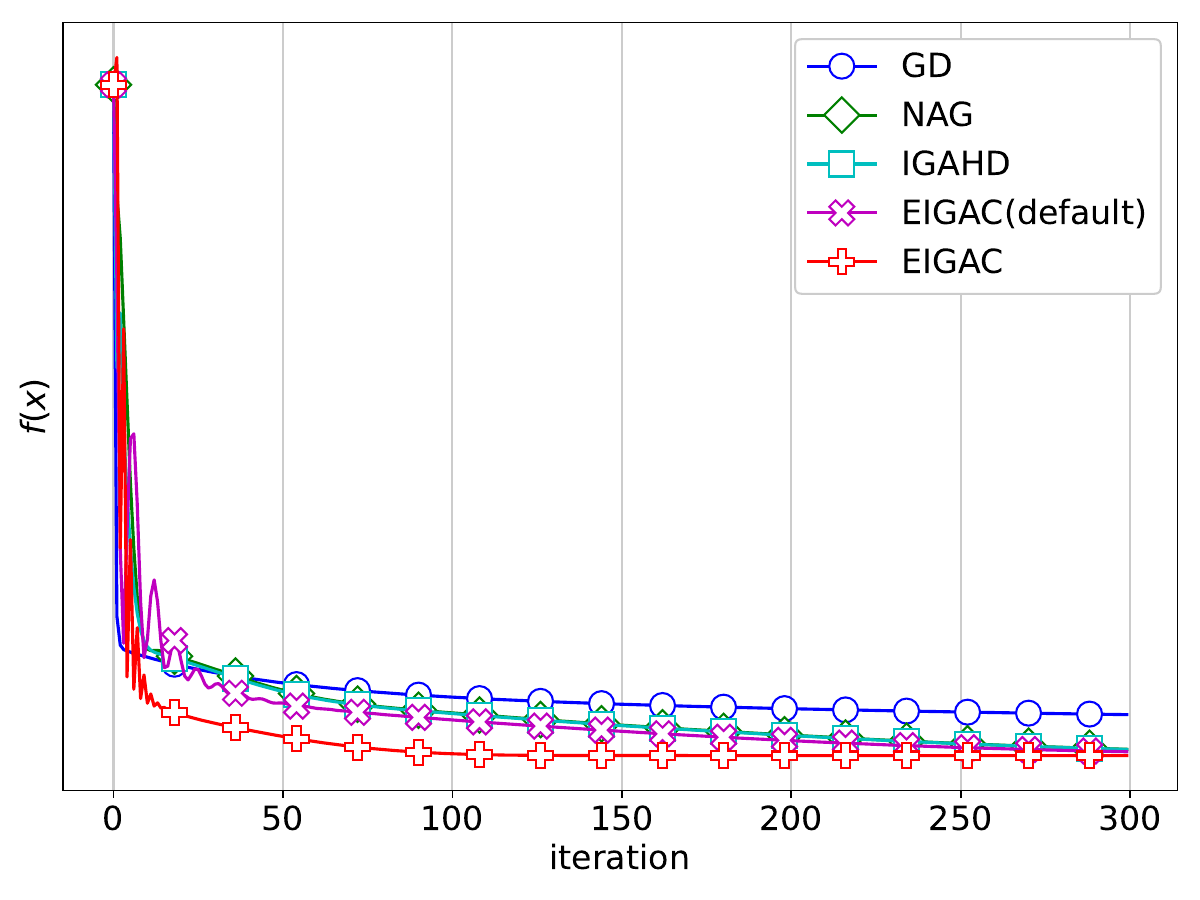}
    \caption{$\ell_{p}^{p}$ minimization on \texttt{covtype}}
    \label{fig:func-lpp-covtype}
    \end{subfigure}
    \caption{Comparison on logistic regression.\label{fig:func-compare}}
\end{figure}

\subsection{Averaged statistics}
\label{sec:avg-stat}
In this subsection, we randomly generate 100 testing minimization problems, denoted as \( \mathcal{F}_{\mathrm{test}} = \{f_i\}_{i=1}^{100} \), for each type of minimization problem using different datasets. The construction methodology follows that described in sec. \ref{sec:methodology}. A problem is defined by the dataset and the type of the minimization problem. For example, the problem \texttt{lpp\_a5a} refers to the \texttt{a5a} dataset with the \( \ell_{p}^{p} \) minimization formulation. Multiple test functions for the same problem can be generated by varying the instances from the dataset used.

Given the type of the minimization problem and the dataset, for each method, we provide two statistics. The first one is the averaged performance measure at the $N$-th iteration:
\[
    \Bar{m}(\mathcal{F}_{\mathrm{test}})=\frac{1}{|\mathcal{F}_{\mathrm{test}}|}\sum_{f\in\mathcal{F}_{\mathrm{test}}}\log\|\nabla f(x_{N})\|.
\]
We set $N=500$ for comparison. The second one is the averaged complexity:
\[
    \Bar{N}(\mathcal{F}_{\mathrm{test}})=\frac{1}{|\mathcal{F}_{\mathrm{test}}|}\sum_{f\in\mathcal{F}_{\mathrm{test}}}\inf\{n\mid \|f(x_{n})\|\leq \varepsilon\}
\]
with $\varepsilon=3\times 10^{-4}$. If the method does not reaching the threshold, we denote its complexity in this test as $500$.

The averaged performance measure with the stand error in the parenthesis is reported in Table \ref{tab:logistic-measure} and \ref{tab:lpp-measure}. The \texttt{Win Rate} row represents the rate of EIGAC outperforms other methods. In these experiments, we observed that the performance of NAG, IGAHD, and EIGAC(default) is similar, both in performance measure and complexity. The non divergence of EIGAC also verifies the effectiveness of the relay strategy. The result shows that EIGAC outperforms other methods with at least a magnitude in most cases. In the hard cases \texttt{logistic\_covtype} and \texttt{lpp\_covtype}, the improvement is still significant and shows the potential of the learning-based method.
\begin{table}[htbp]
    \centering
    \caption{The averaged performance measure in logistic regression problems.}
    \label{tab:logistic-measure}
    \resizebox{\textwidth}{!}{%
    \begin{tabular}{ccccccc}
        \toprule
        & mushrooms & a5a & w3a & phishing & covtype & separable \\
        \midrule
        GD & -1.55(0.016) & -1.81(0.031) & -1.90(0.022) & -1.35(0.007) & -1.89(0.008) & -1.56(0.000) \\
        NAG & -3.37(0.027) & -3.11(0.059) & -3.26(0.061) & -3.01(0.074) & -3.07(0.012) & -3.66(0.000) \\
        EIGAC(default) & -3.02(0.020) & -2.97(0.065) & -3.02(0.044) & -2.80(0.064) & -3.48(0.059) & -3.32(0.000) \\
        IGAHD & -3.02(0.020) & -2.97(0.065) & -3.02(0.043) & -2.80(0.064) & -3.48(0.059) & -3.31(0.000) \\
        EIGAC & -4.83(0.017) & -4.38(0.062) & -4.46(0.093) & -4.82(0.120) & -4.37(0.055) & -5.49(0.090) \\
        Win Rate & 100.00\% & 100.00\% & 100.00\% & 100.00\% & 100.00\% & 100.00\% \\
        \bottomrule
        \end{tabular}
        }
\end{table}

\begin{table}
    \centering
    \caption{The averaged performance measure in $\ell_{p}^{p}$ minimization problems.}
    \label{tab:lpp-measure}
    \resizebox{\textwidth}{!}{
        \begin{tabular}{ccccccc}
            \toprule
            & mushrooms & a5a & w3a & phishing & covtype & separable \\
            \midrule
            GD & -2.49(0.018) & -2.79(0.042) & -3.18(0.053) & -2.36(0.015) & -2.65(0.008) & -2.95(0.001) \\
            NAG & -4.35(0.033) & -4.19(0.058) & -4.72(0.085) & -4.06(0.052) & -4.43(0.091) & -6.11(0.077) \\
            EIGAC(default) & -4.16(0.022) & -3.99(0.069) & -4.66(0.085) & -4.37(0.147) & -4.47(0.043) & -6.15(0.048) \\
            IGAHD & -4.16(0.022) & -4.05(0.068) & -4.66(0.085) & -4.37(0.150) & -4.51(0.037) & -6.14(0.049) \\
            EIGAC & -5.27(0.049) & -5.11(0.081) & -5.71(0.145) & -5.65(0.091) & -5.14(0.110) & -7.55(0.044) \\
            Win Rate & 100.00\% & 100.00\% & 100.00\% & 100.00\% & 100.00\% & 100.00\% \\
            \bottomrule
            \end{tabular}
    }
\end{table}

In Tables \ref{tab:logistic-complexity} and \ref{tab:lpp-complexity}, we present the averaged complexity for logistic regression and \( \ell_{p}^{p} \) minimization problems, respectively. Since the value of \( \log_{10}(3 \times 10^{-4}) \) is approximately \(-3.52\), GD, EIGAC, and IGAHD do not reach the threshold in most logistic regression problems within 500 iterations, as indicated by the performance measure statistics. In most logistic regression problems, the averaged complexity of EIGAC is only half that of other methods. In the easiest case, \texttt{logistic\_separable}, EIGAC requires just \( \frac{1}{40} \) of the iterations needed by NAG to reach the threshold. In the relatively easier \( \ell_{p}^{p} \) minimization problems, EIGAC also shows a consistent improvement in complexity. While other methods achieve the threshold quickly, the averaged complexity of EIGAC is approximately half that of the others in each problem.

\begin{table}
    \centering
    \caption{The averaged complexity in logistic regression problems.}
    \label{tab:logistic-complexity}
    \resizebox{\textwidth}{!}{
        \begin{tabular}{ccccccc}
            \toprule
            & mushrooms & a5a & w3a & phishing & covtype & separable \\
            \midrule
            GD & 500.00(0.000) & 500.00(0.000) & 500.00(0.000) & 500.00(0.000) & 500.00(0.000) & 500.00(0.000) \\
            NAG & 500.00(0.000) & 500.00(0.000) & 500.00(0.000) & 500.00(0.000) & 500.00(0.000) & 424.71(0.456) \\
            EIGAC(default) & 500.00(0.000) & 500.00(0.000) & 500.00(0.000) & 500.00(0.000) & 497.12(5.741) & 500.00(0.000) \\
            IGAHD & 500.00(0.000) & 500.00(0.000) & 500.00(0.000) & 500.00(0.000) & 497.36(5.509) & 500.00(0.000) \\
            EIGAC & 153.48(2.525) & 227.32(17.759) & 216.42(16.462) & 182.15(21.568) & 237.88(34.565) & 11.49(0.502) \\
            Win Rate & 100.00\% & 100.00\% & 100.00\% & 100.00\% & 100.00\% & 100.00\% \\
            \bottomrule
            \end{tabular}
    }
\end{table}

\begin{table}
    \centering
    \caption{The averaged complexity in $\ell_{p}^{p}$ minimization problems.}
    \label{tab:lpp-complexity}
    \resizebox{\textwidth}{!}{
        \begin{tabular}{ccccccc}
            \toprule
            & mushrooms & a5a & w3a & phishing & covtype & separable \\
            \midrule
            GD & 500.00(0.000) & 500.00(0.000) & 500.00(0.000) & 500.00(0.000) & 500.00(0.000) & 500.00(0.000) \\
            NAG & 183.77(4.583) & 211.87(17.323) & 92.61(7.590) & 252.43(16.373) & 167.31(1.650) & 52.15(0.359) \\
            EIGAC(default) & 235.07(6.331) & 245.68(26.391) & 96.17(9.527) & 224.36(21.168) & 203.02(2.040) & 22.10(0.302) \\
            IGAHD & 235.84(6.419) & 239.49(27.857) & 96.03(9.682) & 224.53(21.373) & 204.12(1.996) & 29.98(0.141) \\
            EIGAC & 93.12(2.124) & 122.16(11.176) & 50.93(4.529) & 85.57(5.385) & 109.15(0.999) & 11.00(0.000) \\
            Win Rate & 100.00\% & 100.00\% & 100.00\% & 100.00\% & 100.00\% & 100.00\% \\
            \bottomrule
            \end{tabular}
    }
\end{table}

\section{Conclusion and future directions}
\label{sec:conclusion}
In this paper, we have introduced a framework that integrates ISHD with L2O to develop efficient optimization methods. By establishing convergence conditions and analyzing the stability of discretization schemes, we give a large set of effective optimization algorithms with solid theoretical foundation. The introduction of the learning problem, which minimizes the stopping time subject to convergence and stability constraints, marks a noteworthy step forward in L2O. Our approach, which employs penalty methods, stochastic optimization, and conservative gradients, is designed to effectively address the learning problem. The convergence of the learning process and the promising performance of the learned optimization methods, supported by extensive numerical experiments, suggest the robustness and practical potential of our framework.

However, the analysis in our framework may not be optimal. As shown in the numerical experiments, the EIGAC with default coefficients also surpasses the NAG and IGAHD, despite the former having a theoretical $\mathcal{O}(1/k)$ convergence rate based on our analysis. We aim to improve the theoretical results to better align with the numerical experiments in future work. Another future direction of this work is to generalize our methodology to other optimization problems, such as non-smooth optimization, constrained optimization, and non-convex optimization. 

\bibliographystyle{spmpsci}
\bibliography{main}

\begin{thebibliography}{10}
\providecommand{\url}[1]{{#1}}
\providecommand{\urlprefix}{URL }
\expandafter\ifx\csname urlstyle\endcsname\relax
  \providecommand{\doi}[1]{DOI~\discretionary{}{}{}#1}\else
  \providecommand{\doi}{DOI~\discretionary{}{}{}\begingroup
  \urlstyle{rm}\Url}\fi

\bibitem{ahmad2020plugandplay}
Ahmad, R., Bouman, C.A., Buzzard, G.T., Chan, S., Liu, S., Reehorst, E.T.,
  Schniter, P.: Plug-and-play methods for magnetic resonance imaging: Using
  denoisers for image recovery.
\newblock IEEE Signal Processing Magazine \textbf{37}(1), 105--116 (2020).
\newblock \doi{10.1109/MSP.2019.2949470}

\bibitem{2016l2obygd}
Andrychowicz, M., Denil, M., Colmenarejo, S.G., Hoffman, M.W., Pfau, D.,
  Schaul, T., de~Freitas, N.: Learning to learn by gradient descent by gradient
  descent.
\newblock In: D.D. Lee, M.~Sugiyama, U.~von Luxburg, I.~Guyon, R.~Garnett
  (eds.) Advances in Neural Information Processing Systems 29: Annual
  Conference on Neural Information Processing Systems 2016, December 5-10,
  2016, Barcelona, Spain, pp. 3981--3989 (2016).
\newblock
  \urlprefix\url{https://proceedings.neurips.cc/paper/2016/hash/fb87582825f9d28a8d42c5e5e5e8b23d-Abstract.html}

\bibitem{attouch2022fast}
Attouch, H.: Fast algorithmic methods for optimization and learning. {R}ecent
  trends. (2022)

\bibitem{attouch2023closed}
Attouch, H., Bo\c{t}, R.I., Csetnek, E.R.: Fast optimization via inertial
  dynamics with closed-loop damping.
\newblock J. Eur. Math. Soc. (JEMS) \textbf{25}(5), 1985--2056 (2023)

\bibitem{attouchFirstorderOptimizationAlgorithms2020}
Attouch, H., Chbani, Z., Fadili, J., Riahi, H.: First-order optimization
  algorithms via inertial systems with {H}essian-driven damping.
\newblock Math. Program. \textbf{193}(1), 113--155 (2022)

\bibitem{attouch2014dynamical}
Attouch, H., Peypouquet, J., Redont, P.: A dynamical approach to an inertial
  forward-backward algorithm for convex minimization.
\newblock SIAM J. Optim. \textbf{24}(1), 232--256 (2014)

\bibitem{attouch2016fast}
Attouch, H., Peypouquet, J., Redont, P.: Fast convex optimization via inertial
  dynamics with {H}essian driven damping.
\newblock J. Differential Equations \textbf{261}(10), 5734--5783 (2016).
\newblock \doi{10.1016/j.jde.2016.08.020}

\bibitem{aubin1984differential}
Aubin, J.P., Cellina, A.: Differential inclusions, \emph{Grundlehren der
  mathematischen Wissenschaften [Fundamental Principles of Mathematical
  Sciences]}, vol. 264.
\newblock Springer-Verlag, Berlin (1984).
\newblock Set-valued maps and viability theory

\bibitem{aujolFastConvergenceInertial2022}
Aujol, J.F., Dossal, C., Ho\`ang, V.H., Labarri\`ere, H., Rondepierre, A.: Fast
  convergence of inertial dynamics with {H}essian-driven damping under geometry
  assumptions.
\newblock Appl. Math. Optim. \textbf{88}(3), 81 (2023)

\bibitem{aujolOptimalConvergenceRates2019}
Aujol, J.F., Dossal, C., Rondepierre, A.: Optimal convergence rates for
  {N}esterov acceleration.
\newblock SIAM J. Optim. \textbf{29}(4), 3131--3153 (2019)

\bibitem{aujolConvergenceRatesHeavyBall2022}
Aujol, J.F., Dossal, C., Rondepierre, A.: Convergence rates of the heavy-ball
  method under the {L}ojasiewicz property.
\newblock Math. Program. \textbf{198}(1), 195--254 (2023)

\bibitem{aujol2023fista}
Aujol, J.F., Dossal, C., Rondepierre, A.: {FISTA} is an automatic geometrically
  optimized algorithm for strongly convex functions.
\newblock Math. Program. pp. 1--43 (2023)

\bibitem{pang2022equilibrium}
Ba, Q., Pang, J.S.: Exact penalization of generalized {N}ash equilibrium
  problems.
\newblock Oper. Res. \textbf{70}(3), 1448--1464 (2022).
\newblock \doi{10.1287/opre.2019.1942}.
\newblock \urlprefix\url{https://doi.org/10.1287/opre.2019.1942}

\bibitem{banertDataDrivenNonsmoothOptimization2020}
Banert, S., Ringh, A., Adler, J., Karlsson, J., \"{O}ktem, O.: Data-driven
  nonsmooth optimization.
\newblock SIAM J. Optim. \textbf{30}(1), 102--131 (2020)

\bibitem{banert2024siopt}
Banert, S., Rudzusika, J., \"{O}ktem, O., Adler, J.: Accelerated
  forward-backward optimization using deep learning.
\newblock SIAM Journal on Optimization \textbf{34}(2), 1236--1263 (2024).
\newblock \doi{10.1137/22M1532548}.
\newblock \urlprefix\url{https://doi.org/10.1137/22M1532548}

\bibitem{michel2005sto}
Bena\"{\i}m, M., Hofbauer, J., Sorin, S.: Stochastic approximations and
  differential inclusions.
\newblock SIAM J. Control Optim. \textbf{44}(1), 328--348 (2005).
\newblock \doi{10.1137/S0363012904439301}.
\newblock \urlprefix\url{https://doi.org/10.1137/S0363012904439301}

\bibitem{bolte2021implicit}
Bolte, J., Le, T., Pauwels, E., Silveti-Falls, T.: Nonsmooth implicit
  differentiation for machine-learning and optimization.
\newblock In: M.~Ranzato, A.~Beygelzimer, Y.~Dauphin, P.~Liang, J.W. Vaughan
  (eds.) Advances in Neural Information Processing Systems, vol.~34, pp.
  13537--13549. Curran Associates, Inc. (2021)

\bibitem{bolte2021conservative}
Bolte, J., Pauwels, E.: Conservative set valued fields, automatic
  differentiation, stochastic gradient methods and deep learning.
\newblock Math. Program. \textbf{188}(1), 19--51 (2021)

\bibitem{chen2022learning}
Chen, T., Chen, X., Chen, W., Wang, Z., Heaton, H., Liu, J., Yin, W.: Learning
  to optimize: a primer and a benchmark.
\newblock J. Mach. Learn. Res. \textbf{23}, Paper No. [189], 59 (2022)

\bibitem{chen2023lion}
Chen, X., Liang, C., Huang, D., Real, E., Wang, K., Pham, H., Dong, X., Luong,
  T., Hsieh, C., Lu, Y., Le, Q.V.: Symbolic discovery of optimization
  algorithms.
\newblock In: A.~Oh, T.~Naumann, A.~Globerson, K.~Saenko, M.~Hardt, S.~Levine
  (eds.) Advances in Neural Information Processing Systems 36: Annual
  Conference on Neural Information Processing Systems 2023, NeurIPS 2023, New
  Orleans, LA, USA, December 10 - 16, 2023 (2023).
\newblock
  \urlprefix\url{http://papers.nips.cc/paper\_files/paper/2023/hash/9a39b4925e35cf447ccba8757137d84f-Abstract-Conference.html}

\bibitem{liu2018theoreticallinear}
Chen, X., Liu, J., Wang, Z., Yin, W.: Theoretical linear convergence of
  unfolded {ISTA} and its practical weights and thresholds.
\newblock In: S.~Bengio, H.M. Wallach, H.~Larochelle, K.~Grauman,
  N.~Cesa{-}Bianchi, R.~Garnett (eds.) Advances in Neural Information
  Processing Systems 31: Annual Conference on Neural Information Processing
  Systems 2018, NeurIPS 2018, December 3-8, 2018, Montr{\'{e}}al, Canada, pp.
  9079--9089 (2018).
\newblock
  \urlprefix\url{https://proceedings.neurips.cc/paper/2018/hash/cf8c9be2a4508a24ae92c9d3d379131d-Abstract.html}

\bibitem{2017l2owithoutgd}
Chen, Y., Hoffman, M.W., Colmenarejo, S.G., Denil, M., Lillicrap, T.P.,
  Botvinick, M.M., de~Freitas, N.: Learning to learn without gradient descent
  by gradient descent.
\newblock In: D.~Precup, Y.W. Teh (eds.) Proceedings of the 34th International
  Conference on Machine Learning, {ICML} 2017, Sydney, NSW, Australia, 6-11
  August 2017, \emph{Proceedings of Machine Learning Research}, vol.~70, pp.
  748--756. {PMLR} (2017).
\newblock \urlprefix\url{http://proceedings.mlr.press/v70/chen17e.html}

\bibitem{clarke1990nonsmooth}
Clarke, F.H.: Optimization and nonsmooth analysis, \emph{Classics in Applied
  Mathematics}, vol.~5, second edn.
\newblock Society for Industrial and Applied Mathematics (SIAM), Philadelphia,
  PA (1990).
\newblock \doi{10.1137/1.9781611971309}.
\newblock \urlprefix\url{https://doi.org/10.1137/1.9781611971309}

\bibitem{coste2000introduction}
Coste, M.: An introduction to o-minimal geometry.
\newblock Istituti editoriali e poligrafici internazionali Pisa (2000)

\bibitem{pang2022chance}
Cui, Y., Liu, J., Pang, J.S.: Nonconvex and nonsmooth approaches for affine
  chance-constrained stochastic programs.
\newblock Set-Valued Var. Anal. \textbf{30}(3), 1149--1211 (2022).
\newblock \doi{10.1007/s11228-022-00639-y}.
\newblock \urlprefix\url{https://doi.org/10.1007/s11228-022-00639-y}

\bibitem{pang2023modern}
Cui, Y., Pang, J.S.: Modern nonconvex nondifferentiable optimization,
  \emph{MOS-SIAM Series on Optimization}, vol.~29.
\newblock Society for Industrial and Applied Mathematics (SIAM), Philadelphia,
  PA; Mathematical Optimization Society, Philadelphia, PA (2022)

\bibitem{davis2020sgdtame}
Davis, D., Drusvyatskiy, D., Kakade, S., Lee, J.D.: Stochastic subgradient
  method converges on tame functions.
\newblock Found. Comput. Math. \textbf{20}(1), 119--154 (2020)

\bibitem{dries1998tame}
van~den Dries, L.: Tame topology and o-minimal structures, \emph{London
  Mathematical Society Lecture Note Series}, vol. 248.
\newblock Cambridge University Press, Cambridge (1998)

\bibitem{van1996geometric}
van~den Dries, L., Miller, C.: {Geometric categories and o-minimal structures}.
\newblock Duke Mathematical Journal \textbf{84}(2), 497 -- 540 (1996)

\bibitem{Dua:2019}
Dua, D., Graff, C.: {UCI} machine learning repository (2017)

\bibitem{filippov1988differential}
Filippov, A.F.: Differential equations with discontinuous righthand sides,
  \emph{Mathematics and its Applications (Soviet Series)}, vol.~18.
\newblock Kluwer Academic Publishers Group, Dordrecht (1988).
\newblock Translated from the Russian

\bibitem{luke2023transformerl2o}
G{\"{a}}rtner, E., Metz, L., Andriluka, M., Freeman, C.D., Sminchisescu, C.:
  Transformer-based learned optimization.
\newblock In: {IEEE/CVF} Conference on Computer Vision and Pattern Recognition,
  {CVPR} 2023, Vancouver, BC, Canada, June 17-24, 2023, pp. 11970--11979.
  {IEEE} (2023).
\newblock \doi{10.1109/CVPR52729.2023.01152}.
\newblock \urlprefix\url{https://doi.org/10.1109/CVPR52729.2023.01152}

\bibitem{golub2013matrix}
Golub, G.H., Van~Loan, C.F.: Matrix computations, fourth edn.
\newblock Johns Hopkins Studies in the Mathematical Sciences. Johns Hopkins
  University Press, Baltimore, MD (2013)

\bibitem{lecun2010learning}
Gregor, K., LeCun, Y.: Learning fast approximations of sparse coding.
\newblock In: J.~F{\"{u}}rnkranz, T.~Joachims (eds.) Proceedings of the 27th
  International Conference on Machine Learning (ICML-10), June 21-24, 2010,
  Haifa, Israel, pp. 399--406. Omnipress (2010).
\newblock \urlprefix\url{https://icml.cc/Conferences/2010/papers/449.pdf}

\bibitem{griewank2008evaluating}
Griewank, A., Walther, A.: Evaluating derivatives: Principles and techniques of
  algorithmic differentiation, second edn.
\newblock Society for Industrial and Applied Mathematics (SIAM), Philadelphia,
  PA (2008)

\bibitem{kamilov2023plugandplay}
Kamilov, U.S., Bouman, C.A., Buzzard, G.T., Wohlberg, B.: Plug-and-play methods
  for integrating physical and learned models in computational imaging: Theory,
  algorithms, and applications.
\newblock IEEE Signal Processing Magazine \textbf{40}(1), 85--97 (2023).
\newblock \doi{10.1109/MSP.2022.3199595}

\bibitem{lambert1991numerical}
Lambert, J.D.: Numerical methods for ordinary differential systems: The initial
  value problem.
\newblock John Wiley \& Sons, Ltd., Chichester (1991)

\bibitem{li2024linear}
Li, B., Shi, B., xiang Yuan, Y.: Linear convergence of forward-backward
  accelerated algorithms without knowledge of the modulus of strong convexity
  (2024)

\bibitem{li2017l2o}
Li, K., Malik, J.: Learning to optimize.
\newblock In: 5th International Conference on Learning Representations, {ICLR}
  2017, Toulon, France, April 24-26, 2017, Conference Track Proceedings.
  OpenReview.net (2017).
\newblock \urlprefix\url{https://openreview.net/forum?id=ry4Vrt5gl}

\bibitem{linControltheoreticPerspectiveOptimal2021}
Lin, T., Jordan, M.I.: A control-theoretic perspective on optimal high-order
  optimization.
\newblock Math. Program. \textbf{195}(1-2), 929--975 (2022)

\bibitem{liu2019alista}
Liu, J., Chen, X., Wang, Z., Yin, W.: {ALISTA:} analytic weights are as good as
  learned weights in {LISTA}.
\newblock In: 7th International Conference on Learning Representations, {ICLR}
  2019, New Orleans, LA, USA, May 6-9, 2019. OpenReview.net (2019).
\newblock \urlprefix\url{https://openreview.net/forum?id=B1lnzn0ctQ}

\bibitem{liu2023mathstructure}
Liu, J., Chen, X., Wang, Z., Yin, W., Cai, H.: Towards constituting
  mathematical structures for learning to optimize.
\newblock In: A.~Krause, E.~Brunskill, K.~Cho, B.~Engelhardt, S.~Sabato,
  J.~Scarlett (eds.) International Conference on Machine Learning, {ICML} 2023,
  23-29 July 2023, Honolulu, Hawaii, {USA}, \emph{Proceedings of Machine
  Learning Research}, vol. 202, pp. 21426--21449. {PMLR} (2023).
\newblock \urlprefix\url{https://proceedings.mlr.press/v202/liu23e.html}

\bibitem{luoDifferentialEquationSolvers2021}
Luo, H., Chen, L.: From differential equation solvers to accelerated
  first-order methods for convex optimization.
\newblock Math. Program. \textbf{195}(1-2), 735--781 (2022)

\bibitem{marx2022path}
Marx, S., Pauwels, E.: Path differentiability of {ODE} flows.
\newblock J. Differential Equations \textbf{338}, 321--351 (2022)

\bibitem{monga2021algorithm}
Monga, V., Li, Y., Eldar, Y.C.: Algorithm unrolling: Interpretable, efficient
  deep learning for signal and image processing.
\newblock IEEE Signal Processing Magazine \textbf{38}(2), 18--44 (2021).
\newblock \doi{10.1109/MSP.2020.3016905}

\bibitem{taylor2023systematic}
Moucer, C., Taylor, A., Bach, F.: A systematic approach to {L}yapunov analyses
  of continuous-time models in convex optimization.
\newblock SIAM J. Optim. \textbf{33}(3), 1558--1586 (2023)

\bibitem{nemirovskiProblemComplexityMethod1983}
Nemirovsky, A.S., Yudin, D.B.: Problem complexity and method efficiency in
  optimization.
\newblock Wiley-Interscience Series in Discrete Mathematics. John Wiley \&
  Sons, Inc., New York (1983)

\bibitem{nesterovLecturesConvexOptimization2018}
Nesterov, Y.: Lectures on convex optimization, \emph{Springer Optimization and
  Its Applications}, vol. 137, second edn.
\newblock Springer, Cham (2018)

\bibitem{nesterov1983method}
Nesterov, Y.E.: A method for solving the convex programming problem with
  convergence rate {$O(1/k\sp{2})$}.
\newblock Dokl. Akad. Nauk SSSR \textbf{269}(3), 543--547 (1983)

\bibitem{platt1998fast}
Platt, J.: Fast training of support vector machines using sequential minimal
  optimization.
\newblock In: Advances in Kernel Methods - Support Vector Learning. MIT Press
  (1998)

\bibitem{rockafellar1998variational}
Rockafellar, R.T., Wets, R.J.B.: Variational analysis, \emph{Grundlehren der
  mathematischen Wissenschaften [Fundamental Principles of Mathematical
  Sciences]}, vol. 317.
\newblock Springer-Verlag, Berlin (1998).
\newblock \doi{10.1007/978-3-642-02431-3}.
\newblock \urlprefix\url{https://doi.org/10.1007/978-3-642-02431-3}

\bibitem{royden1968real}
Royden, H.L., Fitzpatrick, P.: Real analysis, vol.~2.
\newblock Macmillan New York (1968)

\bibitem{shiUnderstandingAccelerationPhenomenon2021}
Shi, B., Du, S.S., Jordan, M.I., Su, W.J.: Understanding the acceleration
  phenomenon via high-resolution differential equations.
\newblock Math. Program. \textbf{195}(1-2), 79--148 (2022)

\bibitem{stewart1990matrix}
Stewart, G.W., Sun, J.G.: Matrix perturbation theory.
\newblock Computer Science and Scientific Computing. Academic Press, Inc.,
  Boston, MA (1990)

\bibitem{suDifferentialEquationModeling}
Su, W., Boyd, S., Cand\`es, E.J.: A differential equation for modeling
  {N}esterov's accelerated gradient method: theory and insights.
\newblock J. Mach. Learn. Res. \textbf{17}, Paper No. 153, 43 (2016)

\bibitem{cortes2019ratematching}
Vaquero, M., Cortes, J.: Convergence-rate-matching discretization of
  accelerated optimization flows through opportunistic state-triggered control.
\newblock In: H.~Wallach, H.~Larochelle, A.~Beygelzimer, F.~d\textquotesingle
  Alch\'{e}-Buc, E.~Fox, R.~Garnett (eds.) Advances in Neural Information
  Processing Systems, vol.~32. Curran Associates, Inc. (2019)

\bibitem{vaquero2023resource}
Vaquero, M., Mestres, P., Cortés, J.: Resource-aware discretization of
  accelerated optimization flows: The heavy-ball dynamics case.
\newblock IEEE Transactions on Automatic Control \textbf{68}(4), 2109--2124
  (2023).
\newblock \doi{10.1109/TAC.2022.3171307}

\bibitem{whitney1935critical}
Whitney, H.: {A function not constant on a connected set of critical points}.
\newblock Duke Mathematical Journal \textbf{1}(4), 514 -- 517 (1935).
\newblock \doi{10.1215/S0012-7094-35-00138-7}.
\newblock \urlprefix\url{https://doi.org/10.1215/S0012-7094-35-00138-7}

\bibitem{wibisonoVariationalPerspectiveAccelerated2016}
Wibisono, A., Wilson, A.C., Jordan, M.I.: A variational perspective on
  accelerated methods in optimization.
\newblock Proc. Natl. Acad. Sci. USA \textbf{113}(47), E7351--E7358 (2016)

\bibitem{wilkie1996model}
Wilkie, A.J.: Model completeness results for expansions of the ordered field of
  real numbers by restricted {P}faffian functions and the exponential function.
\newblock J. Amer. Math. Soc. \textbf{9}(4), 1051--1094 (1996)

\bibitem{wilson2019accelerating}
Wilson, A.C., Mackey, L., Wibisono, A.: Accelerating rescaled gradient descent:
  Fast optimization of smooth functions.
\newblock In: Advances in Neural Information Processing Systems, vol.~32.
  Curran Associates, Inc. (2019)

\bibitem{yang2020admmcsnet}
Yang, Y., Sun, J., Li, H., Xu, Z.: Admm-csnet: A deep learning approach for
  image compressive sensing.
\newblock IEEE Transactions on Pattern Analysis and Machine Intelligence
  \textbf{42}(3), 521--538 (2020).
\newblock \doi{10.1109/TPAMI.2018.2883941}

\bibitem{zhang2018istanet}
Zhang, J., Ghanem, B.: Ista-net: Interpretable optimization-inspired deep
  network for image compressive sensing.
\newblock In: 2018 {IEEE} Conference on Computer Vision and Pattern
  Recognition, {CVPR} 2018, Salt Lake City, UT, USA, June 18-22, 2018, pp.
  1828--1837. Computer Vision Foundation / {IEEE} Computer Society (2018).
\newblock \doi{10.1109/CVPR.2018.00196}.
\newblock
  \urlprefix\url{http://openaccess.thecvf.com/content\_cvpr\_2018/html/Zhang\_ISTA-Net\_Interpretable\_Optimization-Inspired\_CVPR\_2018\_paper.html}

\bibitem{zhang2020l0l1smooth}
Zhang, J., He, T., Sra, S., Jadbabaie, A.: Why gradient clipping accelerates
  training: {A} theoretical justification for adaptivity.
\newblock In: 8th International Conference on Learning Representations, {ICLR}
  2020, Addis Ababa, Ethiopia, April 26-30, 2020. OpenReview.net (2020).
\newblock \urlprefix\url{https://openreview.net/forum?id=BJgnXpVYwS}

\bibitem{zhangDirectRungeKuttaDiscretization2018}
Zhang, J., Mokhtari, A., Sra, S., Jadbabaie, A.: Direct {R}unge-{K}utta
  discretization achieves acceleration.
\newblock In: Advances in Neural Information Processing Systems, vol.~31.
  Curran Associates, Inc. (2018)

\bibitem{zhang2019discretization}
Zhang, J., Sra, S., Jadbabaie, A.: Acceleration in first order quasi-strongly
  convex optimization by ode discretization.
\newblock In: 2019 IEEE 58th Conference on Decision and Control (CDC), pp.
  1501--1506 (2019).
\newblock \doi{10.1109/CDC40024.2019.9030046}

\bibitem{zhang2021revisiting}
Zhang, P., Orvieto, A., Daneshmand, H., Hofmann, T., Smith, R.S.: Revisiting
  the role of euler numerical integration on acceleration and stability in
  convex optimization.
\newblock In: A.~Banerjee, K.~Fukumizu (eds.) The 24th International Conference
  on Artificial Intelligence and Statistics, {AISTATS} 2021, April 13-15, 2021,
  Virtual Event, \emph{Proceedings of Machine Learning Research}, vol. 130, pp.
  3979--3987. {PMLR} (2021)

\bibitem{zheng2015improving}
Zheng, H., Yang, Z., Liu, W., Liang, J., Li, Y.: Improving deep neural networks
  using softplus units.
\newblock In: 2015 International {J}oint {C}onference on {N}eural {N}etworks
  (IJCNN), pp. 1--4. IEEE (2015)

\end{thebibliography}

\end{document}